\documentclass[12pt]{amsart}

\usepackage{setspace}

 \usepackage{graphicx}
 \usepackage[all]{xy}
 \usepackage{mathbbol}
 \usepackage{amscd,amsfonts,amssymb,amsmath,amsthm,latexsym}
 \usepackage{tikz-cd}
 \usepackage{color}
 \usepackage{hyperref}

 \usepackage{epigraph} % For including quotations 

\makeatletter
\newcommand{\doublewidetilde}[1]{{%
  \mathpalette\double@widetilde{#1}%
}}
\newcommand{\double@widetilde}[2]{%
  \sbox\z@{$\m@th#1\widetilde{#2}$}%
  \ht\z@=.9\ht\z@
  \widetilde{\box\z@}%
}
\makeatother

\newcommand*{\Scale}[2][4]{\scalebox{#1}{$#2$}}%
\subjclass{13F60, 16G20, 16W60}

\begin{document}

%%%% DANIEL'S COMMANDS

\newcommand{\suchthat}{\mid}
\newcommand{\indexset}{\mathtt{I}}
\newcommand{\bfv}{\mathbf{v}}
\newcommand{\bbZ}{\mathbb{Z}}
\newcommand{\field}{\mathbb{k}}
\newcommand{\pathalg}[1]{\field\langle #1\rangle}
\newcommand{\compalg}[1]{\field\langle\hspace{-0.075cm}\langle #1\rangle\hspace{-0.075cm}\rangle}
\newcommand{\idealM}{\mathfrak{m}}
\newcommand{\jacobalg}[1]{\mathcal{P}(#1)}
\newcommand{\coker}{\operatorname{coker}}
\newcommand{\image}{\operatorname{im}}
\newcommand{\End}{\operatorname{End}}
\newcommand{\Hom}{\operatorname{Hom}}
\newcommand{\Gr}{\operatorname{Gr}}
\newcommand{\Rep}{\mathbf{Rep}}
\newcommand{\Mod}{\mathbf{Mod}}
\newcommand{\DecMod}{\mathbf{DecMod}}
\newcommand{\soc}{\operatorname{soc}}
\newcommand{\rad}{\operatorname{rad}}
\newcommand{\myid}{1\hspace{-0.075cm}1}
\newcommand{\surf}{(\Sigma,\mathbb{M})}

\newcommand{\diag}{\operatorname{diag}}

\newcommand\Daniel[1]{\color{red} Daniel: #1}

\numberwithin{equation}{section}

\theoremstyle{plain}
    \newtheorem{theorem}{Theorem}[section]
    \newtheorem{defi}[theorem]{Definition}
    \newtheorem{lemma}[theorem]{Lemma}
    \newtheorem{prop}[theorem]{Proposition}
    \newtheorem{coro}[theorem]{Corollary}
    \newtheorem{question}[theorem]{Question}
    \newtheorem{conj}[theorem]{Conjecture}

\theoremstyle{definition} 
    \newtheorem{remark}[theorem]{Remark}
    \newtheorem{example}[theorem]{Example}

\theoremstyle{remark}
\newtheorem{case}{Case}
\newtheorem{subcase}{Subcase}[case]

\title[DWZ-mutations of infinite-dimensional modules I]{Derksen-Weyman-Zelevinsky mutations\\ of infinite-dimensional modules I: Foundations}
\author{Daniel Labardini Fragoso}
\address{Daniel Labardini Fragoso\newline
Dipartimento di Matematica “Tullio Levi-Civita”,\newline Universit\`a degli Studi di Padova, Italy
}
\email{daniel.labardinifragoso@unipd.it, labardini@math.unipd.it}
\date{\today}
\dedicatory{Dedicated to Jerzy Weyman on the occasion of his $70^{\operatorname{th}}$ birthday\\
and to the memory of Andrei Zelevinsky}

\maketitle

\begin{abstract}
Derksen-Weyman-Zelevinsky's mutation theory of finite-dimensional representations of quivers with potential is generalized to the framework of infinite-dimensional modules.
\end{abstract}

{
\tableofcontents
}

\newpage

\setlength\epigraphwidth{.5\textwidth}
\setlength\epigraphrule{0pt}
\epigraph{It was twenty years ago today,\\
Sergeant Pepper taught the band to play.\\
They’ve been going in and out of style,\\
but they’re guaranteed to raise a smile.\\
So may I introduce to you\\ 
the act you’ve known for all these years?%\\
%Sergeant Pepper’s Lonely Hearts Club Band.
}{\textit{John Lennon \& Paul McCartney}}

\section{Introduction}

It was almost twenty years ago that Harm Derksen, Jerzy Weyman and Andrei Zelevinsky developed a monumental mutation theory of quivers with potential and their finite-dimensional representations \cite{derksen2008quivers}, thus creating a one-of-a-kind quiver-representation-theoretic categorification of skew-symmetric cluster algebras \cite{derksen2010quivers} that they then applied to solve several conjectures from Sergey Fomin and Andrei Zelevinsky's milestone \textsc{Cluster Algebras IV: Coefficients} \cite{fomin2007cluster}. 

Derksen-Weyman-Zelevinsky's QP-mutation theory itself quickly became the cornerstone of a multitude of other developments and theories. For example, their clever homological interpretation of the $E$-invariant was exploited in \cite{cerulli2012quivers,cerulli2013linear} to prove Fomin-Zelevinsky's conjectured linear independence of cluster monomials in the case of cluster algebras coming from quivers. That same homological interpretation and the $E$-rigidity of the decorated representations corresponding to cluster monomials (proved in \cite{derksen2010quivers}) was then one of the main inspirations behind one of the most celebrated breakthroughs of the last fifteen years in the representation theory of algebras, namely, the $\tau$-tilting theory of Adachi-Iyama-Reiten \cite{adachi2014tau}. 

Other highly successful headways have been Keller-Yang's equivalences \cite{keller2011derived} between derived categories of Ginzburg dg-algebras of quivers with potential related by Derksen-Weyman-Zelevinsky mutations, as well as the generalized cluster categories defined by Amiot
\cite{amiot2009cluster}.  
A combination of these Authors' constructions and results with those from \cite{labardini2009quivers,labardini2016quivers} allowed the association of well-defined $2$-CY and $3$-CY triangulated categories to a given surface with marked points (see \cite[\S3.4]{amiot2011ongeneralized} for the $2$-CY case and \cite[\S9]{bridgeland2015quadratic} for the $3$-CY case) that have been heavily exploited by many Authors.

But despite the proven mathematical success of Derksen-Weyman-Zelevinsky's mutation theory of finite-dimensional representations, the importance of infinite-dimensional modules in the representation theory of algebras --see e.g. \cite{crawley1998infinite,crawley1991tame,crawley1992modules}, and the fact that Adachi-Iyama-Reiten's mutations of $\tau$-tilting pairs has indeed been further pursued in infinite-dimensional settings --see e.g. \cite{aihara2012silting,angeleri2025mutation,hrbek2025mutation,kimura2014tilting,vitoria2023mutations} and Remark \ref{rem:AIR-perspective-vs-DWZ-perspective} below, thus far the natural problem of generalizing DWZ's theory so as to encompass also infinite-dimensional representations or modules has not been addressed. In this paper we present such a generalization and open a series of articles where connections with both classical and contemporary developments surrounding the representation theory of algebras will be established.

\begin{remark}\label{rem:AIR-perspective-vs-DWZ-perspective} AIR- and DWZ-mutations are transversal but compatible. To be more pecise,
in AIR's perspective the underlying associative algebra is kept fixed and $\tau$-tilting pairs are mutated within a fixed category, whereas DWZ's perspective follows that of Bernstein-Gelfand-Ponomarev's reflections, where the underlying algebra \emph{mutates} and to each module over the `old' algebra one associates a module over the `new', mutated algebra\footnote{Thus, perhaps \emph{reflection} would be better suited than \emph{mutation} to refer to DWZ mutations.}. By \cite[Theorems 7.1 and 10.5, and Corollary 10.9]{derksen2010quivers}, when $(Q,S)$ has finite-dimensional Jacobian algebra, $\tau$-tilting pairs over $\jacobalg{Q,S}$ related by an Adachi-Iyama-Reiten mutation are taken by any Derksen-Weyman-Zelevinsky mutation to $\tau$-tilting pairs over $\jacobalg{\mu_k(Q,S)}$ related by an Adachi-Iyama-Reiten mutation. Cluster categories and the derived categories of Ginzburg dg algebras unify both perspectives, and DWZ mutations become a sort of `change of coordinates' of the objects, where each `coordinatization' is provided by picking a \emph{cluster-tilting object} $T$ and `projecting' each object $X$ to the module category of the opposite endomorphism algebra of $T$, the `coordinates' of $X$ with respect to $\End(T)^{\operatorname{op}}$ being the $\End(T)^{\operatorname{op}}$-module $\Hom(T,X)$. The endomorphism algebra $\End(T)^{\operatorname{op}}$ is isomorphic to the Jacobian algebra of a quiver with potential.
\end{remark}

Let $(Q,S)$ be a quiver with potential and $\jacobalg{Q,S}$ be its Jacobian algebra. The main construction of the present paper is the \emph{Derksen-Weyman-Zelevinsky mutation} of an arbitrary $\jacobalg{Q,S}$-module. The preeminent guideline behind our definition could not be more simple-minded: for finite-dimensional modules, Derksen-Weyman-Zelevinsky show that the action of a $k$-hook $a^*b^*$ on the mutated module is always equal to the action of $-\partial_{ba}(S)$ on the original module. Thus, given an arbitrary $\jacobalg{Q,S}$-module $\overline{M}$ we form DWZ's $\alpha$-$\beta$-$\gamma$-triangle \eqref{eq:DWZ-alphabetagamma-triangle} and define a quiver representation $\overline{M}$ by replacing $M_k$ with a space $\overline{M}_k$, and the maps $M_{\operatorname{in}}\overset{\alpha}{\rightarrow} M_k\overset{\beta}{\rightarrow} M_{\operatorname{out}}$ with linear maps $M_{\operatorname{in}}\overset{\overline{\beta}}{\leftarrow} \overline{M}_k\overset{\overline{\alpha}}{\leftarrow} M_{\operatorname{out}}$, attaching to the composite arrows $[ba]$ the composite linear maps $M_bM_a$. The linear maps attached to each of the arrows $a^*$ and $b^*$ are by definition the corresponding components of the maps $\overline{\beta}$ and $\overline{\alpha}$. All of this is done practically in the exact same way as \cite[\S10]{derksen2008quivers} (except that we provide four equivalent descriptions of the quiver representation $\overline{M}$), so the composition $\overline{M}_{a^*}\overline{M}_{b^*}$ will still turn out to equal the linear map $-M_{\partial_{ba}(S)}$, exactly as in the finite-dimensional case; see Lemma \ref{lemma:overlinebeta-overlinealpha=-gamma}.

Now, inherent to the limit process ideated by Derksen-Weyman-Zelevinsky's to delete $2$-cycles algebraically (see Remark~\ref{rem:necessity-of-complete-path-algs}) is the necessity to work over complete path algebras instead of just path algebras, see Remark~\ref{rem:necessity-of-complete-path-algs}. However, not every module over the path algebra is a module over the complete path algebra, see \S\ref{sec:quiver-reps-vs-modules-over-comp-path-algs}, the problem being that the action of an infinite linear combination of paths cannot be defined simply as the infinite linear combination of the linear maps given by the actions of the paths, as the latter infinite linear combinations are not really defined. This forces us to prove that the action of the path algebra $\pathalg{\widetilde{\mu}_k(Q)}$ on the quiver representation $\overline{M}$ can be extended to an action of the complete path algebra $\compalg{\widetilde{\mu}_k(Q)}$ making $\overline{M}$ a $\compalg{\widetilde{\mu}_k(Q)}$-module. We do so in Definition \ref{def:action-of-arb-u-on-overlineM} and in our fist main result, Proposition~\ref{prop:overlineM-is-left-compalg-module}. The former defines the action of an arbitrary infinite linear combination of paths $u\in\compalg{\widetilde{\mu}_k(Q)}$ on $\overline{M}$ and is founded on three fundamental facts
from \cite{derksen2008quivers}:
\begin{itemize}
\item that every $\field^{Q_0}$-$\field^{Q_0}$-bimodule homomorphism $\field^{Q_1}\rightarrow \idealM\subseteq\compalg{Q}$ extends uniquely to a $\field^{Q_0}$-ring homomorphism;
\item that replacing each $k$-hook $ab$ with the composite arrow $[ab]$ constitutes a $\field$-algebra isomorphism $e_{\widehat{k}}\compalg{Q}e_{\widehat{k}}\rightarrow\compalg{e_{\widehat{k}}\widetilde{\mu}_k(Q)e_{\widehat{k}}}$;
\item that this induces an isomorphism $e_{\widehat{k}}\jacobalg{Q,S}e_{\widehat{k}}\rightarrow e_{\widehat{k}}\jacobalg{\widetilde{\mu}_k(Q),\widetilde{\mu}_k(S)}e_{\widehat{k}}$.
\end{itemize}
Proposition~\ref{prop:overlineM-is-left-compalg-module} exploits the identity $\overline{M}_{a^*}\overline{M}_{b^*}=-M_{\partial_{ba}(S)}$ to show that what we obtain is indeed a $\compalg{\widetilde{\mu}_k(Q)}$-module structure on $\overline{M}$.

Proving that the cyclic derivatives of $\widetilde{\mu}_k(S)$ annihilate $\overline{M}$ is then done in the exact same way as in the finite-dimensional case, cf. \cite[Proposition 10.7]{derksen2008quivers}. However, proving that the whole\footnote{Recall that the Jacobian ideal is not defined as the $2$-sided ideal generated by the cyclic derivatives, but as the $\idealM$-adic topological closure of such ideal.} Jacobian ideal $J(\widetilde{\mu}_k(S))$ annihilates $\overline{M}$ is way harder in the general case than in the case of finite-dimensional modules, mainly because if $M$ is finite-dimensional, then $\overline{M}$ is locally nilpotent, so the proof reduces to knowing that $\overline{M}$ is annihilated by the cyclic derivatives of $\widetilde{\mu}_k(S)$. But such argument using local nilpotency is not available for arbitrary modules. So we proceed by first analyzing the $\alpha$-$\beta$-$\gamma$-triangles of the indecomposable projectives and of their radicals (Proposition \ref{prop:spanning ker a/ker gamma for radical} and Corollary \ref{coro:most-basic-properties-of-alpha-beta-gamma-maps-from-projective}), which have the nice feature of being defined in terms of the formal algebraic combinatorics that governs paths and their infinite linear combinations. We then pick an arbitrary element $u$ of the $\idealM$-adic topological closure of the $2$-sided ideal generated by the cyclic derivatives, and specialize this analysis to prove in Proposition \ref{prop:Jacobian-ideal-J(widetildemuk(S))-annihilates-overlineM}, the second main result of this paper, that $u$ annihilates $\overline{M}$.

Knowing that $\overline{M}$ is indeed a $\jacobalg{\widetilde{\mu}_k(Q,S)}$-module, we take any reduction (see \eqref{eq:alg-isomorphisms-induced-by-splitting-thm}, \eqref{eq:functors-induced-by-splitting-data} and Definitions \ref{def:reduced-and-trivial-parts} and \ref{def:premut-and-mut-of-a-QP}) $\varphi\circ\iota :\mu_k(Q,S)
\rightarrow \widetilde{\mu}_k(Q,S)
$, and define \emph{the Derksen-Weyman-Zelevinsky mutation of $M$ with respect to $\varphi$} to be $\overline{M}$ with the $\jacobalg{\mu_k(Q,S)}$-module structure induced by pulling back along $\overline{\varphi}\circ\overline{\iota}$ the $\jacobalg{\widetilde{\mu}_k(Q,S)}$-module structure from the previous paragraph, see Definition \ref{def:mut-of-dec-rep}.

In \S\ref{sec:examples} we give some examples and indicate some of the fascinating general phenomena they will be shown to be instances of in the forthcoming sequels to this paper. In \S\ref{sec:muts-commute-with-duality} we show that DWZ-mutations commute with vector space duality. In~\S\ref{sec:iso-preservation-and-involutivity} we prove that DWZ-mutations are involutive and send isomorphic modules to isomorphic modules, and non-isomorphic modules to non-isomorphic modules.

Throughout the paper, $Q=(Q_0,Q_1,h,t)$ will be a finite quiver, $\field$ will be an arbitrary field, and $\field^{Q_0}:=\times_{j\in Q_0}\field$. Modules will be left modules. We will also adopt the following notational conventions without further warning or apology: the letter $\iota$ will always denote inclusion, and for any linear map $f:V\rightarrow W$ and subspaces $U\subseteq V,X\subseteq W$ such that $f(U)\subseteq X$, the map $V/U\rightarrow W/X$ induced by $f$ will always be denoted $\overline{f}$..

\section{Quiver mutations}\label{sec:quiver-mutations}

Recall that a \emph{quiver} is a finite directed graph, i.e., a quadruple $Q=(Q_0,Q_1,h,t)$, where $Q_0$ is a set of
\emph{vertices}, $Q_1$ is a set of \emph{arrows}, and $h:Q_1\rightarrow Q_0$ and $t:Q_1\rightarrow Q_0$ are the \emph{head}
and \emph{tail} functions. We write $a:i\rightarrow j$ to indicate that $a$ is an arrow of $Q$ with $t(a)=i$, $h(a)=j$. With the only exception of Example \ref{ex:functor-F-is-not-dense}, all of our quivers will be \emph{loop-free}, i.e., without arrows $a$ such that $t(a)=h(a)$.

A \emph{path of length} $d>0$ in $Q$ is a sequence $a_1a_2\ldots a_d$ of arrows with $t(a_j)=h(a_{j+1})$ for $j=1,\ldots,d-1$. A path
$a_1a_2\ldots a_d$ of length $d>0$ is a $d$\emph{-cycle} if $h(a_1)=t(a_d)$. A quiver is \emph{2-acyclic} if it has no 2-cycles.

Paths are composed as functions: if $a=a_1\cdots a_d$ and $b=b_1\cdots b_{d'}$ are paths with $h(b)=t(a)$, then the concatenation $ab$ is defined as the path $a_1\cdots a_db_1\cdots b_{d'}$, going from $t(b_{d'})$ to $h(a_1)$. For $k\in Q_0$, a $k$\emph{-hook} in $Q$ is a path $ab$ of length 2 in which $a$ and $b$ are arrows such that $t(a)=k=h(b)$.

The following definition is a well-known reformulation of Fomin-Zelevinsky's matrix mutation \cite[Definition~4.2]{fomin2002cluster} in the case of skew-symmetric matrices.

\begin{defi}\label{def:threesteps} Let $Q$ be a quiver and $k\in Q_0$ a vertex not incident to any $2$-cycle of $Q$. The
\emph{mutation} of $Q$ with respect to $k$ is the quiver $\mu_k(Q)$ with vertex set $Q_0$ obtained after applying the following three-step
procedure to $Q$:
\begin{itemize}
\item[(Step 1)] For each $k$-hook $ab$ introduce an arrow $[ab]:t(b)\rightarrow h(a)$;
\item[(Step 2)] replace each incident to $k$ with an arrow $c^*$ in the opposite direction;
\item[(Step 3)] choose a maximal collection of disjoint 2-cycles and remove them.
\end{itemize}
The result of applying only the first two steps will be called the \emph{premutation} $\widetilde{\mu}_k(Q)$.
\end{defi}

\begin{theorem}\cite{fomin2002cluster,derksen2008quivers}
    If $Q$ is $2$-acyclic, then for every vertex $k\in Q_0$ we have $\mu_k(\mu_k(Q))\cong Q$ through a quiver isomorphism that pointwisely fixes the common vertex set $Q_0$. Furthermore, up to quiver isomorphism, $\mu_k(\mu_k(Q))$ can be obtained from $\widetilde{\mu}_k(\widetilde{\mu}_k(Q))$ by deleting from the latter a maximal set of disjoint $2$-cycles.
\end{theorem}

\begin{remark}
    If $Q$ is $2$-acyclic, then $\{[ba],[a^*b^*]\suchthat a,b\in Q_1,h(a)=t(b)\}$ is a maximal set of disjoint $2$-cycles of $\widetilde{\mu}_k(\widetilde{\mu}_k(Q))$.
\end{remark}

\section{Derksen-Weyman-Zelevinsky's amalgamation idea}\label{sec:DWZ-amalgamation-idea}

Suppose $k\in Q_0$ is a \emph{sink} (resp. \emph{source}) of $Q$, i.e., $k$ is not tail (resp. head) of any arrow of $Q$. Let $a_1,\ldots,a_s$ (resp. $b_1,\ldots,b_r$) be the arrows having $k$ as head (resp. tail). For any representation $M$ of $Q$, Bernstein-Gelfand-Ponomarev \cite{bernstein1973coxeter} consider the map
\[
\Scale[0.9]{
\xymatrix{
& M_k \\
M_{\operatorname{in}}(k):=\bigoplus_{i=1}^sM_{t(a_i)} \ar[ur]^{\alpha:=\left[\begin{array}{ccc}(a_1)_M & \cdots & (a_s)_M\end{array}\right]\qquad} & 
}
\qquad \text{(resp.}
\ \xymatrix{
 M_k \ar[dr]^{\beta:=\left[\begin{array}{c}(b_1)_M\\ \vdots \\ (b_r)_M\end{array}\right]}& \\
  &\bigoplus_{j=1}^rM_{h(b_j)}=:M_{\operatorname{out}(k)}
}
\text{)}
}
\]
and define a representation $\overline{M}$ of $\mu_k(Q)$, called the \emph{sink-to-source reflection} (resp. \emph{source-to-sink reflection}) of $M$, by setting $\overline{M}_j:=M_j$ for $j\in Q_0\setminus\{k\}$, $\overline{M}_c:=M_c$ for every arrow $c\in Q_1$ not incident to $k$, and
\[
\Scale[0.9]{
\xymatrix{
& \overline{M}_k:=\ker\alpha \ar[dl]_{ \left[\begin{array}{c}(a_1^*)_{\overline{M}} \\ \vdots \\ (a_s^*)_{\overline{M}}\end{array}\right]:=\operatorname{incl}\qquad } \\
M_{\operatorname{in}}(k)%:=\bigoplus_{i=1}^sM_{t(a_i)}  
& 
}
\qquad \text{(resp.}
\ 
\xymatrix{
 \operatorname{coker}\beta=:\overline{M}_k & \\
 & %\bigoplus_{j=1}^rM_{h(b_j)}=:
 M_{\operatorname{out}(k)} \ar[ul]_{\qquad \operatorname{proj}=:\left[\begin{array}{ccc}(b_1^*)_{\overline{M}} & \cdots & (b_r^*)_{\overline{M}}\end{array}\right]}
}
\text{)}
}
\]

But suppose that $k\in Q_0$ is an arbitrary vertex, i.e., not necessarily a sink or a source of~$Q$.
Let $a_1,\ldots,a_s$ be the arrows having $k$ as head, and $b_1,\ldots,b_r$ the arrows having $k$ as tail. Following Bernstein-Gelfand-Ponomarev, for any representation $M$ of $Q$ we consider the maps
\[
\xymatrix{
& M_k \ar[dr]^{\beta:=\left[\begin{array}{c}(b_1)_M\\ \vdots \\ (b_r)_M\end{array}\right]}& \\
M_{\operatorname{in}}:=\bigoplus_{i=1}^sM_{t(a_i)} \ar[ur]^{\alpha:=\left[\begin{array}{ccc}(a_1)_M & \cdots & (a_s)_M\end{array}\right]\qquad} & &\bigoplus_{j=1}^rM_{h(b_j)}=:M_{\operatorname{out}}
}
\]
and define a quiver representation of $\widetilde{\mu}_k(Q)$ through the assembled linear maps
\begin{equation}\label{eq:BGP-DWZ-mutation-tentative}
\Scale[0.9]{
\xymatrix{
& \overline{M}_k:=\operatorname{coker}\beta\oplus\ker\alpha \ar[dl]_{\text{{\tiny $\left[\begin{array}{c}(a_1^*)_{\overline{M}} \\ \vdots \\ (a_s^*)_{\overline{M}}\end{array}\right]:=\left[\begin{array}{cc}0 & \operatorname{incl}\end{array}\right] $}}\qquad\qquad\qquad }& \\
\bigoplus_{i=1}^sM_{t(a_i)}  
\ar[rr]_{\beta\alpha}
& & \bigoplus_{j=1}^rM_{h(b_j)} \ar[ul]_{\qquad\qquad\qquad \text{{\tiny $\left[\begin{array}{c}\operatorname{proj}\\ 0\end{array}\right]=:\left[\begin{array}{ccc}(b_1^*)_{\overline{M}} & \cdots & (b_r^*)_{\overline{M}}\end{array}\right]$}}}.
}}
\end{equation}
With Fomin-Zelevinsky's theory of $F$-polynomials and $g$-vectors of cluster variables \cite{fomin2007cluster} as main guideline, the Derksen-Weyman-Zelevinsky approach searches a common vector subspace of $\coker\beta$ and $\ker\alpha$ along which to amalgamate them, and refines the diagram \eqref{eq:BGP-DWZ-mutation-tentative} once such subspace has been found and the amalgamation has taken place.

Derksen-Weyman-Zelevinsky discovered in \cite{derksen2008quivers} that if one fixes a potential $S$ on $Q$ and assumes that $\overline{M}$ is annihilated by the cyclic derivatives of $S$, then there is a good candidate to common vector subspace of $\operatorname{coker}\beta$ and $\ker\alpha$ along which to amalgamate them, see \S\ref{subsec:the-four-defs-of-ovelineM}. Furthermore, if the potential $S$ is non-degenerate, then this candidate is the correct choice for cluster algebra purposes \cite[Theorem~5.1]{derksen2010quivers}. According to \cite[Corollary~7.4]{derksen2008quivers}, if the ground field is uncountable, then every $2$-acyclic quiver admits a non-degenerate potential. Roughly speaking, a non-degenerate potential always allows to perform the deletion of $2$-cycles algebraically not only after one mutation, but after any finite sequence of mutations. Algebraic deletion amounts to applying certain algebra automorphisms of complete path algebras to place arrows in $2$-cycles inside Jacobian ideals, thus effectively deleting them in the corresponding algebra quotients, see \S\ref{subsec:red-part-of-quiver-and-module}.

\section{Quiver representations vs. modules over complete path algebras}\label{sec:quiver-reps-vs-modules-over-comp-path-algs}

The canonical inclusion $\pathalg{Q}\hookrightarrow\compalg{Q}$ induces a forgetful functor $F:\compalg{Q}$-$\Mod\rightarrow \pathalg{Q}$-$\Mod$, easily seen to be faithful. The essential image of $F$ consists of those $\pathalg{Q}$-modules $M$ for which the action of $\pathalg{Q}$ on $M$ can be extended to an action of $\compalg{Q}$ on $M$ under which $M$ is a $\compalg{Q}$-module. For example, every locally nilpotent $\pathalg{Q}$-module lies in the essential image of $F$, and every finite-dimensional module in the essential image of $F$ is necessarily nilpotent \cite[\S10]{derksen2008quivers}, hence locally nilpotent.

\begin{defi}
    A module $M\in \pathalg{Q}$-$\Mod$ is \emph{locally nilpotent} if each element $m\in M$ is annihilated by some positive power of the arrow ideal of $\pathalg{Q}$.
\end{defi}

Not every $\pathalg{Q}$-$\Mod$ is locallly nilpotent; e.g., if there is a path from $j\in Q_0$ to a cycle of $Q$, then neither $\pathalg{Q}e_j$ nor $\compalg{Q}e_j$ is locally nilpotent.
But if $M\in \pathalg{Q}$-$\Mod$ is locally nilpotent, then there is exactly one action of $\compalg{Q}$ on $M$ that makes $M$ a $\compalg{Q}$-module and extends the action of $\pathalg{Q}$ on $M$. It would be interesting to know whether this holds for all modules in the essential image of $F$, or to characterize the quivers and modules for which this holds. Needless to say, it would be interesting to know whether $F$ is full, or for which quivers it is full.

\begin{example}\label{ex:functor-F-is-not-dense}
Suppose $Q$ has exactly one vertex and exactly one arrow. Then $\pathalg{Q}= \field[X]$ and $\compalg{Q}=\field[\hspace{-0.05cm}[X]\hspace{-0.05cm}]$. Take $\alpha\in\field\setminus\{0\}$ and consider the $1$-dimensional $\pathalg{Q}$-module $M_\alpha=\field[X]/(X-\alpha)\field[X]$. The action of $\pathalg{Q}$ on $M_\alpha$ cannot be extended to an action of $\compalg{Q}$ on $M_\alpha$ since the element $\alpha-X$ is invertible in $\compalg{Q}$ but acts as zero on $M_\alpha$. Therefore, $F$ is not dense.
\end{example}

The argument in this example can be easily generalized to prove:

\begin{prop}\label{prop:functor-F-is-not-dense}
    If $Q$ is not acyclic, then $F$ is not dense.
\end{prop}

We know that the category $\pathalg{Q}$-$\Mod$ of left $\pathalg{Q}$-modules is equivalent to the category $\Rep_{\field}(Q)$ of representations of $Q$ by $\field$-vector spaces and $\field$-linear maps.
Thus, Proposition \ref{prop:functor-F-is-not-dense} says that when $Q$ is not acyclic, providing the information of a representation of $Q$ does not suffice in order to define a module over the complete path algebra. It would be interesting to give an explicit description of all the quiver representations of $Q$ that correspond to the objects in essential the image of $F$ under the equivalence $\pathalg{Q}$-$\Mod \simeq \Rep_{\field}(Q)$. Equivalently, give an explicit description of the essential image of $F$.

For more on the image of $F$ and on the category of locally nilpotent modules, we refer the reader to 
\cite{iovanov2013complete,dascalescu2013quiver,chin2002almost}.

\section{Quivers with potential and their mutations}

Following \cite{cohn1959on}, if $A$ is a unital ring and $R$ is a unital subring of $A$, we shall say that $A$ is an \emph{$R$-ring}. Concordantly, whenever $A$ and $B$ are $R$-rings and $\varphi:A\rightarrow B$ is a ring homomorphism such that $\varphi|_R=\myid_R$, we will call $\varphi$ an \emph{$R$-ring homomorphism}. The author thanks Raphael Bennett-Tennenhaus for pointing him to this terminology.

For a quiver $Q$, the path algebra $\pathalg{Q}$ and the complete path algebra $\compalg{Q}$ are $\field^{Q_0}$-rings. Given quivers $Q$ and $Q'$ with the same vertex set $Q_0=Q_0'$, a function $\varphi:\compalg{Q}\rightarrow\compalg{Q'}$ is a $\field^{Q_0}$-ring homomorphism if and only if $\varphi$ is a $\field$-linear ring homomorphism such that $\varphi(e_j)=e_j$ for every vertex $j\in Q_0$, if and only if $\varphi$ is a $\field$-algebra homomorphism and a $\field^{Q_0}$-$\field^{Q_0}$-bimodule homomorphism.

\subsection{Reduced parts of a quiver with potential and a module}\label{subsec:red-part-of-quiver-and-module}

Following \cite{derksen2008quivers}, we will write \emph{QP} as a shorthand for \emph{quiver with potential}.

\begin{defi}{\cite[\S4]{derksen2008quivers}}
    Let $(Q',S)$ and $(Q'',S'')$ be QPs such that $Q_0'=Q_0''$. Their \emph{direct sum} is the QP $(Q',S')\oplus(Q'',S''):=(Q'\oplus Q'',S'+S'')$, where
    \begin{enumerate}
        \item $Q'\oplus Q''$ has vertex set $Q_0'$, arrow set $Q_1'\sqcup Q_1''$, tail function $t'\cup t'':Q_1'\sqcup Q_1''\rightarrow Q_0'$ and head function $h'\cup h'':Q_1'\sqcup Q_1''\rightarrow Q_0'$;
        \item $S'+S''$ is the sum of $S'$ and $S''$ in $\compalg{Q'\oplus Q''}$.
    \end{enumerate}
\end{defi}

\begin{defi}{\cite[\S4]{derksen2008quivers}}
Let $(Q,S)$ be a QP. We say that
\begin{itemize}
    \item $(Q,S)$ is \emph{reduced} if the degree-$2$ component of $S$ is zero;
    \item $(Q,S)$ is \emph{trivial} if $S$ has degree $2$ and every arrow $b\in Q_1$ is a $\field$-linear combination $b=\sum_{a\in Q_1}x_a\partial_a(S)$.
\end{itemize}
\end{defi}

\label{subsec:splitting-thm}
\begin{defi}{\cite[Definition 4.2]{derksen2008quivers}}\label{defi:right-equivalence}
    Let $(Q,S)$ and $(Q',S')$ be QPs such that $Q_0=Q_0'$. A \textbf{right-equivalence} $\varphi:(Q',S')\rightarrow (Q,S)$ is a $\field^{Q_0}$-ring isomorphism $\varphi:\compalg{Q'}\rightarrow\compalg{Q}$ such that $\varphi(S')$ is cyclically equivalent to $S$. 
\end{defi}

By \cite[Propositions 2.4 and 3.7]{derksen2008quivers}, every right-equivalence $\varphi:(Q',S')\rightarrow (Q,S)$ induces a quiver isomorphism $Q'\rightarrow Q$ pointwisely fixing the vertex set $Q_0$, as well as a
commutative diagram of $\field^{Q_0}$-ring homomorphisms
\begin{equation}\label{eq:alg-iso-induced-by-a-right-equiv}
\xymatrix{\compalg{Q'}  \ar[r]^{\varphi}_{\cong} \ar[d] & \compalg{Q} \ar[d] \\
\jacobalg{Q',S'} \ar[r]_{\cong}^{\overline{\varphi}} &\jacobalg{Q,S}.
}
\end{equation}
whose rows are isomorphisms, and whose vertical arrows are projections.

\begin{theorem}\cite[Theorem 4.6]{derksen2008quivers}\label{thm:DWZ-splitting-thm}
    Suppose $Q$ is a loop-free quiver. For each potential $S\in\compalg{Q}$ there exist a reduced QP $(Q',S')$ and a trivial QP $(Q'',S'')$ such that $Q= Q'\oplus Q''$, as well as a right-equivalence 
    \[
    \varphi:(Q'\oplus Q'',S'+S'')\rightarrow (Q,S).
    \]
    Furthermore, for any pair of QPs $((C',W'),(C'',W''))$ with  $(C',W')$ reduced and $(C'',W'')$ trivial, if there exists a right-equivalence $\psi:(C'\oplus C'',W'+W'')\rightarrow (Q,S)$,
    then $C'\cong Q'$ and $C''\cong Q''$ as quivers, and there exists a pair of right-equivalences $(C',W')\rightarrow(Q',S')$, $(C'',W'')\rightarrow(Q'',S'')$. 
\end{theorem}

Notice that in Theorem \ref{thm:DWZ-splitting-thm}, $Q'$ and $Q''$ turn out to be subquivers of $Q$ since $Q=Q'\oplus Q''$, so we can take the inclusion $\iota:\compalg{Q'}\rightarrow \compalg{Q}$.
By \cite[Proposition 3.7]{derksen2008quivers} and the comment we have made right after Definition \ref{defi:right-equivalence}, there is a commutative diagram of $\field^{Q_0}$-ring homomorphisms
\begin{equation}\label{eq:alg-isomorphisms-induced-by-splitting-thm}
\xymatrixcolsep{3pc}\xymatrix{\compalg{Q'} \ar[r]^-{\iota} \ar[d] & \compalg{Q'\oplus Q''} \ar[r]^-{\varphi} \ar[d] & \compalg{Q} \ar[d] \\
\jacobalg{Q',S'}\ar[r]_-{\cong}^-{\overline{\iota}} &\jacobalg{Q'\oplus Q'',S'+S''}\ar[r]_-{\cong}^-{\overline{\varphi}} &\jacobalg{Q,S}.
}
\end{equation}
whose bottom row consists of isomorphisms, and whose vertical arrows are the standard projections to the quotients. Thus, once the pair of QPs $((Q',S'),(Q'',S''))$ and the right-equivalence $\varphi:(Q'\oplus Q'',S'+S'')\rightarrow (Q,S)$ have been fixed, we have the commutative diagram of functors 
\begin{equation}\label{eq:functors-induced-by-splitting-data}
\Scale[0.975]{
\xymatrixcolsep{3pc}\xymatrix{\compalg{Q'}\text{-}\Mod   & \compalg{Q'\oplus Q''}\text{-}\Mod \ar[l]_-{\iota^{\natural}}   & \compalg{Q}\text{-}\Mod \ar[l]_-{\varphi^{\natural}} \\
\jacobalg{Q',S'}\text{-}\Mod \ar[u] &\jacobalg{Q'\oplus Q'',S'+S''}\text{-}\Mod \ar[l]^-{\cong}_-{\overline{\iota}^{\natural}} \ar[u]  & \jacobalg{Q,S}\text{-}\Mod \ar[l]^-{\cong}_-{\overline{\varphi}^{\natural}} \ar[u].
}}
\end{equation}
Explicitly, for each $\jacobalg{Q,S}$-module $M$, the $\jacobalg{Q',S'}$-module $(\overline{\iota}^\natural\circ\overline{\varphi}^\natural)(M)$ has $M$ itself as underlying additive group, with the action of 
$u\in\compalg{Q'}$ on $m\in M$ given~by
\begin{equation}\label{eq:def-of-u-star_varphi-m}
u\star_\varphi m:= \varphi(u)\cdot m.
\end{equation}

\begin{defi}{\cite[Definitions 4.13 and 10.4]{derksen2008quivers}}\label{def:reduced-and-trivial-parts}
    With the notation from Theorem~\ref{thm:DWZ-splitting-thm} and the subsequent paragraphs, we say that
    \begin{enumerate}
    \item $(Q_{\operatorname{red}},S_{\operatorname{red}}):=(Q',S')$ (resp. $(Q_{\operatorname{triv}},S_{\operatorname{triv}}):=(Q'',S'')$) and any QP right-equivalent to it is a \emph{reduced part} (resp. \emph{trivial part}) of $(Q,S)$;
    \item the right-equivalence $\varphi:(Q_{\operatorname{red}},S_{\operatorname{red}})\oplus(Q_{\operatorname{triv}},S_{\operatorname{triv}})\rightarrow(Q,S)$ is a \emph{splitting} of $(Q,S)$;
    \item the composition $\varphi\circ\iota$ is the \emph{reduction of $(Q,S)$ with respect to $\varphi$}, we use the notation $\varphi\circ\iota:(Q_{\operatorname{red}},S_{\operatorname{red}})\rightarrow(Q,S)$;
    \item the $\jacobalg{Q',S'}$-module $\varphi^{\#}(M):=(\overline{\iota}^\natural\circ\overline{\varphi}^\natural)(M)$ given by the action \eqref{eq:def-of-u-star_varphi-m} is the \emph{reduced part of the $\jacobalg{Q,S}$-module $M$ with respect to $\varphi$}.
    \end{enumerate}
\end{defi}

\begin{remark}\label{rem:necessity-of-complete-path-algs}
Thanks to \eqref{eq:alg-iso-induced-by-a-right-equiv} and \cite[Proposition 4.5]{derksen2008quivers},
    Theorem \ref{thm:DWZ-splitting-thm} is the key behind the possibility to algebraically delete arrows belonging to $2$-cycles. Its proof is tightly fastened to two main technical bollards: 
    \begin{itemize}
    \item the guarantee that sending each arrow $a\in Q_1$ to itself plus a linear combination of paths of length $\geq 2$ parallel to $a$ always defines a $\field^{Q_0}$-ring \textbf{isomorphism};
    \item the guarantee that certain sequences $(\psi_n)_{n\in\mathbb{Z}_{>0}}$ of $\field^{Q_0}$-ring isomorphisms always have a well defined limit $\lim_{n\to\infty}\psi_n$.
    \end{itemize}
    Complete path algebras have these two desired properties \cite[Proposition 2.4 and the proof of Lemma 4.7]{derksen2008quivers}, but path algebras rarely do\footnote{A typical example is the rule $X\mapsto X+X^2$, which defines a $\field$-algebra automorphism of $\field[\hspace{-0.05cm}[X]\hspace{-0.05cm}]$, but a non-surjective $\field$-algebra endomorphism of $\field[X]$. Loop-free examples abound too.}.
    This is the reason behind the necessity to use complete path algebras instead of path algebras as the framework to develop the mutation theory of QPs.
\end{remark}

\subsection{Premutations and mutations of quivers with potential}
\label{subsec:premuts-and-muts-of-QPs}

\begin{defi}{\cite[Equations (5.4) and (5.8), and Definition 5.5]{derksen2008quivers}}\label{def:premut-and-mut-of-a-QP}
    For a QP $(Q,S)$ and a vertex $k\in Q_0$ not incident to any $2$-cycle of $Q$,
    \begin{enumerate}\item the \emph{premutation} of $(Q,S)$ with respect to $k$ is the QP $\widetilde{\mu}_k(Q,S):=(\widetilde{\mu}_k(Q),\widetilde{\mu}_k(S))$, where the quiver $\widetilde{\mu}_k(Q)$ is given by Definition \ref{def:threesteps}, and
    \begin{equation}\label{eq:def-of-widetilde-mu-k(S)}
        \widetilde{\mu}_k(S):=[S]+\sum_{\substack{a,b\in Q_1\\ h(a)=k=t(b)}} b^*[ba]a^*;
    \end{equation}
\item the \emph{mutation} of $(Q,S)$ with respect to $k$ is the reduced part 
\[
\mu_k(Q,S):=(\widetilde{\mu}_k(Q)_{\operatorname{red}},\widetilde{\mu}_k(S)_{\operatorname{red}}).
\]
\end{enumerate}
\end{defi}

\begin{remark}
    Depending on the choice of $S$, the quiver $\widetilde{\mu}_k(Q)_{\operatorname{red}}$ may fail to be $2$-acyclic. In other words, it is not true that for every potential $S$ on $Q$ one has $\widetilde{\mu}_k(Q)_{\operatorname{red}}=\mu_k(Q)$. See, e.g., \cite[Example 8.5]{derksen2008quivers} or \cite[Example 14]{labardini2009quivers}.
\end{remark}

\begin{defi}\cite[Definition 7.2]{derksen2008quivers}\label{def:non-deg-QP}
    A QP $(Q,S)$ is \emph{non-degenerate} if every QP obtained from $(Q,S)$ by applying a finite sequence of QP-mutations has $2$-acyclic underlying quiver.
\end{defi}

\begin{theorem}\cite[Corollary 7.4]{derksen2008quivers}\label{thm:existence-non-deg-potentials}
    If the underlying ground field $\field$ is uncountable, then on any $2$-acyclic quiver $Q$ there exists at least one non-degenerate potential~$S$.
\end{theorem}

\section{The Derksen-Weyman-Zelevinsky mutation of a module}

\subsection{Preparation}\label{subsec:Preparation}

The following definition is inspired in the notion of \emph{decorated representation} first introduced by Marsh-Reineke-Zelevinsky and Derksen-Weyman-Zelevinsky \cite{marsh2003generalized,derksen2008quivers}.

\begin{defi}\label{def:decorated-representation} Let $J\subseteq \mathfrak{m}^2$ be a closed two-sided ideal of $\compalg{Q}$, and $\Lambda:=\compalg{Q}/J$. A \emph{decorated $\Lambda$-module} is a pair $\mathcal{M}=(M,V)$ consisting of a left $\Lambda$-module $M$ and a left $\field^{Q_0}$-module $V$.
\end{defi}

\begin{remark}\label{rem:fin-dim-not-required-in-def-of-dec-rep} In \cite{marsh2003generalized,derksen2008quivers}, both $M$ and $V$ are required to have finite dimension over $\field$. Here, we allow both of them to be infinite-dimensional over $\field$.
\end{remark}

Let $Q$ be any quiver and take a vertex $k\in Q_0$. Suppose that $Q$ does not have $2$-cycles incident to $k$.
Write $e_{\widehat{k}}:=1-e_k$, and let
\[
e_{\widehat{k}}\widetilde{\mu}_k(Q)e_{\widehat{k}}:=(Q_0\setminus\{k\},e_{\widehat{k}}\widetilde{\mu}_k(Q)_1e_{\widehat{k}},h,t)
\]
be the quiver obtained by deleting from $\widetilde{\mu}_k(Q)$ the vertex $k$ and the arrows incident to $k$. Thus, the arrows of $e_{\widehat{k}}\widetilde{\mu}_k(Q)e_{\widehat{k}}$ are those arrows of $Q$ that are not incident to $k$, together with the composite arrows $[ba]:t(a)\rightarrow h(b)$, for $a,b\in Q_1$ with $t(b)=k=h(a)$, that were added in the first step of quiver mutation.
According to \cite[Lemma 6.2]{derksen2008quivers}, the correspondence that acts as the identity on those arrows of $Q$ that are not incident to $k$ and sends each composite arrow $[ba]$ to the length-$2$ path $ba$, induces a $\field$-algebra isomorphism
\begin{equation}\label{eq:inverse-of-DWZs-[-]-isomorphism}
\compalg{e_{\widehat{k}}\widetilde{\mu}_k(Q)e_{\widehat{k}}}\rightarrow e_{\widehat{k}}\compalg{Q}e_{\widehat{k}}.
\end{equation}

Since the quiver $\widetilde{\mu}_k(Q)$ does not have $2$-cycles incident to $k$ either, we also have the corresponding $\field$-algebra isomorphism $\compalg{e_{\widehat{k}}\widetilde{\mu}_k^2(Q)e_{\widehat{k}}}\rightarrow e_{\widehat{k}}\compalg{\widetilde{\mu}_k(Q)}e_{\widehat{k}}$, whose inverse we shall denote
\begin{equation}\label{eq:putting-brackets-for-a*b*}
\kappa:e_{\widehat{k}}\compalg{\widetilde{\mu}_k(Q)}e_{\widehat{k}}\rightarrow \compalg{e_{\widehat{k}}\widetilde{\mu}_k^2(Q)e_{\widehat{k}}},
\end{equation}
Thus, the effect of $\kappa$ on any element of $e_{\widehat{k}}\compalg{\widetilde{\mu}_k(Q)}e_{\widehat{k}}$ consists in substituting each appearance of length-$2$ path $a^*b^*$, for $a,b\in Q_1$ such that $t(b)=k=h(a)$, with the composite arrow $[a^*b^*]:h(b)\rightarrow t(a)$.

Now, let $S\in\compalg{Q}$ be a potential on $Q$. By \cite[Proposition 2.4]{derksen2008quivers}, there is a unique $\field$-ring homomorphism
\[
\psi:\compalg{e_{\widehat{k}}\widetilde{\mu}_k^2(Q)e_{\widehat{k}}}\rightarrow e_{\widehat{k}}\compalg{Q}e_{\widehat{k}}
\]
such that for each arrow $c$ of $e_{\widehat{k}}\widetilde{\mu}_k^2(Q)e_{\widehat{k}}$:
\begin{equation}\label{eq:alg-morphism-from-compalgtildemuktildemukQ-to-compalgQ}
\Scale[0.975]{
    \psi(c) := \begin{cases}
    c & \text{if $c\in Q_1$};\\
    ba & \text{if $c=[ba]$ with $a,b\in Q_1$ such that $t(b)=k=h(a)$};\\
    -\partial_{ba}(S) & \text{if $c=[a^*b^*]$ with $a,b\in Q_1$ such that $t(b)=k=h(a)$}.
    \end{cases}
    }
\end{equation}

Let $a_1,\ldots,a_s$ (resp. $b_1,\ldots,b_r$) be the arrows of $Q$ ending at $k$ (resp. starting at $k$). Then each element $u$ of $\compalg{\widetilde{\mu}_k(Q)}$ can be uniquely written as a finite sum
\begin{equation}\label{eq:expression-of-element-of-tilde-mu-k(Q)}
u= ze_k+u_{\widehat{k}} + \sum_{i=1}^sv_{\widehat{k},i}a_i^* + \sum_{j=1}^rb_j^*w_{\widehat{k},j}+\sum_{i=1}^s\sum_{j=1}^rb_j^* x_{\widehat{k},i,j}a_i^*
\end{equation}
with $z\in\field$, $u_{\widehat{k}}\in e_{\widehat{k}}\compalg{\widetilde{\mu}_k(Q)}e_{\widehat{k}}$,
$v_{\widehat{k},i}\in e_{\widehat{k}}\compalg{\widetilde{\mu}_k(Q)}e_{t(a_i)}$,
$w_{\widehat{k},j}\in e_{h(b_j)}\compalg{\widetilde{\mu}_k(Q)}e_{\widehat{k}}$ and $x_{\widehat{k},i,j}\in e_{h(b_j)}\compalg{\widetilde{\mu}_k(Q)}e_{t(a_i)}$. Note that $e_{\widehat{k}}\compalg{\widetilde{\mu}_k(Q)}e_{\widehat{k}}$ contains each of the sets
\[
e_{\widehat{k}}\compalg{\widetilde{\mu}_k(Q)}e_{t(a_i)}, \qquad 
e_{h(b_j)}\compalg{\widetilde{\mu}_k(Q)}e_{\widehat{k}}, \qquad 
e_{h(b_j)}\compalg{\widetilde{\mu}_k(Q)}e_{t(a_i)}.
\]

Next, take a decorated $\jacobalg{Q,S}$-module $\mathcal{M}=(M,V)$. Following \cite[Equations (10.2), (10.3), (10.4) and (10.5)]{derksen2008quivers}, we form Derksen-Weyman-Zelevinsky's $\alpha$-$\beta$-$\gamma$-triangle
\begin{equation}\label{eq:DWZ-alphabetagamma-triangle}
\xymatrix{
 & M_k  \ar[dr]^{\beta:=\left[\begin{array}{c}(b_1)_M \\ \vdots \\ (b_r)_M \end{array}\right]} & \\
 M_{\operatorname{in}}:=\bigoplus_{i=1}^s M_{t(a_i)}  \ar[ur]^{\alpha:=\left[\begin{array}{ccc}(a_1)_M  & \cdots & (a_s)_M \end{array}\right]\qquad} & & \bigoplus_{j=1}^r M_{h(b_j)}=:M_{\operatorname{out}}  \ar@/^0.55pc/[ll]^{\gamma:=\text{ {\tiny $\left[\begin{array}{ccc}
 {\partial_{b_1a_1}(S)}_M  & \cdots & {\partial_{b_r a_1}(S)}_M \\
 \vdots & \ddots & \vdots\\
 {\partial_{b_1a_s}(S)}_M & \cdots & {\partial_{b_ra_s}(S)}_M
 \end{array}\right]$}}}
 }
\end{equation}

Just as in \cite[Lemma 10.6]{derksen2008quivers}, we have 
\begin{equation}\label{eq:alphagamma=0=gammabeta}
\alpha\gamma=0=\gamma\beta,
\end{equation}
i.e., 
$\image\gamma\subseteq\ker\alpha$ and $\image\beta\subseteq\ker\gamma$. Following \cite[Equations (10.8), (10.9)]{derksen2008quivers}, let 
\begin{itemize}
\item $\rho:M_{\operatorname{out}}\rightarrow\ker\gamma$ be any $\field$-linear map such that the composition $\ker\gamma\overset{\iota}{\hookrightarrow} M_{\operatorname{out}}\overset{\rho}{\rightarrow}\ker\gamma$ is the identity; and
\item $\sigma:\ker\alpha/\image\gamma\rightarrow\ker\alpha$ be any $\field$-linear map such that the composition $\frac{\ker\alpha}{\image\gamma}\overset{\sigma}{\rightarrow}\ker\alpha\twoheadrightarrow\frac{\ker\alpha}{\image\gamma}$ is the identity (where $\ker\alpha\twoheadrightarrow\frac{\ker\alpha}{\image\gamma}$ is the natural projection).
\end{itemize}
Such retraction $\rho$ and section $\sigma$ exist because every $\field$-vector space has a basis.

Finally, $\iota:\ker\alpha\rightarrow M_{\operatorname{in}}$ will be the inclusion and $p:M_{\operatorname{out}}\rightarrow\coker\beta$ will be the canonical projection. The restriction of $\iota$ to $\image\gamma$ will be denoted $\iota$ too, and the restriction of $p$ to $\ker\gamma$ will also be denoted $p$. This restriction is precisely the canonical projection $\ker\gamma\rightarrow\ker\gamma/\image\beta$. 

\subsection{\texorpdfstring{The four definitions of $\overline{M}$}{The four definitions of M}}\label{subsec:the-four-defs-of-ovelineM} As we have hinted in \S\ref{sec:DWZ-amalgamation-idea}, the Derksen-Weyman-Zelevinsky approach refines the diagram \eqref{eq:BGP-DWZ-mutation-tentative} by amalgamating $\coker\beta$ and $\ker\alpha$ along their common subspace $\image\gamma$. There are several ways to state this amalgamation formally. We start with the one done originally in \cite[Equations (10.6), (10.7) and (10.10)]{derksen2008quivers}, followed by three other alternative descriptions, two of which have already appeared in \cite{geiss2016representation,labardini2022landau}. 

\subsubsection{\texorpdfstring{Definition of $\overline{M}_k$, $\overline{\alpha}$ and $\overline{\beta}$ resorting to two non-canonical choices}{Definition of Mk, a and b resorting to two non-canonical choices}}
\label{subsubsec:def-of-overlineM-through-two-non-canonical-choices}

Set
\begin{equation}\label{eq:def-of-overlineM-overlineV-two-non-canonical-choices}
\Scale[0.95]{
\overline{M}_j:=\begin{cases}
    M_j & \text{if} \ j\neq k;\\
    \frac{\ker\gamma}{\image\beta}\oplus\image\gamma\oplus\frac{\ker\alpha}{\image\gamma}\oplus V_k & \text{if} \ j=k;
\end{cases}
\qquad 
\overline{V}_j:=\begin{cases}
V_j & \text{if} \ j\neq k;\\
\frac{\ker\beta}{\ker\beta\cap\image\alpha} & \text{if} \ j=k.
\end{cases}
}
\end{equation}
\begin{equation}\label{eq:def-of-ai*-and-bj*-two-non-canonical-choices}
\Scale[0.975]{
\xymatrix{
 & & & \overline{M}_k \ar[dlll]_{\text{{\tiny $\left[\begin{array}{c}(a_1^*)_{\overline{M}} \\ \vdots \\ (a_s^*)_{\overline{M}}\end{array} \right]$}}:=\overline{\beta}:=\left[\begin{array}{cccc}0 & \iota & \iota\sigma & 0\end{array}\right]\qquad} & & &\\
 M_{\operatorname{in}} & & & & & & M_{\operatorname{out}} \ar[ulll]_{\qquad\qquad \left[\begin{array}{c} -p\circ\rho\\ -\gamma \\ 0 \\ 0\end{array}\right]=:\overline{\alpha}=:\text{{\tiny $\left[\begin{array}{ccc}(b_1^*)_{\overline{M}}  \ldots & (b_r^*)_{\overline{M}}\end{array}\right]$}}}
}
}
\end{equation}

\subsubsection{\texorpdfstring{Definition of $\overline{M}_k$, $\overline{\alpha}$ and $\overline{\beta}$ giving preference to $\ker\alpha$}{Definition of Mk, a and b giving preference to kera}}
\label{subsubsec:def-of-overlineM-giving-preference-to-ker-alpha} 
Set
\begin{equation}\label{eq:def-of-overlineM-overlineV-keralpha-preference}
\overline{M}_j:=\begin{cases}
    M_j & \text{if} \ j\neq k;\\
    \frac{\ker\gamma}{\image\beta}\oplus\ker\alpha\oplus V_k & \text{if} \ j=k;
\end{cases}
\qquad 
\overline{V}_j:=\begin{cases}
V_j & \text{if} \ j\neq k;\\
\frac{\ker\beta}{\ker\beta\cap\image\alpha} & \text{if} \ j=k.
\end{cases}
\end{equation}
\begin{equation}\label{eq:def-of-ai*-and-bj*-keralpha-preference}
    \xymatrix{
 & & & \overline{M}_k \ar[dlll]_{\text{{\tiny $\left[\begin{array}{c}(a_1^*)_{\overline{M}} \\ \vdots \\ (a_s^*)_{\overline{M}}\end{array} \right]$}}:=\overline{\beta}:=\left[\begin{array}{ccc}0 & \iota & 0\end{array}\right] \qquad} & & &\\
 M_{\operatorname{in}} & & & & & & M_{\operatorname{out}} \ar[ulll]_{\qquad\qquad \left[\begin{array}{c} -p\circ\rho\\ -\gamma \\ 0\end{array}\right]=:\overline{\alpha}=:\text{{\tiny $\left[\begin{array}{ccc}(b_1^*)_{\overline{M}}  \ldots & (b_r^*)_{\overline{M}}\end{array}\right]$}}}
}
\end{equation}

\subsubsection{\texorpdfstring{Definition of $\overline{M}_k$, $\overline{\alpha}$ and $\overline{\beta}$ giving preference to $\coker\beta$}{Definition of Mk, a and b giving preference to coker b}}
\label{subsubsec:def-of-overlineM-giving-preference-to-coker-beta}

Set
\begin{equation}\label{eq:def-of-overlineM-overlineV-cokerbeta-preference}
\overline{M}_j:=\begin{cases}
    M_j & \text{if} \ j\neq k;\\
    \coker\beta\oplus\frac{\ker\alpha}{\image\gamma}\oplus V_k & \text{if} \ j=k;
\end{cases}
\qquad 
\overline{V}_j:=\begin{cases}
V_j & \text{if} \ j\neq k;\\
\frac{\ker\beta}{\ker\beta\cap\image\alpha} & \text{if} \ j=k.
\end{cases}
\end{equation}
\begin{equation}\label{eq:def-of-ai*-and-bj*-cokerbeta-preference}
        \xymatrix{
 & & & \overline{M}_k \ar[dlll]_{\text{{\tiny $\left[\begin{array}{c}(a_1^*)_{\overline{M}} \\ \vdots \\ (a_s^*)_{\overline{M}}\end{array} \right]$}}:=\overline{\beta}:=\left[\begin{array}{ccc}\overline{\gamma} & \iota\sigma & 0\end{array}\right] \qquad} & & &\\
 M_{\operatorname{in}} & & & & & & M_{\operatorname{out}} \ar[ulll]_{\qquad\qquad \left[\begin{array}{c}-p \\ 0 \\ 0\end{array}\right]=:\overline{\alpha}=:\text{{\tiny $\left[\begin{array}{ccc}(b_1^*)_{\overline{M}}  \ldots & (b_r^*)_{\overline{M}}\end{array}\right]$}}}
}
\end{equation}
where $\overline{\gamma}:\coker\beta\rightarrow M_{\operatorname{in}}$ is induced by $\gamma$ 
(recall from \eqref{eq:alphagamma=0=gammabeta} that $\image\beta\subseteq\ker\gamma$).

\subsubsection{\texorpdfstring{Definition of $\overline{M}_k$, $\overline{\alpha}$ and $\overline{\beta}$ via a pushout of $\coker\beta \leftarrowtail \image\gamma\hookrightarrow\ker\alpha$}{Definition of Mk, a and b via a pushout of coker b<-im c->ker a}}
\label{subsubsec:def-of-overlineM-via-a-pushout}
Consider the short exact sequence of $\field$-vector spaces and $\field$-linear maps
\[
\xymatrix{0 \ar[r] & \ker\gamma \ar[r] & M_{\operatorname{out}}  \ar[r]^{\gamma} & \image\gamma \ar[r] & 0.
}
\]
Since we have a retraction $\rho:M_{\operatorname{out}}\rightarrow\ker\gamma$, we also have the corresponding section $\mathfrak{s}:\image\gamma\rightarrow M_{\operatorname{out}}$ such that $\gamma\mathfrak{s}=\myid_{\image\gamma}$, given by $\mathfrak{s}(x):=y-\rho(y)$ where $y$ is any element of $M_{\operatorname{out}}$ such that $\gamma(y)=x\in\image\gamma$. Then $\overline{\gamma}p\mathfrak{s}=\myid_{\image\gamma}$, i.e., the composition $p\mathfrak{s}:\image\gamma\rightarrow \coker\beta$ is a section to $\overline{\gamma}:\coker\beta\rightarrow\image\gamma$. For $j\in Q_0$ define
\begin{equation}\label{eq:def-of-overlineM-overlineV-via-a-pushout}
\Scale[0.95]{
\overline{M}_j:=\begin{cases}
    M_j & \text{if} \ j\neq k;\\
    \frac{\coker\beta\oplus\ker\alpha}{\{(p\mathfrak{s}(m),-m)\suchthat m\in\image\gamma\}}\oplus V_k & \text{if} \ j=k;
\end{cases}
\qquad 
\overline{V}_j:=\begin{cases}
V_j & \text{if} \ j\neq k;\\
\frac{\ker\beta}{\ker\beta\cap\image\alpha} & \text{if} \ j=k.
\end{cases}
}
\end{equation}
Notice that the first direct summand in the definition of $\overline{M}_k$ results from taking the pushout of the $\field$-linear maps $\coker\beta \underset{p\mathfrak{s}}{\leftarrowtail} \image\gamma\hookrightarrow\ker\alpha$.
We define
\begin{equation}\label{eq:def-of-ai*-and-bj*-via-a-pushout}
    \xymatrix{
 & & & \overline{M}_k \ar[dlll]_{\text{{\tiny $\left[\begin{array}{c}(a_1^*)_{\overline{M}} \\ \vdots \\ (a_s^*)_{\overline{M}}\end{array} \right]$}}:=\overline{\beta}:=\left[\begin{array}{ccc}\overline{\iota}  & 0\end{array}\right] \qquad} & & &\\
 M_{\operatorname{in}} & & & & & & M_{\operatorname{out}} \ar[ulll]_{\qquad\qquad \left[\begin{array}{c}-p \\ 0 \\ 0\end{array}\right]=:\overline{\alpha}=:\text{{\tiny $\left[\begin{array}{c} -q p \rho - q p\mathfrak{s}\gamma_k \\ 0\end{array}\right]$}}}
}
\end{equation}
where 
$\overline{\iota}$ is the unique $\field$-linear map that makes the diagram
\begin{equation}\label{eq:overlinebeta-using-pushout-of-cokerbeta<-imgamma->keralpha}
\xymatrix{\image\gamma_k \ar@{^{(}->}[r] \ar@{^{(}->}[d]_{p\mathfrak{s}} & \ker\alpha_k \ar[d]^{\jmath} \ar@/^0.75pc/[ddr]^{\iota} &\\
\coker\beta \ar[r]_-{q} \ar@/_0.75pc/[drr]_{\overline{\gamma_k}} & \frac{\coker\beta\oplus\ker\alpha}{\{(p\mathfrak{s}(m),-m)\suchthat m\in\image\gamma_k\}} \ar@{-->}[dr]|-{\overline{\iota}} & \\
& & M_{\operatorname{in}}
}
\end{equation}
commute ($q$ and $\jmath$ are the canonical maps to the pushout).

\subsection{Equivalence between the four definitions of \texorpdfstring{$\overline{M}$}{M}}\label{subsec:equiv-between-all-defs-of-overlineM}
In \S\ref{subsubsec:def-of-overlineM-through-two-non-canonical-choices}, \S\ref{subsubsec:def-of-overlineM-giving-preference-to-ker-alpha}, \S\ref{subsubsec:def-of-overlineM-giving-preference-to-coker-beta} and \S\ref{subsubsec:def-of-overlineM-via-a-pushout} we have defined four quiver representations of the premutation $\widetilde{\mu}_k(Q)$, all denoted $\overline{M}$. In this subsection we show that they are all isomorphic.

\begin{prop}\label{prop:three-defs-of-overlineM-are-equivalent} Let $(Q,S)$  be a QP, $\mathcal{M}=(M,V)$ a decorated module over the Jacobian algebra $\jacobalg{Q,S}$, and $k\in Q_0$ a vertex not incident to any $2$-cycle of~$Q$. The four representations of $\widetilde{\mu}_k(Q)$ defined in Subsections \ref{subsubsec:def-of-overlineM-through-two-non-canonical-choices}, \ref{subsubsec:def-of-overlineM-giving-preference-to-ker-alpha}, \ref{subsubsec:def-of-overlineM-giving-preference-to-coker-beta} and \ref{subsubsec:def-of-overlineM-via-a-pushout} are isomorphic as quiver representations, i.e., as left $\pathalg{\widetilde{\mu}_k(Q)}$-modules.
\end{prop}

\begin{proof} 
Let $f:\frac{\coker\beta\oplus\ker\alpha}{\{(p\mathfrak{s}(m),-m)\suchthat m\in\image\gamma_k\}}\rightarrow \frac{\ker\gamma}{\image\beta}\oplus \ker\alpha$ be the unique linear map making the diagram
\begin{equation*}
\xymatrix{\image\gamma_k \ar@{^{(}->}[r] \ar@{^{(}->}[d]_{p\mathfrak{s}} & \ker\alpha_k \ar[d]^{\jmath} \ar@/^0.75pc/[ddr]^{{\text{\tiny $\left[\begin{array}{c}0\\ \myid\end{array}\right]$}}} &\\
\coker\beta \ar[r]^-{q} \ar@/_0.75pc/[drr]|-{{\text{\tiny $\left[\begin{array}{c}\overline{\rho}\\ \overline{\gamma_k}\end{array}\right]$}}} & \frac{\coker\beta\oplus\ker\alpha}{\{(p\mathfrak{s}(m),-m)\suchthat m\in\image\gamma_k\}} \ar@{-->}[dr]|-{f} & \\
& & \frac{\ker\gamma}{\image\beta}\oplus \ker\alpha
}
\end{equation*}
commute.
Using \eqref{eq:def-of-ai*-and-bj*-two-non-canonical-choices}, \eqref{eq:def-of-ai*-and-bj*-keralpha-preference}, \eqref{eq:def-of-ai*-and-bj*-cokerbeta-preference}, \eqref{eq:def-of-ai*-and-bj*-via-a-pushout} and \eqref{eq:overlinebeta-using-pushout-of-cokerbeta<-imgamma->keralpha}, we deduce that the diagram of $\field$-vector spaces and $\field$-linear maps
    \[
    \xymatrix{
    M_{\operatorname{in}} \ar@/_0.5pc/[dd]_{\myid} & & & \frac{\coker\beta\oplus\ker\alpha}{\{(p\mathfrak{s}(m),-m)\suchthat m\in\image\gamma_k\}}\oplus V_k \ar@/_0.5pc/[dd]_{{\text{\tiny $\left[\begin{array}{cc} f & 0 \\ 0 & \myid_{V_k} \end{array}\right]$}}} \ar[lll]|-{\left[\begin{array}{ccc}\overline{\iota}  & 0\end{array}\right]} & & & M_{\operatorname{out}} \ar@/_0.5pc/[dd]_{\myid} \ar[lll]|-{{\text {\tiny $\left[\begin{array}{c} -q p \rho - \jmath\gamma_k \\ 0\end{array}\right]$}}} \\
     & & &  & & & \\
     M_{\operatorname{in}} \ar@/_0.5pc/[dd]_{\myid} \ar@/_0.5pc/[uu]_{\myid} & & & \frac{\ker\gamma}{\image\beta}\oplus \ker\alpha \oplus V_k \ar@/_0.5pc/[uu]_{\text{{\tiny $\left[\begin{array}{ccc}q & \jmath & 0\\ 0 & 0 & \myid\end{array}\right]$}}} \ar@/_0.5pc/[dd]_{\text{{\tiny $\left[\begin{array}{ccc}\myid & 0 & 0\\ 0 & \myid-\sigma \pi & 0\\ 0 & q & 0\\ 0 & 0 & \myid \end{array}\right]$}}} \ar[lll]|-{\left[\begin{array}{ccc}0 & \iota & 0\end{array}\right]} & & & M_{\operatorname{out}} \ar@/_0.5pc/[dd]_{\myid} \ar@/_0.5pc/[uu]_{\myid} \ar[lll]|-{\left[\begin{array}{c}-p\rho\\ -\gamma \\ 0\end{array}\right]} \\
      & & &  & & & \\
     M_{\operatorname{in}} \ar@/_0.5pc/[dd]_{\myid} \ar@/_0.5pc/[uu]_{\myid} & & & \frac{\ker\gamma}{\image\beta}\oplus \image\gamma \oplus \frac{\ker\alpha}{\image\gamma}\oplus V_k \ar@/_0.5pc/[dd]_{\text{{\tiny $\left[\begin{array}{cccc} \overline{\iota} & \overline{\myid-\iota\rho}\overline{\gamma_k}^{-1} & 0 & 0\\ 0 & 0 & \myid & 0\\ 0 & 0 & 0 & \myid  \end{array}\right]$}}} \ar[lll]|-{\text{{\tiny $\left[\begin{array}{cccc}0 & \iota & \iota\sigma & 0\end{array}\right]$}} }  \ar@/_0.5pc/[uu]_{\text{{\tiny $\left[\begin{array}{cccc} \myid & 0 & 0 & 0 \\ 0 & \iota & \sigma & 0\\ 0 & 0 &  0 & \myid \end{array}\right]$}}}   & & & M_{\operatorname{out}} \ar@/_0.5pc/[dd]_{\myid} \ar@/_0.5pc/[uu]_{\myid} \ar[lll]|-{\left[\begin{array}{c}-p\rho \\ -\gamma \\ 0 \\ 0\end{array}\right]}  \\
    & & &  & & & \\
     M_{\operatorname{in}} \ar@/_0.5pc/[uu]_{\myid} &&& \coker\beta\oplus  \frac{\ker\alpha}{\image\gamma}\oplus V_k \ar@/_0.5pc/[uu]_{\text{{\tiny $\left[\begin{array}{ccc}\overline{\rho} & 0 & 0\\ \overline{\gamma} & 0 & 0\\ 0 &\myid & 0\\ 0 & 0 &\myid \end{array}\right]$}}} \ar[lll]|-{\left[\begin{array}{ccc}\overline{\gamma} & \iota\sigma & 0\end{array}\right]} & & & M_{\operatorname{out}} \ar@/_0.5pc/[uu]_{\myid} \ar[lll]|-{\left[\begin{array}{c}-p\\ 0 \\ 0\end{array}\right]} 
    }
    \]
    commutes, and that its $2$-cycles are formed by mutually inverse maps. Proposition \ref{prop:three-defs-of-overlineM-are-equivalent} follows.
\end{proof}

The following lemma is an immediate consequence of \eqref{eq:def-of-ai*-and-bj*-two-non-canonical-choices} (resp. \eqref{eq:def-of-ai*-and-bj*-keralpha-preference}, resp. \eqref{eq:def-of-ai*-and-bj*-cokerbeta-preference}, resp. \eqref{eq:def-of-ai*-and-bj*-via-a-pushout} and \eqref{eq:overlinebeta-using-pushout-of-cokerbeta<-imgamma->keralpha}).

\begin{lemma}\label{lemma:overlinebeta-overlinealpha=-gamma} The linear maps $\overline{\alpha}:\overline{M}_k\rightarrow M_{\operatorname{in}}$, $\overline{\beta}:M_{\operatorname{out}}\rightarrow \overline{M}_k$ and $\gamma:M_{\operatorname{out}}\rightarrow M_{\operatorname{in}}$ from \eqref{eq:DWZ-alphabetagamma-triangle}, \eqref{eq:def-of-ai*-and-bj*-two-non-canonical-choices}, \eqref{eq:def-of-ai*-and-bj*-keralpha-preference} and \eqref{eq:def-of-ai*-and-bj*-cokerbeta-preference} satisfy
\begin{equation*}
\left[\begin{array}{ccc}
(a_1^*b_1^*)_{\overline{M}} & \ldots & (a_1^*b_r^*)_{\overline{M}}\\
\vdots & \ddots & \vdots\\
(a_s^*b_1^*)_{\overline{M}} & \ldots & (a_s^*b_r^*)_{\overline{M}} 
\end{array}
\right]   
= \overline{\beta}\overline{\alpha} = -\gamma = -\left[\begin{array}{ccc}
 {\partial_{b_1a_1}(S)}_M  & \cdots & {\partial_{b_r a_1}(S)}_M \\
 \vdots & \ddots & \vdots\\
 {\partial_{b_1a_s}(S)}_M & \cdots & {\partial_{b_ra_s}(S)}_M
 \end{array}\right].
\end{equation*}
\end{lemma}

\begin{remark}
    All the definitions and results from \S\ref{subsubsec:def-of-overlineM-through-two-non-canonical-choices}, \S\ref{subsubsec:def-of-overlineM-giving-preference-to-ker-alpha}, \S\ref{subsubsec:def-of-overlineM-giving-preference-to-coker-beta}, \S\ref{subsubsec:def-of-overlineM-via-a-pushout} and \S\ref{subsec:equiv-between-all-defs-of-overlineM} remain valid if instead of assuming that $M$ is a left $\jacobalg{Q,S}$-module one supposes that $S\in\pathalg{Q}$ and that $M$ is a left $\pathalg{Q}$-module annihilated by all cyclic derivatives of $S$.
\end{remark}

\subsection{\texorpdfstring{Showing $\overline{M}$ is a $\jacobalg{\widetilde{\mu}_k(Q,S)}$-module}{Showing -M is a  P(mk(Q,S))-module}}
\label{subsec:widetildemu_k(M)-is-module-not-only-quiver-rep}
We continue to work under the assumptions and notations from \S\ref{subsec:Preparation}.
We have seen that the four quiver representations that we have denoted $\overline{M}$ are, indeed, isomorphic as quiver representations. By the well-known equivalence between quiver representations and modules over path algebras, $\overline{M}$ can be unapologetically thought to be a single left $\pathalg{\widetilde{\mu}_k(Q)}$-module, denoted $\overline{M}$ as well, with $\eta:\pathalg{\widetilde{\mu}_k(Q)}\rightarrow \End_{\field}(\overline{M})$ as the $\field^{Q_0}$-ring homomorphism giving the action of $\pathalg{\widetilde{\mu}_k(Q)}$ on $\overline{M}$. In this subsection we prove that this action $\eta$ can be extended to an action $\widetilde{\eta}:\compalg{\widetilde{\mu}_k(Q)}\rightarrow \End_{\field}(\overline{M})$ making $\overline{M}$ a left $\jacobalg{\widetilde{\mu}_k(Q),\widetilde{\mu}_k(S)}$-module. 
Notice that the $\field$-linear maps $\eta(a_i^*):=(a_i^*)_{\overline{M}}$ and $\eta(b_i^*):=(b_j^*)_{\overline{M}}$ are the ones from \eqref{eq:def-of-ai*-and-bj*-two-non-canonical-choices} (resp. \eqref{eq:def-of-ai*-and-bj*-keralpha-preference}, resp. \eqref{eq:def-of-ai*-and-bj*-cokerbeta-preference}).

\begin{defi}\label{def:action-of-arb-u-on-overlineM} Let $(Q,S)$  be a QP, $\mathcal{M}=(M,V)$ a decorated module over the Jacobian algebra $\jacobalg{Q,S}$, and $k\in Q_0$ a vertex not incident to any $2$-cycle of~$Q$. Set 
\begin{equation}\label{eq:overline-M-as-left-R-module}
\overline{M}:=\bigoplus_{i\in Q_0}\overline{M}_i, \qquad \overline{V}:=\bigoplus_{i\in Q_0}\overline{V}_i
\end{equation}
as left $\field^{Q_0}$-modules, where the $\field$-vector spaces $\overline{M}_i$ are defined following \eqref{eq:def-of-overlineM-overlineV-two-non-canonical-choices} (resp. \eqref{eq:def-of-overlineM-overlineV-keralpha-preference}, resp. \eqref{eq:def-of-overlineM-overlineV-cokerbeta-preference}). Furthermore, let
$\theta:\End_{\field}(e_{\widehat{k}}M)\rightarrow \End_{\field}(\overline{M})$
 be the $\field$-linear map sending each $f\in \End_{\field}(e_{\widehat{k}}M)$ to the composition $\overline{M}\twoheadrightarrow  e_{\widehat{k}}M\overset{f}{\rightarrow } e_{\widehat{k}}M\hookrightarrow \overline{M}$, and
let
$
\widetilde{\rho}:\compalg{Q}\rightarrow\End_{\field}(M)
$
be the $\field^{Q_0}$-ring homomorphism giving the left $\compalg{Q}$-module structure on $M$.
For $u\in \compalg{\widetilde{\mu}_k(Q)}$, use the decomposition \eqref{eq:expression-of-element-of-tilde-mu-k(Q)} to define a $\field$-linear map $\widetilde{\eta}(u):\overline{M}\rightarrow\overline{M}$ by means of the rule
\begin{align}\label{eq:action-of-arbitrary-u-on-overlineM}
\widetilde{\eta}(u) &:= \eta(ze_k)+
\theta\widetilde{\rho}\psi\kappa(u_{\widehat{k}}) 
+ \sum_{i=1}^s\theta\widetilde{\rho}\psi\kappa(v_{\widehat{k},i})\circ(\eta(a_i^*) )
+ \\
&\phantom{:=}\ 
\Scale[0.9]{+\sum_{j=1}^r(\eta(b_j^*))\circ\theta\widetilde{\rho}\psi\kappa(w_{\widehat{k},j})+\sum_{i=1}^s\sum_{j=1}^r(\eta(b_j^*)) \circ\theta\widetilde{\rho}\psi\kappa(x_{\widehat{k},i,j}) \circ (\eta(a_i^*))},\notag
\end{align}
where the symbol $\circ$ is used to explicitly indicate the composition within $\End_{\field}(\overline{M})$.
\end{defi}

\begin{prop}\label{prop:overlineM-is-left-compalg-module}
For $(Q,S)$, $\mathcal{M}=(M,V)$, $k\in Q_0$ and $\overline{M}$ as in Definition \ref{def:action-of-arb-u-on-overlineM},  the assignment $u\mapsto \widetilde{\eta}(u)$ is a $\field^{Q_0}$-ring homomorphism $\widetilde{\eta}:\compalg{\widetilde{\mu}_k(Q)}\rightarrow\End_{\field}(\overline{M})$. In fact, it is the unique $\field^{Q_0}$-ring homomorphism $\compalg{\widetilde{\mu}_k(Q)}\rightarrow\End_{\field}(\overline{M})$ fitting into a commutative diagram
\begin{equation}\label{eq:diagram-overlineM-is-indeed-a-module-over-compalgwidetildemukQ}
\begin{tikzcd}        
\pathalg{\widetilde{\mu}_k(Q)} \arrow[dr,"\eta"'] \arrow[hook]{r} 
& \compalg{\widetilde{\mu}_k(Q)} \ar[d,dashed,"\widetilde{\eta}"]
& 
e_{\widehat{k}}\compalg{\widetilde{\mu}_k(Q)}e_{\widehat{k}} \ar[r,"\kappa"] \arrow[hook]{l}
&   
\compalg{e_{\widehat{k}}\widetilde{\mu}_k^2(Q)e_{\widehat{k}}} \ar[d, "\psi"] \\
& \End_{\field}(\overline{M}) 
& \End_{\field}(e_{\widehat{k}}M) \ar[l,"\theta"]
& e_{\widehat{k}}\compalg{Q}e_{\widehat{k}}. \ar[l,"\widetilde{\rho}|"]
\end{tikzcd}
\end{equation}
If $\mathcal{M}$ is locally nilpotent (resp. nilpotent, resp. finite-dimensional), then so is~$(\overline{M},\overline{V})$.
\end{prop}

\begin{proof} 
It is clear that $\widetilde{\eta}$ is $\field$-linear. The leftmost triangle in \eqref{eq:diagram-overlineM-is-indeed-a-module-over-compalgwidetildemukQ} commutes by Lemma \ref{lemma:overlinebeta-overlinealpha=-gamma} and the definition of $\widetilde{\eta}$, $\theta$, $\widetilde{\rho}$, $\psi$ and $\kappa$. The rectangle on the right hand side of \eqref{eq:diagram-overlineM-is-indeed-a-module-over-compalgwidetildemukQ} commutes by the definition of $\widetilde{\eta}$, $\theta$, $\widetilde{\rho}$, $\psi$ and $\kappa$.

To show that $\widetilde{\eta}$ is a ring homomorphism, take $u_1,u_2\in\compalg{\widetilde{\mu}_k(Q)}$ and write
\begin{align*}
    u_1 &= z_1e_k+u_{1,\widehat{k}} + \sum_{i=1}^sv_{1,\widehat{k},i}a_i^* + \sum_{j=1}^rb_j^*w_{1,\widehat{k},j}+\sum_{i=1}^s\sum_{j=1}^rb_j^* x_{1,\widehat{k},i,j}a_i^*\\
    u_2 &= z_2e_k+u_{2,\widehat{k}} + \sum_{i=1}^sv_{2,\widehat{k},i}a_i^* + \sum_{j=1}^rb_j^*w_{2,\widehat{k},j}+\sum_{i=1}^s\sum_{j=1}^rb_j^* x_{2,\widehat{k},i,j}a_i^*
\end{align*}
as in \eqref{eq:expression-of-element-of-tilde-mu-k(Q)}. Then
direct computation applying \eqref{eq:action-of-arbitrary-u-on-overlineM} shows that
\begin{align}\label{eq:tilderho(u1u2)}
&\widetilde{\eta}(u_1u_2)= \\
    &\Scale[0.95]{= z_1z_2\eta(e_k)
    +  \theta\widetilde{\rho}\psi\kappa(u_{1,\widehat{k}}u_{2,\widehat{k}})
    - \sum_{j=1}^r\sum_{i=1}^s\theta\widetilde{\rho}\left(\psi\kappa(v_{1,\widehat{k},i}) \partial_{b_ja_i}(S) \psi\kappa(w_{2,\widehat{k},j})\right)} \notag \\
    & 
    \Scale[0.825]{+ \sum_{i=1}^s\Big(\theta\widetilde{\rho}\big(\psi\kappa(u_{1,\widehat{k}} v_{2,\widehat{k},i}
    + z_2v_{1,\widehat{k},i}) -  \sum_{j=1}^r\sum_{l=1}^s\psi\kappa(v_{1,\widehat{k},l})\partial_{b_ja_l}(S) \psi\kappa(x_{2,\widehat{k},i,j})\big)\Big)
    \circ\eta(a_i^*)} \notag \\
    &
    \Scale[0.82]{+ \sum_{j=1}^rz_1\eta(b_j^*)\circ\Big(\theta\widetilde{\rho}\big(\psi\kappa(w_{2,\widehat{k},j}
    +  w_{1,\widehat{k},j} u_{2,\widehat{k}}) 
    -  \sum_{p=1}^r\sum_{l=1}^s \psi\kappa(x_{1,\widehat{k},l,j})\partial_{b_pa_l}(S)\psi\kappa(w_{2,\widehat{k},p})\big)\Big)} \notag \\
    &
    +\sum_{i=1}^s\sum_{j=1}^r\eta(b_j^*)\circ\Big(\theta\widetilde{\rho}\big(\psi\kappa(z_1 x_{2,\widehat{k},i,j}
    + w_{1,\widehat{k},j}v_{2,\widehat{k},i}
    + z_2 x_{1,\widehat{k},i,j}) \notag \\ 
    & 
    - \sum_{l=1}^s\sum_{p=1}^r \psi\kappa(x_{1,\widehat{k},l,j})\partial_{b_pa_l}(S)\psi\kappa(x_{2,\widehat{k},i,p})\big)\Big)\circ\eta(a_i^*) \qquad \text{and}\notag
\end{align} 
\begin{align}\label{eq:tilderho(u1)tilderho(u2)}
    &\widetilde{\eta}(u_1)\widetilde{\eta}(u_2) = \\
    &=
z_1z_2\eta(e_k)
+\theta\widetilde{\rho}\psi\kappa(u_{1,\widehat{k}})\circ\theta\widetilde{\rho}\psi\kappa(u_{2,\widehat{k}}) \notag \\
%\nonumber
&
+ \sum_{j=1}^r\sum_{l=1}^s\theta\widetilde{\rho}\psi\kappa(v_{1,\widehat{k},l})\circ\eta(a_l^*)\circ\eta(b_j^*)\circ\theta\widetilde{\rho}\psi\kappa(w_{2,\widehat{k},j})\notag \\
%\nonumber
&
+ \sum_{i=1}^s\Big(\theta\widetilde{\rho}\psi\kappa(u_{1,\widehat{k}})\circ\theta\widetilde{\rho}\psi\kappa(v_{2,\widehat{k},i})
+  \sum_{i=1}^sz_2\theta\widetilde{\rho}\psi\kappa(v_{1,\widehat{k},l})\notag \\
%\nonumber
&
+\sum_{i=1}^s\sum_{j=1}^r\sum_{l=1}^s\theta\widetilde{\rho}\psi\kappa(v_{1,\widehat{k},l})\circ\eta(a_l^*)\circ\eta(b_j^*)\circ \theta\widetilde{\rho}\psi\kappa(x_{2,\widehat{k},i,j})\Big)\circ\eta(a_i^*)\notag \\
%\nonumber
&
+ \sum_{j=1}^r\eta(b_j^*)\circ\Big(z_1 \theta\widetilde{\rho}\psi\kappa(w_{2,\widehat{k},j})
+ \sum_{j=1}^r\theta\widetilde{\rho}\psi\kappa(w_{1,\widehat{k},j})\circ\theta\widetilde{\rho}\psi\kappa(u_{2,\widehat{k}}) \notag \\
%\nonumber
&
+ \sum_{j=1}^r\sum_{l=1}^s\sum_{p=1}^r\theta\widetilde{\rho}\psi\kappa(x_{1,\widehat{k},l,j})\circ\eta(a_l^*)\circ\eta(b_p^*)\circ\theta\widetilde{\rho}\psi\kappa(w_{2,\widehat{k},p})\Big)\notag \\
%\nonumber
&
+\sum_{i=1}^s\sum_{j=1}^r\eta(b_j^*)\Big(z_1 \theta\widetilde{\rho}\psi\kappa(x_{2,\widehat{k},i,j}) \notag \\
%\nonumber
&
+ \sum_{i=1}^s\sum_{j=1}^r\theta\widetilde{\rho}\psi\kappa(w_{1,\widehat{k},j})\circ\theta\widetilde{\rho}\psi\kappa(v_{2,\widehat{k},i}) 
+  \sum_{l=1}^s\sum_{j=1}^rz_2\theta\widetilde{\rho}\psi\kappa(x_{1,\widehat{k},l,j}) \notag \\
%\nonumber
&
\Scale[0.9]{+\sum_{i=1}^s\sum_{j=1}^r\sum_{l=1}^s\sum_{p=1}^r\theta\widetilde{\rho}\psi\kappa(x_{1,\widehat{k},l,j})\circ\eta(a_l^*)\circ\eta(b_p^*)\circ \theta\widetilde{\rho}\psi\kappa(x_{2,\widehat{k},i,p})\Big)\circ\eta(a_i^*)}. \notag
\end{align}

Now, by Lemma \eqref{lemma:overlinebeta-overlinealpha=-gamma}, 
\[
\eta(a_l^*)\circ\eta(b_p^*)
=-\theta\widetilde{\rho}(\partial_{b_pa_l}(S))= \theta\widetilde{\rho}\psi\kappa(a^*b^*).
\]
This, and the fact that each of $\theta$, $\widetilde{\rho}$, $\psi$ and $\kappa$ is a $\field$-linear ring homomorphism implies that the right hand side of \eqref{eq:tilderho(u1u2)} is equal to the right hand side of \eqref{eq:tilderho(u1)tilderho(u2)}. Thus, $\widetilde{\eta}$ is ring homomorphism. Together with the commutativity of the leftmost triangle in \eqref{eq:diagram-overlineM-is-indeed-a-module-over-compalgwidetildemukQ}, this implies 
that $\widetilde{\eta}$ is a $\field^{Q_0}$-ring homomorphism as well.

The asserted uniqueness of $\widetilde{\eta}$ follows from the definitions of $\eta$, $\theta$, $\widetilde{\rho}$, $\psi$ and $\kappa$, combined with the fact that by \eqref{eq:expression-of-element-of-tilde-mu-k(Q)}, every element $u$ of $\compalg{\widetilde{\mu}_k(Q)}$ can be written as a finite sum of products of elements of $\pathalg{\widetilde{\mu}_k(Q)}\cup e_{\widehat{k}}\compalg{\widetilde{\mu}_k(Q)}e_{\widehat{k}}$. 

Take any path $c$ on $\widetilde{\mu}_k(Q)$ not starting nor ending in $k$, let $\ell$ be its length. Then $\kappa(c)$ is a path of length at least $2\ell/3$ on $e_{\widehat{k}}\widetilde{\mu}_k^2(Q)e_{\widehat{k}}$ because there are no $2$-cycles incident to $k$. Hence all the paths appearing in $\psi\kappa(c)\in e_{\widehat{k}}\compalg{Q}e_{\widehat{k}}$ have length at least $2\ell/3$.  From this and \eqref{eq:action-of-arbitrary-u-on-overlineM} we deduce that if $M$ is locally nilpotent (resp. nilpotent), then so is $\overline{M}$.
If $\mathcal{M}$ is finite-dimensional, then $(\overline{M},\overline{V})$ is obviously finite-dimensional.
\end{proof}

\begin{remark}
    If $\mathcal{M}=(M,V)$ is finite-dimensional, then the $\compalg{\widetilde{\mu}_k(Q)}$-module structure of $\widetilde{M}$ is the same as the one in \cite[paragraph preceding Proposition~10.7]{derksen2008quivers}.
\end{remark}

Recall that the Jacobian ideal $J(\widetilde{\mu}_k(S))$ is the topological closure of the two-sided ideal that the set of cyclic derivatives $\{\partial_{c}(\widetilde{\mu}_k(S))\suchthat c\in \widetilde{\mu}_k(Q)_1\}$ generates in the complete path algebra $\compalg{\widetilde{\mu}_k(Q)}$. Now that we know that $\overline{M}$ is a module over the complete path algebra $\compalg{\widetilde{\mu}_k(Q)}$, we shall prove that it is annihilated by $J(\widetilde{\mu}_k(S))$. Unlike the case when $M$ is assumed to be finite-dimensional (see \cite[Proposition~10.7]{derksen2008quivers}), the road will be long. The main reason why the finite-dimensional case (and in fact, the locally nilpotent case) is a lot simpler, is that for a locally nilpotent module $N$, if one manages to prove that the cyclic derivatives of the potential annihilate it, then an easy argument using the local nilpotency of $N$ shows that the whole Jacobian ideal annihilates $N$, but such easy argument cannot be applied when $N$ is not locally nilpotent.

In what follows, for $u\in\compalg{\widetilde{\mu}_k(Q)}$ we will use the alternative notations
\[
\overline{M}_u \qquad \text{and} \qquad u_{\overline{M}}
\]
 indistinctly to denote the same $\field$-linear endomorphism $\widetilde{\eta}(u)\in\End_{\field}(\overline{M})$.

\begin{lemma}\label{lemma:cyclic-derivatives-annihilate-overlineM}
    For every arrow $c\in \widetilde{\mu}_k(Q)_1$ we have $\overline{M}_{\partial_{c}(\widetilde{\mu}_k(S))}=0$.
\end{lemma}

\begin{proof}
    The proof of \cite[Proposition 10.7]{derksen2008quivers} applies practically without change. 
\end{proof}

It follows from Lemma \ref{lemma:cyclic-derivatives-annihilate-overlineM} that the two-sided ideal of $\compalg{\widetilde{\mu}_k(Q)}$ generated by the cyclic derivatives of $\widetilde{\mu}_k(Q)$ annihilates $\overline{M}$. The proof that also the $\idealM$-adic topological closure of this two-sided ideal annihilates $\overline{M}$ requires a lengthy preparation, for which we fix an auxiliary QP $(Q',S')$ with $Q'$ loop-free, but not necessarily $2$-acyclic, a vertex $k\in Q'_0$ not incident to any $2$-cycle of $Q'$, and an arbitrary vertex $\ell\in Q'_0$, maybe different from~$k$, maybe not. The reason for our choice of distinct notation for $(Q',S')$ is the following: the results obtained for $(Q',S')$, namely Proposition \ref{prop:spanning ker a/ker gamma for radical} and Corollary \ref{coro:most-basic-properties-of-alpha-beta-gamma-maps-from-projective} , will then be applied to $(Q',S')=(\widetilde{\mu}_k(Q),\widetilde{\mu}_k(S))$ and $(Q',S')=(\widetilde{\mu}_k(Q)^{\operatorname{op}},\widetilde{\mu}_k(S)^{\operatorname{op}})$ during the proof of Proposition \ref{prop:Jacobian-ideal-J(widetildemuk(S))-annihilates-overlineM}.

Suppose that the arrows of $Q'$ ending (resp. starting) at $k$ are $a_1,\ldots,a_s$ (resp. $b_1,\ldots,b_r$), and that the arrows from $\ell$ to $k$ (resp. from $k$ to $\ell$) are precisely $a_1,\ldots,a_{\overline{s}}$ (resp. $b_1,\ldots,b_{\overline{r}}$). Note that $\overline{s}\cdot \overline{r}=0$, that is, $\overline{s}$ and $\overline{r}$ cannot simultaneously be non-zero, since $Q'$ does not have $2$-cycles incident to~$k$; also, if $\ell=k$, then $\overline{s}=0=\overline{r}$ since $Q'$ does not have loops; finally, if $Q'$ does not have arrows from $\ell$ to $k$ (resp. from $k$ to $\ell$), then $\overline{s}=0$ (resp. $\overline{r}=0$).

Let us describe the radical $\rad_{(Q',S')}P_{(Q',S')}(\ell)$. Set $\idealM$ to be the two-sided ideal of $\compalg{Q'}$ generated by the arrows of $Q'$, and let $\pi:\compalg{Q'}\rightarrow\jacobalg{Q',S'}$ denote the projection from the complete path algebra to the Jacobian algebra. Then $\rad_{(Q',S')}P_{(Q',S')}(\ell)=\pi(\idealM)e_{\ell}$, regardless of whether $P_{(Q',S')}(\ell)$ is infinite-dimensional over the ground field $\field$ or not. Furthermore, as a representation of $Q'$, to a vertex $j$ it attaches the space $e_j\pi(\idealM)e_{\ell}$, and for an arrow $a:j\rightarrow i$, the corresponding linear map $e_j\pi(\idealM)e_{\ell}\rightarrow e_i\pi(\idealM)e_{\ell}$ is induced by merely composing with $a$ each path from $\ell$ to $j$, i.e., $p\mapsto a\cdot p$.

Consider Derksen-Weyman-Zelevinsky's triangle of linear maps arising at $k$ from the radical $\rad_{(Q',S')}P_{(Q',S')}(\ell)$, namely:
\begin{equation}\label{eq:alpha-beta-gamma-maps-for-radical-of-projective}
\xymatrix{
 & e_k\pi(\idealM)e_\ell  \ar[dr]^{\mathbf{b}_k:=\left[\begin{array}{c}b_1\cdot \\ \vdots \\ b_r\cdot \end{array}\right]} & \\
 \bigoplus_{i=1}^s e_{t(a_i)} \pi(\idealM) e_\ell \ar[ur]^{\mathbf{a}_k:=\left[\begin{array}{ccc}a_1\cdot  & \cdots & a_s\cdot \end{array}\right]\qquad} & & \bigoplus_{j=1}^r e_{h(b_j)} \pi(\idealM)e_\ell \ar@/^0.55pc/[ll]^{\mathbf{c}_k:=\text{ {\tiny $\left[\begin{array}{ccc}
 \partial_{b_1a_1}(S') \cdot & \cdots & \partial_{b_r a_1}(S')\cdot \\
 \vdots & \ddots & \vdots\\
 \partial_{b_1a_s}(S') \cdot & \cdots & \partial_{b_ra_s}(S')\cdot
 \end{array}\right]$}}}
 }
\end{equation}
By \eqref{eq:alphagamma=0=gammabeta}, $\image\mathbf{c}_k\subseteq\ker\mathbf{a}_k$.
The next proposition specifies an explicit spanning set for the vector space $\ker\mathbf{a}_k/\image\mathbf{c}_k$. The result was stated in \cite[Proposition 3.2]{labardini2022landau} under the assumptions that $\dim_{\field}(P_{(Q',S')}(\ell))<\infty$ and that $k\neq \ell$. However, the proof in \emph{loc. cit.} works \emph{as is} also when $\dim_{\field}(P_{(Q',S')}(\ell))$ is infinite or when $k=\ell$. For the convenience of the reader we include the proof below, with the difference that an effort has been made to present it in a less dense way~here.

\begin{prop}\label{prop:spanning ker a/ker gamma for radical} Each of the elements
\[
\bfv_1:=\left(\begin{array}{c}\partial_{b_1a_1}(S')+e_{t(a_1)}J(S')e_{\ell} \\ \vdots \\ \partial_{b_1a_s}(S')+e_{t(a_s)}J(S')e_{\ell}\end{array}\right),\ldots,\bfv_{\overline{r}}:=\left(\begin{array}{c}\partial_{b_{\bar{r}}a_1}(S')+e_{t(a_1)}J(S')e_{\ell}\\ \vdots\\ \partial_{b_{\bar{r}}a_s}(S')+e_{t(a_s)}J(S')e_{\ell}\end{array}\right)
\]
of $\bigoplus_{i=1}^s e_{t(a_i)}\pi(\idealM) e_\ell$
belongs to $\ker\mathbf{a}_k$. Their projections 
$
\bfv_1+\image\mathbf{c}_k,\ldots,\bfv_{\overline{r}}+\image\mathbf{c}_k
$ 
span the quotient
$\ker\mathbf{a}_k/\image\mathbf{c}_k$ as a $\field$-vector space. In particular, $\dim_{\field}(\ker\mathbf{a}_k/\image\mathbf{c}_k)$ is finite.
\end{prop}

\begin{proof}
It is obvious that said elements
$\bfv_1,\ldots,\bfv_{\overline{r}}$ belong to $\ker\mathbf{a}_k$. 
Take any element $(p_1,\ldots,p_s)\in\bigoplus_{a\in Q'_1:h(a)=k} e_{t(a)} \idealM e_\ell$ such that $a_1p_1+\cdots+a_sp_s\in J(S')$, so it is the $\idealM$-adic limit of some sequence $\left(\xi_n\right)_{n\in\bbZ_{>0}}$ of elements
\[
\xi_n
=
\sum_{\omega\in Q'_1}\sum_{q}\lambda_{n,\omega,q}\partial_{\omega}(S')\rho_{n,\omega,q}
\]
of $\compalg{Q'}$,
with $\lambda_{n,\omega,q}\in e_k\compalg{Q'}e_{t(\omega)}$ and $\rho_{n,\omega,q}\in e_{h(\omega)}\compalg{Q'}e_{\ell}$, where each sum $\sum_{q}\lambda_{n,\omega,q}\partial_{\omega}(S')\rho_{n,\omega,q}$ has only finitely many summands.
For each $n$, each $\omega\in Q'_1$ and each $q$, write
\[
\lambda_{n,\omega,q}=\lambda_{n,\omega,q}^{(0)}+ \lambda_{n,\omega,q}^{(>0)}
\]
with $\lambda_{n,\omega,q}^{(0)}=x_{n,\omega,q}e_k\in\mathbb{C}e_k$ the degree-$0$ component of $\lambda_{n,\omega,q}$ in the complete path algebra $\compalg{Q'}$, and $\lambda_{n,\omega,q}^{(>0)}:=\lambda_{n,\omega,q}-\lambda_{n,\omega,q}^{(0)}\in e_k\idealM e_{t(\omega)}$. This way, 
\[
\Scale[0.9]{\xi_n=
\sum_{\omega\in Q'_1}\sum_q\lambda_{n,\omega,q}\partial_{\omega}(S')\rho_{n,\omega,q}=
\sum_{\omega\in Q'_1}\sum_q\left(\lambda_{n,\omega,q}^{(0)}\partial_{\omega}(S')\rho_{n,\omega,q}+
\lambda_{n,\omega,q}^{(>0)}\partial_{\omega}(S')\rho_{n,\omega,q}\right)}.
\]
Notice that for every $\omega\in Q'_1$, the element $\lambda_{n,\omega,q}^{(>0)}$ can be written as $\lambda_{n,\omega,q}^{(>0)}=\sum_{i=1}^s a_i\nu_{n,\omega,q,i}$ with $\nu_{n,\omega,q,i}\in e_{t(a_i)}\compalg{Q'}e_{t(\omega)}$. 

For each term of the form $\lambda_{n,\omega,q}^{(0)}\partial_{\omega}(S')\rho_{n,\omega,q}$ with $\lambda_{n,\omega,q}^{(0)}\neq 0$
we necessarily have $k=t(\omega)$, i.e. $\omega\in\{b_1,\ldots,b_r\}$. Another way of stating this is that for every $\omega\in Q'_1\setminus\{b_1,\ldots,b_r\}$, we have $t(\omega)\neq k$, thus $\lambda_{n,\omega,q}=\lambda_{n,\omega,q}^{(>0)}$. For $\omega=b_j$, $j=1,\ldots,r$, we can write
\[
\lambda_{n,\omega,q}^{(0)}\partial_{\omega}(S')\rho_{n,\omega,q}=
\lambda_{n,b_j,q}^{(0)}\partial_{b_j}(S')\rho_{n,b_j,q}=\lambda_{n,b_j,q}^{(0)}\sum_{i=1}^sa_i\partial_{b_ja_i}(S')\rho_{n,b_j,q}.
\]

On the other hand, for $j=\overline{r}+1,\ldots,r$ we have $h(b_j)\neq \ell$. Since $\rho_{n,\omega,q}\in e_{h(\omega)}\compalg{Q'}e_{\ell}$, we deduce that for $j=\overline{r}+1,\ldots,r$, the degree-$0$ component of $\rho_{n,b_j,q}$ in the complete path algebra $\compalg{Q'}$ is necessarily $0$. 
Thus, if we denote by $\rho_{n,b_j,q}^{(0)}:=y_{n,b_j,q}e_{h(b_j)}\in\field e_{h(b_j)}$ the degree-$0$ component of $\rho_{n,b_j,q}$ in the complete path algebra $\compalg{Q'}$ for every $j=1,\ldots,r$, and if we write $\rho_{n,b_j,q}^{(>0)}:=\rho_{n,b_j,q}-\rho_{n,b_j,q}^{(0)}$, then for every $j=\overline{r}+1,\ldots,r$ we have $\rho_{n,b_j,q}^{(0)}=0$.

Therefore, {%\small
\begin{align*}
\xi_n & \Scale[0.65]{=
\sum_{j=1}^{r}\sum_q\left(\lambda^{(0)}_{n,b_j,q}\sum_{i=1}^sa_i\partial_{b_ja_i}(S')\rho_{n,b_j,q}+\sum_{i=1}^s a_i\nu_{n,b_j,q,i}\partial_{b_j}(S')\rho_{n,b_j,q}\right)
%+\\
%&\qquad 
+
\sum_{\omega\in Q'_1\setminus\{b_1,\ldots,b_r\}}\sum_q\left(\sum_{i=1}^s a_i\nu_{n,\omega,q,i}\partial_{\omega}(S')\rho_{n,\omega,q}\right)}\\
&
\Scale[0.725]{=
\sum_{i=1}^sa_i\left(\sum_{j=1}^{r}\sum_q\left(x_{n,b_j,q}\partial_{b_ja_i}(S')\rho_{n,b_j,q}+\nu_{n,b_j,q,i}\partial_{b_j}(S')\rho_{n,b_j,q}\right) +
%\right.
%\\
%&\qquad \left.+
\sum_{\omega\in Q'_1\setminus\{b_1,\ldots,b_r\}}\sum_q\left(\nu_{n,\omega,q,i}\partial_{\omega}(S')\rho_{n,\omega,q}\right)\right)}
\\
&
{=
\sum_{i=1}^sa_i\left(\sum_{j=1}^{r}\sum_q\left(x_{n,b_j,q}\partial_{b_ja_i}(S')\rho_{n,b_j,q}\right)
+
\sum_{\omega\in Q'_1}\sum_q\left(\nu_{n,\omega,q,i}\partial_{\omega}(S')\rho_{n,\omega,q}\right)\right)
}\\
&\Scale[0.88]{=
\sum_{i=1}^sa_i\left(\sum_{j=1}^{r}\sum_q\left(x_{n,b_j,q}\partial_{b_ja_i}(S')(\rho_{n,b_j,q}^{(0)}+\rho_{n,b_j,q}^{(>0)})\right)
+
\sum_{\omega\in Q'_1}\sum_q\left(\nu_{n,\omega,q,i}\partial_{\omega}(S')\rho_{n,\omega,q}\right)\right)
}\\
&\Scale[0.775]{=
\sum_{i=1}^sa_i\left(\sum_{j=1}^{r}\sum_q\left(x_{n,b_j,q}\partial_{b_ja_i}(S')\rho_{n,b_j,q}^{(0)}+x_{n,b_j,q}\partial_{b_ja_i}(S')\rho_{n,b_j,q}^{(>0)}\right) +
%+\right.
%\\
%&\qquad +\left.
\sum_{\omega\in Q'_1}\sum_q\left(\nu_{n,\omega,q,i}\partial_{\omega}(S')\rho_{n,\omega,q}\right)\right)
}\\
&\Scale[0.69]{=
\sum_{i=1}^sa_i\left(\sum_{j=1}^{r}\sum_q\left(x_{n,b_j,q}\partial_{b_ja_i}(S')\rho_{n,b_j,q}^{(0)}\right)+\sum_{j=1}^{r}\sum_q\left(x_{n,b_j,q}\partial_{b_ja_i}(S')\rho_{n,b_j,q}^{(>0)}\right) + 
%+\right.
%\\
%&\qquad +\left.
\sum_{\omega\in Q'_1}\sum_q\left(\nu_{n,\omega,q,i}\partial_{\omega}(S')\rho_{n,\omega,q}\right)\right)
}\\
&\Scale[0.69]{=
\sum_{i=1}^sa_i\left(\sum_{j=1}^{\overline{r}}\sum_q\left(x_{n,b_j,q}y_{n,b_j,q}\partial_{b_ja_i}(S')\right)+\sum_{j=1}^{r}\sum_q\left(x_{n,b_j,q}\partial_{b_ja_i}(S')\rho_{n,b_j,q}^{(>0)}\right) 
%+\right.
%\\
%&\qquad \left.
+
\sum_{\omega\in Q'_1}\sum_q\left(\nu_{n,\omega,q,i}\partial_{\omega}(S')\rho_{n,\omega,q}\right)\right)
}.
\end{align*}}

Take any positive integer $\varepsilon$. Then there exists $N\in\bbZ_{>0}$ such that for all $n>N$ we have
$a_1p_1+\cdots+a_sp_s-\xi_n\in\mathfrak{m}^\varepsilon$,
and hence, 
\begin{align*}\Scale[0.9]{
\left(\begin{array}{c}p_1\\ 
\vdots \\ p_s\end{array}\right)}
&
\Scale[0.9]{- 
\sum_{j=1}^{\overline{r}}\sum_qx_{n,b_j,q}y_{n,b_j,q}\left(\begin{array}{c}\partial_{b_ja_1}(S')\\
\vdots \\
\partial_{b_ja_s}(S')
\end{array}\right)
-
\sum_{j=1}^{r}\sum_qx_{n,b_j,q}\left(\begin{array}{c}\partial_{b_ja_1}(S')\rho_{n,b_j,q}^{(>0)}\\
\vdots\\
\partial_{b_ja_s}(S')\rho_{n,b_j,q}^{(>0)}\end{array}\right)+}\\
&\ 
\Scale[0.9]{-
\sum_{\omega\in Q'_1}\sum_q\left(\begin{array}{c}\nu_{n,\omega,q,1}\partial_{\omega}(S')\rho_{n,\omega,q}\\
\vdots\\
\nu_{n,\omega,q,s}\partial_{\omega}(S')\rho_{n,\omega,q}\end{array}\right) }
\in \underset{s}{\underbrace{\idealM^{\varepsilon-1}\times \cdots\times \idealM^{\varepsilon-1}}},
\end{align*} 
which means that in $\bigoplus_{a\in Q'_1:h(a)=k} e_{t(a)} \idealM e_\ell$, we have
\begin{align}\label{eq:seq-converges-to-vector-of-ps}
&\lim_{n\to\infty}
\Scale[0.85]{\left[\sum_{j=1}^{\overline{r}}\sum_qx_{n,b_j,q}y_{n,b_j,q}\left(\begin{array}{c}\partial_{b_ja_1}(S')\\
\vdots \\
\partial_{b_ja_s}(S')
\end{array}\right)
+
\sum_{j=1}^{r}\sum_qx_{n,b_j,q}\left(\begin{array}{c}\partial_{b_ja_1}(S')\rho_{n,b_j,q}^{(>0)}\\
\vdots\\
\partial_{b_ja_s}(S')\rho_{n,b_j,q}^{(>0)}\end{array}\right)+ \right.
}\\
%\nonumber
&\qquad 
\Scale[0.85]{\left.
+
\sum_{\omega\in Q'_1}\sum_q\left(\begin{array}{c}\nu_{n,\omega,q,1}\partial_{\omega}(S')\rho_{n,\omega,q}\\
\vdots\\
\nu_{n,\omega,q,s}\partial_{\omega}(S')\rho_{n,\omega,q}\end{array}\right) \right]
= \left(\begin{array}{c}p_1\\ 
\vdots \\ p_s\end{array}\right)\in \bigoplus_{a\in Q'_1:h(a)=k} e_{t(a)} \idealM e_\ell
}.\notag
\end{align}

Now, left multiplication by the matrix
\begin{equation}\label{eq:gamma-for-Jac-alg}
\left[\begin{array}{ccc}
 \partial_{b_1a_1}(S') \cdot & \cdots & \partial_{b_r a_1}(S')\cdot \\
 \vdots & \ddots & \vdots\\
 \partial_{b_1a_s}(S') \cdot & \cdots & \partial_{b_ra_s}(S')\cdot
 \end{array}\right]:\bigoplus_{b\in Q'_1:t(b)=k} e_{h(b)} \idealM e_\ell\rightarrow \bigoplus_{a\in Q'_1:h(a)=k} e_{t(a)} \idealM e_\ell
\end{equation}
is a continuous $\field$-linear map, so by \cite[Lemma 13.6]{derksen2008quivers}, its image is a closed $\field$-vector subspace of $\bigoplus_{a\in Q'_1:h(a)=k} e_{t(a)} \idealM e_\ell$. On the other hand, $\bigoplus_{a\in Q'_1:h(a)=k} e_{t(a)} J(S') e_\ell$ is a closed $\field$-vector subspace of $\bigoplus_{a\in Q'_1:h(a)=k} e_{t(a)} \idealM e_\ell$, whereas the $\field$-vector subspace of $\bigoplus_{a\in Q'_1:h(a)=k} e_{t(a)} \idealM e_\ell$ spanned by the set
\begin{equation}\label{eq:set-of-2nd-order-cyclic-derivatives}
\left\{\left(\begin{array}{c}\partial_{b_1a_1}(S')\\
\vdots \\
\partial_{b_1a_s}(S')
\end{array}\right),\ldots,\left(\begin{array}{c}\partial_{b_{\overline{r}}a_1}(S')\\
\vdots \\
\partial_{b_{\overline{r}}a_s}(S')
\end{array}\right)\right\}
\end{equation}
is closed by \cite[Lemma 13.2]{derksen2008quivers}.

The $\field$-vector space $\bigoplus_{a\in Q'_1:h(a)=k} e_{t(a)} \idealM e_\ell$ contains $\bigoplus_{a\in Q'_1:h(a)=k} e_{t(a)} J(S') e_\ell$, the image of the map \eqref{eq:gamma-for-Jac-alg}, and the set \eqref{eq:set-of-2nd-order-cyclic-derivatives}. By the preceding paragraph and \cite[Lemma 13.4]{derksen2008quivers}, the union of these three subsets spans a closed $\field$-vector subspace of $\bigoplus_{a\in Q'_1:h(a)=k} e_{t(a)} \idealM e_\ell$. Being closed, this subspace contains the limit vector~\eqref{eq:seq-converges-to-vector-of-ps}. This finishes the proof of Proposition~\ref{prop:spanning ker a/ker gamma for radical}.
\end{proof}

\begin{coro}
    \label{coro:most-basic-properties-of-alpha-beta-gamma-maps-from-projective}  The $\field$-linear maps from \eqref{eq:alpha-beta-gamma-maps-for-radical-of-projective} and the $\field$-linear maps 
\[
\xymatrix{
 & e_k\jacobalg{Q',S'}e_\ell  \ar[dr]^{\widehat{\mathbf{b}}:=\left[\begin{array}{c}b_1\cdot \\ \vdots \\ b_r\cdot \end{array}\right]} & \\
 \underset{a\in Q'_1:h(a)=k}{\bigoplus} e_{t(a)} \jacobalg{Q',S'} e_\ell \ar[ur]^{\widehat{\mathbf{a}}:=\left[\begin{array}{ccc}a_1\cdot  & \cdots & a_s\cdot \end{array}\right]\qquad} & & \underset{b\in Q'_1:t(b)=k}{\bigoplus} e_{h(b)} \jacobalg{Q',S'}e_\ell \ar@/^0.55pc/[ll]^{\widehat{\mathbf{c}}:=\text{ {\tiny $\left[\begin{array}{ccc}
 \partial_{b_1a_1}(S') \cdot & \cdots & \partial_{b_r a_1}(S')\cdot \\
 \vdots & \ddots & \vdots\\
 \partial_{b_1a_s}(S') \cdot & \cdots & \partial_{b_ra_s}(S')\cdot
 \end{array}\right]$}}}
 }
\]
satisfy
\[
\ker\widehat{\mathbf{a}}=\ker\mathbf{a}=\image\widehat{\mathbf{c}}, \quad
\image\widehat{\mathbf{b}}=\field^{\delta_{k,\ell}}\oplus \image\mathbf{b} \quad \text{and} \quad \field^{\delta_{k,\ell}}\oplus\coker\widehat{\mathbf{b}}=(\field e_{\ell})^{\overline{r}}\oplus\coker\mathbf{b},
\]
where $\delta_{k,\ell}$ is Kronecker's delta between $k$ and $\ell$.
\end{coro}

\begin{proof}
Notice that
$e_k\jacobalg{Q',S'}e_\ell =(\field e_k)^{\delta_{k,\ell}}\oplus e_k\pi(\idealM)e_\ell$. Similarly,
\begin{align*}
    \text{for} \ i& =1,\ldots,s: &
    e_{t(a_i)}\jacobalg{Q',S'}e_\ell&=(\field e_{\ell})^{\delta_{t(a_i),\ell}}\oplus e_{t(a_i)}\pi(\idealM)e_\ell,\\
    \text{and for} \ j&=1,\ldots,r: &
    e_{h(b_j)}\jacobalg{Q',S'}e_\ell&=(\field e_{\ell})^{\delta_{h(b_j),\ell}}\oplus e_{h(b_j)}\pi(\idealM)e_\ell,
\end{align*}
from which we deduce that
\begin{equation}\begin{aligned}\label{eq:space-of-paths-in-jacobalg-in-terms-of-arrow-ideal-incoming-arrows-and-outgoing-arrows}
\bigoplus_{a\in Q'_1:h(a)=k} e_{t(a)} \jacobalg{Q',S'} e_\ell &= 
(\field e_{\ell})^{\overline{s}}\oplus \bigoplus_{i=1}^s e_{t(a_i)}\pi(\idealM)e_{\ell},\\
\underset{b\in Q'_1:t(b)=k}{\bigoplus} e_{h(b)} \jacobalg{Q',S'}e_\ell &= (\field e_\ell)^{\overline{r}}\oplus\bigoplus_{j=1}^re_{h(b_j)}\pi(\idealM)e_{\ell}.
\end{aligned}
\end{equation}

We obviously have $\ker\mathbf{a}\subseteq\ker\widehat{\mathbf{a}}$. Take $u\in \ker\widehat{\mathbf{a}}$. Then $\widehat{\mathbf{a}}(u)=0$ in the Jacobian algebra $\jacobalg{Q',S'}$, that is, $\widehat{\mathbf{a}}(u)\in e_kJ(S')e_\ell\subseteq \compalg{Q'}$. Use \eqref{eq:space-of-paths-in-jacobalg-in-terms-of-arrow-ideal-incoming-arrows-and-outgoing-arrows} to write $u=\sum_{i=1}^{\overline{s}}x_ie_\ell+\sum_{i=1}^{s}p_i$ with $x_1,\ldots,x_{\overline{s}}\in\field$ and $p_i\in e_{t(a_i)}\pi(\idealM)e_{\ell}$. Then 
\[
\widehat{\mathbf{a}}(u)=\sum_{i=1}^{\overline{s}}x_i a_i+\sum_{i=1}^{s}a_ip_i.
\]
Since $Q'$ does not have $2$-cycles incident to $k$, the expression of any element of $e_kJ(S')e_\ell$ as a possibly infinite linear combination of pahts involves only paths of length at least $2+\delta_{k,\ell}\geq 2$. 
This implies $x_1=\cdots=x_{\overline{s}}=0$, so $u\in\ker(\mathbf{a})$. Therefore, $\ker\widehat{\mathbf{a}}=\ker\mathbf{a}$. 

By \eqref{eq:alphagamma=0=gammabeta}, $\image\widehat{\mathbf{c}}\subseteq\ker\widehat{\mathbf{a}}$. On the other hand, from Proposition \ref{prop:spanning ker a/ker gamma for radical} one can easily deduce that $\ker\mathbf{a}\subseteq\image\widehat{\mathbf{c}}$.

From $e_k\jacobalg{Q',S'}e_\ell =(\field e_k)^{\delta_{k,\ell}}\oplus e_k\pi(\idealM)e_\ell$ and the fact that no $2$-cycle of $Q'$ is incident to $k$, it follows that
$
\left(\begin{array}{c}b_1\\ \vdots\\ b_r\end{array}\right)\in\image\widehat{\mathbf{b}}\setminus \image\mathbf{b}$. Hence
$\image\widehat{\mathbf{b}}=(\field\left(\begin{array}{c}b_1\\ \vdots\\ b_r\end{array}\right))^{\delta_{k,\ell}}\oplus \image\mathbf{b},
$
as claimed.

To prove that $\field^{\delta_{k,\ell}}\oplus\coker\widehat{\mathbf{b}}=(\field e_{\ell})^{\overline{r}}\oplus\coker\mathbf{b}$, we analyze two cases, namely, when $k=\ell$ and when $k\neq \ell$. If $k=\ell$, then $\overline{r}=0$ because $Q'$ does not have loops. One of the well-known isomorphism theorems from standard undergraduate algebra asserts that the rightmost non-zero column in the natural commutative diagram of $\field$-linear maps
\[
\xymatrix{ & 0 \ar[d] & 0 \ar[d] & 0 \ar[d] & \\
0 \ar[r] & \image\mathbf{b} \ar@{^{(}->}[d] \ar[r]^{\myid} & \image\mathbf{b} \ar@{^{(}->}[d] \ar[r] & 0 \ar[d] \ar[r] & 0 \\
0 \ar[r] & \image\widehat{\mathbf{b}} \ar[d] \ar@{^{(}->}[r] & M_{\operatorname{out}} \ar[d] \ar[r] & \frac{M_{\operatorname{out}}}{\image\widehat{\mathbf{b}}}  \ar[d] \ar[r] & 0 \\
0 \ar[r] & \image\widehat{\mathbf{b}}/\image\mathbf{b} \ar[d] \ar[r] & M_{\operatorname{out}}/\image\mathbf{b} \ar[d] \ar[r] &  \frac{M_{\operatorname{out}}/\image\mathbf{b}}{\image\widehat{\mathbf{b}}/\image\mathbf{b}} \ar[d] \ar[r] & 0
\\
 & 0 & 0 & 0 &
}
\]
is exact, which implies 
\[
\coker\mathbf{b} 
= \frac{M_{\operatorname{out}}}{\image\mathbf{b}} \cong \frac{\image\widehat{\mathbf{b}}}{\image\mathbf{b}}\oplus \frac{M_{\operatorname{out}}}{\image\widehat{\mathbf{b}}} \cong (\field\left(\begin{array}{c}b_1\\ \vdots\\ b_r\end{array}\right)) \oplus\coker\widehat{\mathbf{b}}.
\]

If $k\neq \ell$, then
\[
\image\widehat{\mathbf{b}}=\image\mathbf{b}\subseteq \bigoplus_{j=1}^re_{h(b_j)}\pi(\idealM)e_{\ell}\subseteq (\field e_\ell)^{\overline{r}}\oplus\bigoplus_{j=1}^re_{h(b_j)}\pi(\idealM)e_{\ell}=\bigoplus_{j=1}^re_{h(b_j)}\jacobalg{Q',S'}e_{\ell},
\]
which easily yields 
$\coker\widehat{\mathbf{b}}=(\field e_{\ell})^{\overline{r}}\oplus\coker\mathbf{b}$. This finishes the proof of Corollary~\ref{coro:most-basic-properties-of-alpha-beta-gamma-maps-from-projective}.
\end{proof}

After such a long preparation, we are finally ready to prove:

\begin{prop}\label{prop:Jacobian-ideal-J(widetildemuk(S))-annihilates-overlineM}
    In the situation of Definition \ref{def:action-of-arb-u-on-overlineM} and Proposition \ref{prop:overlineM-is-left-compalg-module}, the Jacobian ideal $J(\widetilde{\mu}_k(S))$ annihilates $\overline{M}$.
\end{prop}

\begin{proof} 
    By Lemma \ref{lemma:cyclic-derivatives-annihilate-overlineM} we know that the two-sided ideal of $\compalg{\widetilde{\mu}_k(Q)}$ generated by the cyclic derivatives of $\widetilde{\mu}_k(Q)$ annihilates $\overline{M}$. We shall now show that the $\idealM$-adic topological closure of this two-sided ideal annihilates $\overline{M}$ too. For this, we shall apply Proposition \ref{prop:spanning ker a/ker gamma for radical} and Corollary \ref{coro:most-basic-properties-of-alpha-beta-gamma-maps-from-projective} to $(Q',S'):=(\widetilde{\mu}_k(Q),\widetilde{\mu}_k(S))$. In $Q':=\widetilde{\mu}_k(Q)$, the arrows going into $k$ are $b_1^*,\ldots,b_r^*$, whereas the arrows coming out from $k$ are $a_1^*,\ldots,a_s^*$; thus, the maps $\mathbf{a}_k$ and $\widehat{\mathbf{a}}_k$ are given in terms of $b_1^*,\ldots,b_r^*$, whereas the maps $\mathbf{b}_k$ and $\widehat{\mathbf{b}}_k$ are given in terms of $a_1^*,\ldots,a_s^*$. The maps $\mathbf{c}_k$ and $\widehat{\mathbf{c}}_k$ are given in terms of the second-order cyclic derivatives of $\widetilde{\mu}_k(S)$.

Looking at \eqref{eq:def-of-widetilde-mu-k(S)}, we see that $\partial_{a_i^*b_j^*}(\widetilde{\mu}_k(S))=[b_ja_i]$. Set
\begin{equation}\label{eq:matrix-of-partial-a*b*-of-tildemukS}
C:=\left[\begin{array}{ccc}
    \partial_{a^*_1b^*_1}(\widetilde{\mu}_k(S)) & \cdots & \partial_{a^*_sb^*_1}(\widetilde{\mu}_k(S)) \\
    \vdots & \ddots & \vdots \\
    \partial_{a^*_1b^*_r}(\widetilde{\mu}_k(S)) & \cdots & \partial_{a^*_sb^*_r}(\widetilde{\mu}_k(S))
\end{array}\right]=\left[
\begin{array}{ccc}
    \left[b_1 a_1\right] & \cdots & \left[b_1 a_s\right] \\
    \vdots & \ddots & \vdots \\
    \left[b_r a_1\right] & \cdots & \left[b_r a_s\right]
\end{array}
\right].
\end{equation}

    Take vertices $\ell_1$ and $\ell_2$ of $\widetilde{\mu}_k(Q)$, and an element $u\in e_{\ell_2}J(\widetilde{\mu}_k(S))e_{\ell_1}\subseteq\compalg{\widetilde{\mu}_k(Q)}$. 
    
    \setcounter{case}{0}
    \begin{case}
    If $k\notin\{\ell_1,\ell_2\}$, then $u\cdot \overline{M}=0$ follows from \cite[Proposition 6.1]{derksen2008quivers}. 
    \end{case}

\begin{case}\label{case:k=ell_2-in-proof-that-J(tildemuk(S))-annihilates-overlineM}
    Suppose $k=\ell_2$ and write $u=b_1^*p_1+\cdots+b_r^*p_r$ in $\compalg{\widetilde{\mu}_k(Q)}$ for some 
    \[
    \left(\begin{array}{c}p_1\\ \vdots \\ p_r\end{array}\right)\in\bigoplus_{j=1}^r e_{h(b_j)} \compalg{\widetilde{\mu}_k(Q)} e_{\ell_1}
    \]
    (recall that $h(b_j)$ is the head of $b_j$ in $Q$, that is, the tail of $b_j^*$ in $Q':=\widetilde{\mu}_k(Q)$). Since $u\in e_{k}J(\widetilde{\mu}_k(S))e_{\ell_1}$, the tuple 
    \begin{equation}\label{eq:tuple-of-p-cosets-in-proof-that-overlineM-is-annihilated-by-Jacobian-ideal}
    \left(\begin{array}{c}p_1+e_{h(b_1)}J(\widetilde{\mu}_k(S))e_\ell\\ \vdots \\ p_r+e_{h(b_r)}J(\widetilde{\mu}_k(S))e_\ell\end{array}\right)\in \bigoplus_{j=1}^r e_{h(b_j)} \jacobalg{\widetilde{\mu}_k(Q),\widetilde{\mu}_k(S)} e_{\ell_1}
    \end{equation}
    belongs to $\ker\widehat{\mathbf{a}}_k$, which is equal to $\image\widehat{\mathbf{c}}_k$ by Corollary \ref{coro:most-basic-properties-of-alpha-beta-gamma-maps-from-projective}, so there exists a tuple 
    \[
    \left(
    \begin{array}{c}
    q_1 \\ \vdots \\ q_s
    \end{array}
    \right)\in \bigoplus_{i=1}^s e_{t(a_i)} \compalg{\widetilde{\mu}_k(Q)} e_{\ell_1}
    \]
    such that the element
    \[
    \left(\begin{array}{c}q_1+ e_{t(a_1)} J(\widetilde{\mu}_k(S)) e_{\ell_1} \\ \vdots \\ q_s+ e_{t(a_1)} J(\widetilde{\mu}_k(S)) e_{\ell_1} \end{array}\right) \in \bigoplus_{i=1}^s e_{t(a_i)} \jacobalg{\widetilde{\mu}_k(Q)} e_{\ell_1}
    \]
    is mapped to the tuple \eqref{eq:tuple-of-p-cosets-in-proof-that-overlineM-is-annihilated-by-Jacobian-ideal} by $\widehat{\mathbf{c}}_k$. That is,
    \begin{equation}\label{eq:aux-element-of-Jac-ideal-in-proof-J(tildemuk(S))overlineM=0}
    \left(\begin{array}{c}
    p_1- \sum_{i=1}^s \partial_{a_i^*b_1^*}(\widetilde{\mu}_k(S))q_i \\
    \vdots \\
    p_r- \sum_{i=1}^s \partial_{a_i^*b_r^*}(\widetilde{\mu}_k(S))q_i
    \end{array}\right)\in \bigoplus_{j=1}^r e_{h(b_j)} J(\widetilde{\mu}_k(S)) e_{\ell_1}
    \end{equation}

\begin{lemma}\label{lemma:yet-another-element-must-annihilate-overlineM}
Under the ongoing assumption $k=\ell_2$,
for all $j=1,\ldots,r$, we have
\begin{equation}\label{eq:pj-minus-secondpartial-qi-annihilates-overlineM}
\left(p_j- \sum_{i=1}^s \partial_{a_i^*b_j^*}(\widetilde{\mu}_k(S))q_i\right)\cdot \overline{M}=0.
\end{equation}
\end{lemma}

\begin{proof}
The proof of Lemma \ref{lemma:yet-another-element-must-annihilate-overlineM} is given case by case, depending on whether $\ell_1\neq k$ or $\ell_1= k$. Notice that $k\notin \{h(b_1),\ldots,h(b_r)\}$ because $Q$ does not have loops.

If $\ell_1\neq k$, then as $\ell_2=k\neq h(b_j)$ too, \cite[Proposition 6.1]{derksen2008quivers} implies that 
the desired equality \eqref{eq:pj-minus-secondpartial-qi-annihilates-overlineM} for all $j=1,\ldots,r$. 

So, suppose $\ell_1= k$. Fix $j=1,\ldots,r$ and write
\begin{align*}
    p_j- \sum_{i=1}^s \partial_{a_i^*b_j^*}(\widetilde{\mu}_k(S))q_i 
    &=
    \left[
    \begin{array}{ccc}
    v_{j1} & \cdots    & v_{js} 
    \end{array}
    \right]
    \left[
    \begin{array}{c}
    a_1^*\\
    \vdots \\
    a_s^*
    \end{array}
    \right]\\
\text{with} \quad    
\left[
    \begin{array}{ccc}
    v_{j1} & \cdots    & v_{js} 
    \end{array}
    \right] &\in \bigoplus_{i=1}^{s} e_{h(b_j)}\widetilde{\idealM} e_{t(a_i)}
\end{align*}
where $\widetilde{\idealM}$ is the ideal of $\compalg{\widetilde{\mu}_k(Q)}$ generated by the arrows of $\widetilde{\mu}_k(Q)$ (note that $v_{ji}\in e_{h(b_j)}\widetilde{\idealM} e_{t(a_i)}$ because $Q$ does not have $2$-cycles incident to $k$). Because of \eqref{eq:aux-element-of-Jac-ideal-in-proof-J(tildemuk(S))overlineM=0}, we can apply Corollary \ref{coro:most-basic-properties-of-alpha-beta-gamma-maps-from-projective} to $(Q',S'):=(\widetilde{\mu}_k(Q)^{\operatorname{op}},\widetilde{\mu}_k(S)^{\operatorname{op}})$ and $\ell:=h(b_j)$, and deduce that there exist tuples
\[
\Scale[0.9]{
\left[\begin{array}{ccc}w_{j1} & \cdots & w_{jr}\end{array}\right]  \in \underset{j'=1}{\overset{r}{\bigoplus}} e_{h(b_j)}\compalg{\widetilde{\mu}_k(Q)}e_{h(b_{j'})}, 
\qquad \left[\begin{array}{ccc}x_{j1} & \cdots & x_{js}\end{array}\right]\in \underset{i=1}{\overset{s}{\bigoplus}}e_{h(b_j)}J(\widetilde{\mu}_k(S))e_{t(a_i)}
}
\]
such that 
$
\left[\begin{array}{ccc}v_{j1} & \cdots & v_{js}\end{array}\right] =
\left[\begin{array}{ccc}w_{j1} & \cdots & w_{jr}\end{array}\right]
C
+
\left[\begin{array}{ccc}x_{j1} & \cdots & x_{js}\end{array}\right].
$
Thus,
\begin{align*}
&    p_j- \sum_{i=1}^s \partial_{a_i^*b_j^*}(\widetilde{\mu}_k(S))q_i  =
    \left[
    \begin{array}{ccc}
    v_{j1} & \cdots    & v_{js} 
    \end{array}
    \right]
    \left[
    \begin{array}{c}
    a_1^*\\
    \vdots \\
    a_s^*
    \end{array}
    \right]=\\
    &= 
    \left[\begin{array}{ccc}w_{j1} & \cdots & w_{jr}\end{array}\right]
C
\left[
    \begin{array}{c}
    a_1^*\\
    \vdots \\
    a_s^*
    \end{array}
    \right]
+
\left[\begin{array}{ccc}x_{j1} & \cdots & x_{js}\end{array}\right]
\left[
    \begin{array}{c}
    a_1^*\\
    \vdots \\
    a_s^*
    \end{array}
    \right].
\end{align*}

Using \eqref{eq:matrix-of-partial-a*b*-of-tildemukS}, we deduce that
\begin{align*}
    \left(p_j- \sum_{i=1}^s \partial_{a_i^*b_j^*}(\widetilde{\mu}_k(S))q_i\right)_{\overline{M}}  
    &=
    \left[\begin{array}{ccc}\left(w_{j1}\right)_{\overline{M}} & \cdots & \left(w_{jr}\right)_{\overline{M}}\end{array}\right]
    \beta
    \alpha
    \overline{\beta}+ 
\\    
&+ \left[\begin{array}{ccc}\left(x_{j1}\right)_{\overline{M}} & \cdots & \left(x_{js}\right)_{\overline{M}}\end{array}\right]
\left[
    \begin{array}{c}
    \left(a_1^*\right)_{\overline{M}}\\
    \vdots \\
    \left(a_s^*\right)_{\overline{M}}
    \end{array}
    \right]= 0
\end{align*}
for all $j=1,\ldots,r$. This finishes the proof of Lemma \ref{lemma:yet-another-element-must-annihilate-overlineM}.
\end{proof}

Returning to the analysis of Case \ref{case:k=ell_2-in-proof-that-J(tildemuk(S))-annihilates-overlineM} in the proof of Proposition \ref{prop:Jacobian-ideal-J(widetildemuk(S))-annihilates-overlineM}, from Lemma \ref{lemma:yet-another-element-must-annihilate-overlineM} we deduce that
\[
\Scale[0.85]{
\left(\sum_{j=1}^rb_j^*p_j- \sum_{j=1}^rb_j^*\sum_{i=1}^s \partial_{a_i^*b_j^*}(\widetilde{\mu}_k(S))q_i\right)_{\overline{M}}=
\left(\sum_{j=1}^rb_j^*\left(p_j- \sum_{i=1}^s \partial_{a_i^*b_j^*}(\widetilde{\mu}_k(S))q_i\right)\right)_{\overline{M}}=0.
}
\]
 Using this together with \eqref{eq:def-of-ai*-and-bj*-cokerbeta-preference},\eqref{eq:action-of-arbitrary-u-on-overlineM} and \eqref{eq:matrix-of-partial-a*b*-of-tildemukS}, we see that
\[
\left(\sum_{j=1}^rb_j^*p_j\right)_{\overline{M}}
=
\left(\sum_{j=1}^rb_j^*\sum_{i=1}^s \partial_{a_i^*b_j^*}(\widetilde{\mu}_k(S))q_i\right)_{\overline{M}}=
\overline{\alpha}
\beta\alpha
\left[\begin{array}{c}(q_1)_M\\ \vdots \\ (q_s)_M\end{array}\right] =0.
\]
We infer that $u\cdot\overline{M}=\left(\sum_{j=1}^rb_j^*p_j\right)\cdot\overline{M}=0$.

\end{case}
    
\begin{case}
    Suppose $k=\ell_1$. If $k=\ell_2$, then the desired equality $u\cdot\overline{M}=0$ follows from Case \ref{case:k=ell_2-in-proof-that-J(tildemuk(S))-annihilates-overlineM}. Thus, we can suppose, without loss of generality, that $k\neq\ell_2$. Write
\[
    u
    =
    \left[
    \begin{array}{ccc}
    v_{1} & \cdots    & v_{s} 
    \end{array}
    \right]
    \left[
    \begin{array}{c}
    a_1^*\\
    \vdots \\
    a_s^*
    \end{array}
    \right],
\qquad
\text{with}
\qquad
\left[
    \begin{array}{ccc}
    v_{1} & \cdots    & v_{s} 
    \end{array}
    \right] \in \bigoplus_{i=1}^{s} e_{\ell_2}\widetilde{\idealM} e_{t(a_i)},
\]
where $\widetilde{\idealM}$ is the ideal of $\compalg{\widetilde{\mu}_k(Q)}$ generated by the arrows of $\widetilde{\mu}_k(Q)$ (note that $v_{i}\in e_{\ell_2}\widetilde{\idealM} e_{t(a_i)}$ because $Q$ does not have $2$-cycles incident to $k$). We can apply Corollary \ref{coro:most-basic-properties-of-alpha-beta-gamma-maps-from-projective} to $(Q',S'):=(\widetilde{\mu}_k(Q)^{\operatorname{op}},\widetilde{\mu}_k(S)^{\operatorname{op}})$ and $\ell:=\ell_2$, and deduce that there exist tuples
\[
\Scale[0.9]{
\left[\begin{array}{ccc}w_{1} & \cdots & w_{r}\end{array}\right]  \in \bigoplus_{j=1}^r e_{\ell_2}\compalg{\widetilde{\mu}_k(Q)}e_{h(b_{j})}, 
\qquad 
\left[\begin{array}{ccc}x_{1} & \cdots & x_{s}\end{array}\right]  \in \bigoplus_{i=1}^s e_{\ell_2}J(\widetilde{\mu}_k(S))e_{t(a_i)}
}
\]
 such that 
$
\left[\begin{array}{ccc}v_{1} & \cdots & v_{s}\end{array}\right] =
\left[\begin{array}{ccc}w_{1} & \cdots & w_{r}\end{array}\right]
C
+
\left[\begin{array}{ccc}x_{1} & \cdots & x_{s}\end{array}\right].
$
Thus,
\begin{align*}
   u 
    &= 
    \left[\begin{array}{ccc}w_{1} & \cdots & w_{r}\end{array}\right]
C
\left[
    \begin{array}{c}
    a_1^*\\
    \vdots \\
    a_s^*
    \end{array}
    \right]
+
\left[\begin{array}{ccc}x_{1} & \cdots & x_{s}\end{array}\right]
\left[
    \begin{array}{c}
    a_1^*\\
    \vdots \\
    a_s^*
    \end{array}
    \right].
\end{align*}

Since $x_i\in e_{\ell_2}J(\widetilde{\mu}_k(S))e_{t(a_i)}$ and $\ell_2\neq k \neq t(a_i)$, we have $(x_i)_{\overline{M}}=0$ by \cite[Proposition~6.1]{derksen2008quivers}.
Therefore, 
\begin{align*}
    u_{\overline{M}} 
    &=      \left[\begin{array}{ccc}\left(w_{1}\right)_{\overline{M}} & \cdots & \left(w_{r}\right)_{\overline{M}}\end{array}\right]
    \beta
    \alpha
    \overline{\beta} 
    +\left[\begin{array}{ccc}\left(x_{1}\right)_{\overline{M}} & \cdots & \left(x_{s}\right)_{\overline{M}}\end{array}\right]
\overline{\beta}
    = 0.
\end{align*}

\end{case}

     Proposition \ref{prop:Jacobian-ideal-J(widetildemuk(S))-annihilates-overlineM} is proved.
\end{proof}

\begin{defi}\label{def:premut-of-decorated-rep}
    In the situation of Proposition \ref{prop:Jacobian-ideal-J(widetildemuk(S))-annihilates-overlineM}, $\widetilde{\mu}_k(\mathcal{M}):=(\overline{M},\overline{V})$, which is a decorated module over $\jacobalg{\widetilde{\mu}_k(Q,S)}=\jacobalg{\widetilde{\mu}_k(Q),\widetilde{\mu}_k(S)}$, receives the name of \emph{premutation} of $\mathcal{M}=(M,V)$.
\end{defi}

\subsection{Definition of the Derksen-Weyman-Zelevinsky mutation \texorpdfstring{$\mu_k(\mathcal{M})$}{mk(M)}}

\begin{defi}\label{def:mut-of-dec-rep} Let $(Q,S)$ be a QP, and $k\in Q_0$ a vertex not incident to any $2$-cycle of $Q$. Fix a splitting $\varphi:(\widetilde{\mu}_k(Q),\widetilde{\mu}_k(S)_{\operatorname{red}}+\widetilde{\mu}_k(S)_{\operatorname{triv}})\rightarrow \widetilde{\mu}_k(Q,S)=(\widetilde{\mu}_k(Q),\widetilde{\mu}_k(S))$.
For a decorated $\jacobalg{Q,S}$-module $\mathcal{M}=(M,V)$, the \emph{Derksen-Weyman-Zelevinsky mutation of $\mathcal{M}=(M,V)$ in direction~$k$ with respect to $\varphi$}~is 
\begin{equation}\label{eq:the-DWZ-mut-of-M-wrt-a-fixed-reduction-of-premut}
\mu_k(\mathcal{M}):=(\varphi^{\#}(\overline{M}),\overline{V}),
\end{equation}
where $\widetilde{\mu}_k(\mathcal{M}):=(\overline{M},\overline{V})$ is the $k^{\operatorname{th}}$ premutation of $\mathcal{M}$ as in Definition \ref{def:premut-of-decorated-rep}, and $\varphi^{\#}(\overline{M})$ is the reduced part of the $\jacobalg{\widetilde{\mu}_k(Q,S)}$-module $\overline{M}$ with respect to $\varphi$ as in Definition \ref{def:reduced-and-trivial-parts}. 
\end{defi}

\begin{theorem}\label{thm:J(widetildemuk(M))-is-a-module-over-widetildemu(Q,S)}
    Let $(Q,S)$ be a QP and $\mathcal{M}=(M,V)$ a decorated $\jacobalg{Q,S}$-module. If $k\in Q_0$ is not incident to any $2$-cycle of $Q$, 
    then the Derksen-Weyman-Zelevinsky mutation $\mu_k(\mathcal{M})$ from \eqref{eq:the-DWZ-mut-of-M-wrt-a-fixed-reduction-of-premut} is a decorated module over $\jacobalg{\mu_k(Q,S)}:=(\widetilde{\mu}_k(Q)_{\operatorname{red}},\widetilde{\mu}_k(S)_{\operatorname{red}})$.
\end{theorem}

\begin{remark}
In the case of finite-dimensional decorated representations, Definition \ref{def:mut-of-dec-rep} specializes to \cite[Equation (10.22)]{derksen2008quivers}, in which it is directly inspired.
\end{remark}

\section{Examples}\label{sec:examples}

\begin{example}\label{ex:mut-of-loc-nil-inj-for-gentle-Markov}
    The \emph{Derksen-Weyman-Zelevinsky Markov QP} 
    \[
    \xymatrix{ & & k \ar@<0.5ex>[dr]^{b_1} \ar@<-0.5ex>[dr]_{b_2} & & \\
    Q: & 1 \ar@<0.5ex>[ur]^{a_1} \ar@<-0.5ex>[ur]_{a_2} & & 2 \ar@<0.5ex>[ll]^{c_1} \ar@<-0.5ex>[ll]_{c_2} & S=c_1b_1a_1+c_2b_2a_2,
    } 
    \]
    is non-degenerate by \cite[Example 8.6]{derksen2008quivers}. Its premutation $(\widetilde{\mu}_k(Q),\widetilde{\mu}_k(S))$ is
    \[
    \xymatrix{& & k \ar@<0.5ex>[dl]_{a_1^*} \ar@<-0.5ex>[dl]^{a_2^*} & & \\
    \widetilde{\mu}_k(Q): &1 \ar@/_/@<-1ex>[rr]|-{[b_1a_1]} \ar@/_1pc/@<-1.5ex>[rr]|-{[b_1a_2]} \ar@/_1.5pc/@<-2ex>[rr]|-{[b_2a_1]} \ar@/_2pc/@<-2.5ex>[rr]|-{[b_2a_2]}  & & 2 \ar@<0.5ex>[ll]|-{c_1} \ar@<-0.5ex>[ll]|-{c_2} \ar@<0.5ex>[ul]_{b_1^*} \ar@<-0.5ex>[ul]^{b_2^*} & *\txt{$\widetilde{\mu}_k(S)=c_1[b_1a_1]+c_2[b_2a_2]+\qquad\qquad\quad $ \\ $+a_1^*b_1^*[b_1a_1]+a_2^*b_1^*[b_1a_2]+$\\
    $+a_1^*b_2^*[b_2a_1]+a_2^*b_2^*[b_2a_2].\ $}
    }
    \]
    Hence $\widetilde{\mu}_k(S)_{\operatorname{triv}}=c_1[b_1a_1]+c_2[b_2a_2]$ and  $\widetilde{\mu}_k(S)_{\operatorname{red}}=a_2^*b_1^*[b_1a_2]+a_1^*b_2^*[b_2a_1]$, and Derksen-Weyman-Zelevinsky's reduction process produces a right-equivalence $\varphi:(\widetilde{\mu}_k(Q),\widetilde{\mu}_k(S)_{\operatorname{red}}+\widetilde{\mu}_k(S)_{\operatorname{triv}})\rightarrow(\widetilde{\mu}_k(Q),\widetilde{\mu}_k(S))$ given by $c_1\mapsto c_1+a_1^*b_1^*, c_2\mapsto c_2+a_2^*b_2^*$ and the identity on the rest of the arrows. Thus, $\mu_k(Q,S)$ is
    \[
    \Scale[0.95]{
    \xymatrix{& & k \ar@<0.5ex>[dl]_{a_1^*} \ar@<-0.5ex>[dl]^{a_2^*} & & \\
    \widetilde{\mu}_k(Q)_{\operatorname{red}}: & 1  \ar[rr]|-{[b_1a_2]} \ar@/_/@<-0.5ex>[rr]|-{[b_2a_1]}   & & 2  \ar@<0.5ex>[ul]_{b_1^*} \ar@<-0.5ex>[ul]^{b_2^*} & \mu_k(S):=\widetilde{\mu}_k(S)_{\operatorname{red}}=a_2^*b_1^*[b_1a_2]+a_1^*b_2^*[b_2a_1],
    }
    }
    \]
    and for every $\jacobalg{\widetilde{\mu}_k(Q),\widetilde{\mu}_k(S)}$-module $N$, the $\jacobalg{\mu_k(Q,S)}$-module structure of $\varphi^{\#}(N)$ is induced by the inclusion $\compalg{\widetilde{\mu}_k(Q)_{\operatorname{red}}}\hookrightarrow \compalg{\widetilde{\mu}_k(Q)}$, i.e., by simply forgetting the action of the paths on $\widetilde{\mu}_k(Q)$ that involve at least one arrow from $c_1,c_2,[b_1a_1],[b_2a_2]$. 
    
It is very easy to check that the quiver representation
    \[
    \xymatrix{& & \field^{(\mathbb{Z}_{\geq 0})}\oplus\field^{(\mathbb{Z}_{\geq 0})} \ar@<0.5ex>[dr]^{\text{{\tiny $\begin{array}{l}M_{b_1}=\left[\begin{array}{c|c}0 & 0\\ \hline \myid & 0\end{array}\right] \\ M_{b_2}=\left[\begin{array}{c|c}0 & \myid\\ \hline 0 & 0\end{array}\right]\end{array}$}}} \ar@<-0.5ex>[dr] & \\
    M :& \field\oplus\field^{(\mathbb{Z}_{>0})}\oplus \field^{(\mathbb{Z}_{>0})} \ar@<0.5ex>[ur]^{\text{{\tiny $\begin{array}{l}M_{a_1}=\left[\begin{array}{c|c|c}0 & 0 & 0\\ \hline 0 & \myid & 0\end{array}\right] \\ M_{a_2}=\left[\begin{array}{c|c|c}0 & 0 & \myid\\ \hline 0 & 0 & 0\end{array}\right]\end{array}$}}} \ar@<-0.5ex>[ur] & & \field^{(\mathbb{Z}_{\geq 0})}\oplus\field^{(\mathbb{Z}_{\geq 0})} \ar@<0.5ex>[ll]^{\text{{\tiny $\begin{array}{l}M_{c_1}=\left[\begin{array}{cc|c}1 & 0 & 0\\ \hline 0 & 0 & 0 \\ \hline 0 & \myid & 0\end{array}\right] \quad M_{c_2}=\left[\begin{array}{c|cc} 0 & 1 & 0 \\ 0 & 0 & \myid\\ \hline 0 & 0 & 0\end{array}\right]\end{array}$}}  } \ar@<-0.5ex>[ll]
    }
    \]
    of $Q$
    is locally nilpotent, so its $\pathalg{Q}$-module structure can be extended in a unique way to a $\compalg{Q}$-module structure (see \S\ref{sec:quiver-reps-vs-modules-over-comp-path-algs}). Direct computation shows that it is annihilated by the cyclic derivatives of $S$, hence by the entire Jacobian ideal $J(S)$, given its local nilpotency. It can be seen that $M$ is in fact the injective envelope of the simple $S(1)$ in the category of locally nilpotent $\jacobalg{Q,S}$-modules, details will be provided in one of the sequels to this paper.
    
    Let us mutate $(M,0)$ with respect to the vertex $k$. Its $\alpha$-$\beta$-$\gamma$-triangle \eqref{eq:DWZ-alphabetagamma-triangle} is
    \[
    \xymatrix{
    & \field^{(\mathbb{Z}_{\geq 0})}\oplus\field^{(\mathbb{Z}_{\geq 0})}  \ar[dr]^{\text{{\tiny $\beta=\left[\begin{array}{c|c}0 & 0\\ \hline \myid & 0\\ \hline 0 & \myid\\ \hline 0 & 0\end{array}\right]$}}} & \\
    \left(\field\oplus\field^{(\mathbb{Z}_{>0})}\oplus \field^{(\mathbb{Z}_{>0})}\right)^2 \ar[ur]^{\text{{\tiny $\alpha= \left[\begin{array}{c|c|c|c|c|c}0 & 0 & 0 & 0 & 0 & \myid   \\ \hline 0 & \myid & 0 & 0 & 0 & 0\end{array}\right]$}}\qquad} & & \left(\field^{(\mathbb{Z}_{\geq 0})}\oplus\field^{(\mathbb{Z}_{\geq 0})}\right)^2 \ar[ll]^{\text{{\tiny $
    \gamma= \left[\begin{array}{cc|c|c|cc}1 & 0 & 0 & 0 & 0 & 0\\ \hline 0 & 0 & 0 & 0 & 0 & 0\\ \hline 0 & \myid & 0 & 0 & 0 & 0 \\ \hline 0 & 0 & 0 & 0 & 1 & 0 \\ \hline 0 & 0 & 0 & 0 & 0 & \myid\\ \hline 0 & 0 & 0 & 0 & 0 & 0\end{array}\right]$}}}
    }
    \]
    We have $\ker\gamma/\image\beta=0$ and $\ker\alpha\cong\field\oplus\field^{(\mathbb{Z}_{>0})}\oplus\field\oplus\field^{(\mathbb{Z}_{>0})}$.
    Applying the definition of $\overline{M}_k$, $\overline{\alpha}$ and $\overline{\beta}$ given in  \S\ref{subsubsec:def-of-overlineM-giving-preference-to-ker-alpha}, we obtain
    \[
    \Scale[0.925]{
    \xymatrix{
     & \ar[dl]_{\text{{\tiny $
    \overline{\beta}= \left[\begin{array}{cccc}
    1 & 0 & 0 & 0\\  0 & 0 & 0 & 0\\  0 & \myid & 0 & 0 \\ \hline  0 & 0 & 1 & 0 \\  0 & 0 & 0 & \myid\\  0 & 0 & 0 & 0\end{array}\right]$}}\qquad\qquad } \overline{M}_k=\field\oplus\field^{(\mathbb{Z}_{>0})}\oplus\field\oplus\field^{(\mathbb{Z}_{>0})} & \\
     \left(\field\oplus\field^{(\mathbb{Z}_{>0})}\oplus \field^{(\mathbb{Z}_{>0})}\right)^2 \ar[rr]_{\text{{\tiny $
    \left[\begin{array}{cccc}M_{b_1}M_{a_1} & M_{b_1}M_{a_2} \\ M_{b_2}M_{a_1} & M_{b_2}M_{a_2}\end{array}\right]=
    \left[\begin{array}{ccc|ccc} 0 & 0 & 0 & 0 & 0 & 0\\ 0 & 0 & 0 & 0 & 0 & \myid\\ \hline 0 & \myid & 0 & 0 & 0 & 0\\ 0 & 0 & 0 & 0 & 0 & 0\end{array}\right]$}}} & &  \left(\field^{(\mathbb{Z}_{\geq 0})}\oplus\field^{(\mathbb{Z}_{\geq 0})}\right)^2 \ar[ul]_{\qquad\qquad\qquad\quad \text{{\tiny $
    \overline{\alpha}= \left[\begin{array}{cc|c|c|cc}
    -1 & 0 & 0 & 0 & 0 & 0\\  0 & -\myid & 0 & 0 & 0 & 0\\  0 & 0 & 0 & 0 & -1 & 0 \\  0 & 0 & 0 & 0 & 0 & -\myid\end{array}\right]$}}}
    }
    }
    \]

    We have already established that the $\jacobalg{\mu_k(Q,S)}$-module structure of $\varphi^{\#}(\overline{M})$ is induced by the inclusion $\compalg{\widetilde{\mu}_k(Q)_{\operatorname{red}}}\hookrightarrow \compalg{\widetilde{\mu}_k(Q)}$, hence the quiver representation that underlies the mutated decorated module $\mu_k(M,0)$ is
    \[
    \xymatrix{ & \field^{(\mathbb{Z}_{\geq 0})}\oplus\field^{(\mathbb{Z}_{\geq 0})} \ar@<0.5ex>[dl]_{\text{{\tiny $\begin{array}{l}
    \overline{M}_{a_1^*}= \left[\begin{array}{cc|c}
    1 & 0 & 0 \\  0 & 0 & 0 \\  0 & \myid & 0 \end{array}\right] \\ \overline{M}_{a_2^*}=\left[\begin{array}{c|cc}
     0 & 1 & 0 \\   0 & 0 & \myid\\   0 & 0 & 0\end{array}\right]\end{array}$}}\qquad} \ar@<-0.5ex>[dl] &  \\
      \field\oplus\field^{(\mathbb{Z}_{>0})}\oplus \field^{(\mathbb{Z}_{>0})}  \ar@<0.5ex>[rr] \ar@<-0.5ex>[rr]_{ 
     \text{{\tiny $\begin{array}{l}\overline{M}_{[b_1a_2]}=\left[\begin{array}{ccc}0 & 0 & 0 \\ 0 & 0 & \myid\end{array}\right] \\ \overline{M}_{[b_2a_1]}=\left[\begin{array}{ccc}0 & \myid & 0 \\ 0 & 0 & 0\end{array}\right] \end{array}$}}}   & & \field^{(\mathbb{Z}_{\geq 0})}\oplus\field^{(\mathbb{Z}_{\geq 0})},  \ar@<0.5ex>[ul]_{\quad \text{{\tiny $\begin{array}{l}\overline{M}_{b_1^*}=\left[\begin{array}{cc}-\myid & 0 \\ 0 & 0\end{array}\right] \\ \overline{M}_{b_2^*}=\left[\begin{array}{cc}0 & 0 \\ 0 & -\myid\end{array}\right] \end{array}$}}} \ar@<-0.5ex>[ul] 
    }
    \]
    i.e., the result of taking $\overline{M}$ and simply forgetting the action of the arrows in the degree-$2$ component of $\widetilde{\mu}_k(S)$. 
    Again, it can be seen that $\overline{M}$ is the injective envelope of the simple $S(1)$ in the category of locally nilpotent $\jacobalg{\mu_k(Q,S)}$-modules.
\end{example}

It will be proved in one of the sequels to this paper that Example \ref{ex:mut-of-loc-nil-inj-for-gentle-Markov} is an instance of a general phenomenon, namely, that letting $I_{(Q,S)}^{\operatorname{l.n.}}(\ell)$ be the injective envelope of the simple $S(\ell)$ in the category of locally nilpotent $\jacobalg{Q,S}$-modules for an arbitrary QP $(Q,S)$, we have:

\begin{theorem}
    If $k\in Q_0$ is not incident to any $2$-cycle of $Q$, then for every $\ell\in Q_0\setminus\{k\}$ there is an isomorphism of decorated $\jacobalg{\mu_k(Q,S)}$-modules
    \[
    \mu_k(I_{(Q,S)}^{\operatorname{l.n.}}(\ell),0)\cong(I_{\mu_k(Q,S)}^{\operatorname{l.n.}}(\ell),0).
    \]
\end{theorem}

As for projectives, we have the following generalization of \cite[Theorem~3.1]{labardini2022landau} from the finite- to the infinite-dimensional setting.

\begin{theorem}
    If $k\in Q_0$ is not incident to any $2$-cycle of $Q$, then for every $\ell\in Q_0\setminus\{k\}$ there is an isomorphism of decorated $\jacobalg{\mu_k(Q,S)}$-modules
    \[
    \mu_k(P_{(Q,S)}(\ell),0) \cong (P_{\mu_k(Q,S)}(\ell),0).
    \]
\end{theorem}

\begin{remark} It is the proof of \cite[Theorem 3.1]{labardini2022landau} that brought to the author's attention that it was possible to define Derksen-Weyman-Zelevinsky mutations for infinite-dimensional modules, see \cite[Remark 3]{labardini2022landau}.
\end{remark}

How about $\mu_k(P_{(Q,S)}(\ell),0)$ when $k=\ell$? With the aid of Maximilian Kaipel and Håvard Terland, the author has been able to prove the following result, regardless of whether the modules involved are finite- or infinite-dimensional. The proof will be included in one of the sequels to this paper. 

\begin{theorem}\label{thm:muk-of-Pk-is-Aihara-Iyama-mutation}
If $k\in Q_0$ is not incident to any $2$-cycle of $Q$, then there is an exact sequence in $\jacobalg{\mu_k(Q,S)}$-$\Mod$
\[
\xymatrix{
P_{\mu_k(Q,S)}(k) \ar[r]^-{f} & \bigoplus_{j=1}^r P_{\mu_k(Q,S)}(h(b_j)) \ar[r] & \mu_{k}(P_{(Q,S)}(k)) \ar[r] & 0,
}
\]
where $f$ is a left $\operatorname{add}(\bigoplus_{\ell\neq k}P_{\mu_k(Q,S)}(\ell))$-approximation of $P_{\mu_k(Q,S)}(k)$. Consequently, there is an isomorphism of decorated $\jacobalg{\mu_k(Q,S)}$-modules
\[
 \mu_k^{{\operatorname{DWZ}}}(\jacobalg{Q,S})=\mu_k^{{\operatorname{AI}}}(\jacobalg{\mu_k(Q,S}),
\]
where $\mu_k^{{\operatorname{DWZ}}}$ is the DWZ mutation of modules, and $\mu_k^{{\operatorname{AI}}}$ is Aihara-Iyama's mutation of $2$-term silting complexes.
\end{theorem}

Theorem \ref{thm:muk-of-Pk-is-Aihara-Iyama-mutation} certainly suggests a deeper relation between the Derksen-Weyman-Zelevisky mutation of modules and Aihara-Iyama's mutation of $2$-term silting complexes \cite{aihara2012silting}. This relation will be explored in this paper's sequels.

Example \ref{ex:mut-of-loc-nil-inj-for-gentle-Markov} is part of yet another general phenomenon, arising in the setting of surfaces with marked points, that we will study in detail in the sequels to this paper. The next example illustrates this phenomenon.

\begin{example}
Consider the triangulations $T_1$ and $T_2$, and the curve $\gamma$ in Figure \ref{Fig_TriangsAndSpirals}.
\begin{figure}[t]
                \centering
                \includegraphics[scale=.4]{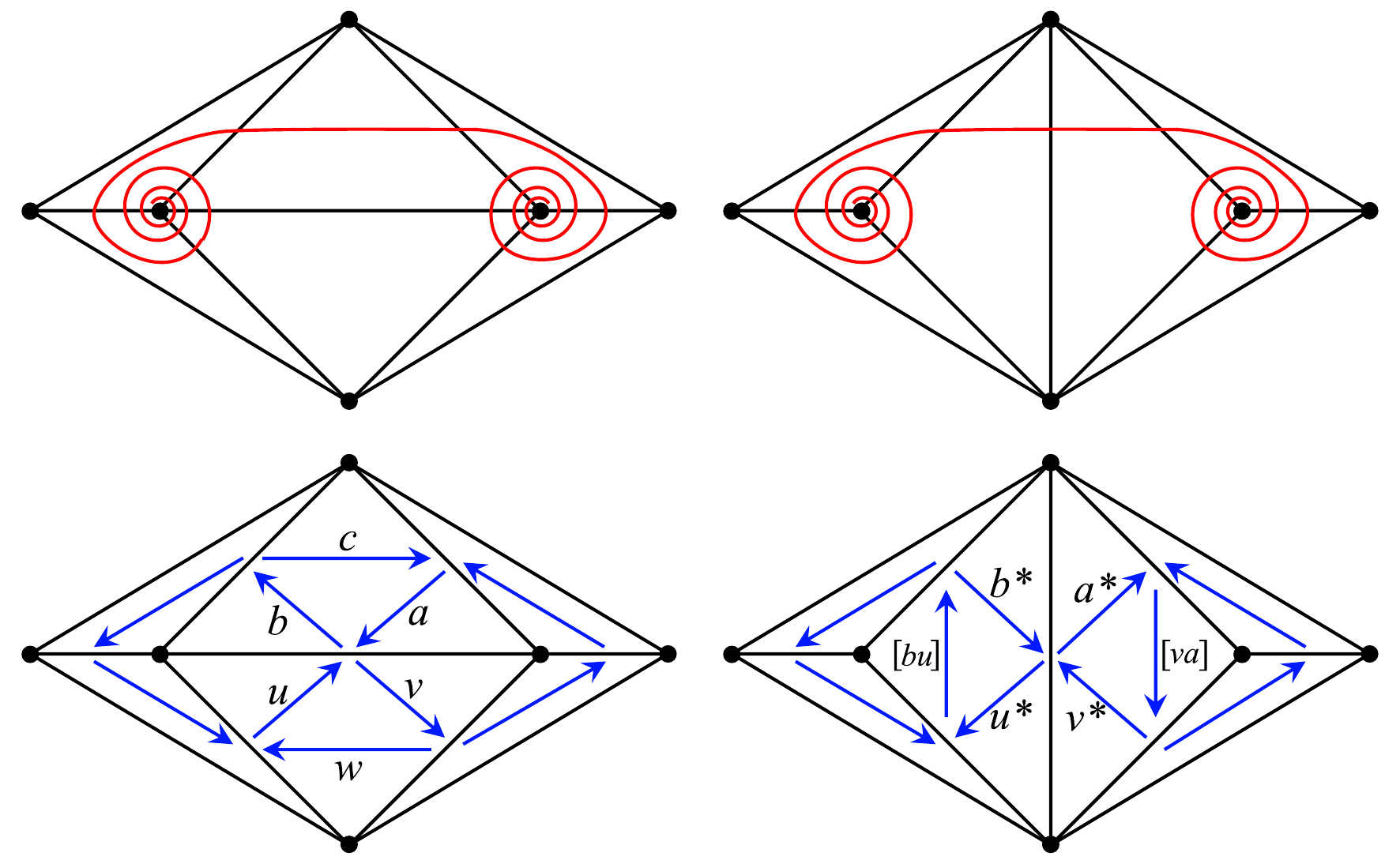}
                \caption{Spiral $\gamma$, triangulations $T_1$ (left) and $T_2$ (right), and their quivers $Q(T_1)$ and $Q(T_2)$.}
                \label{Fig_TriangsAndSpirals}
        \end{figure}
The quivers $Q(T_1)$ and $Q(T_2)$ have been drawn in Figure \ref{Fig_TriangsAndSpirals} as well. To $T_1$ and~$T_2$ we associate the potentials
\begin{align*}
    S(T_1) &:= cba + wvu, & S(T_2) & := b^*[bu]u^*+v^*[va]a^*
\end{align*}
which are the ones defined in \cite[Definition~23]{labardini2009quivers} if one sets the scalars $x_p$ to be $0$ therein. The two triangulations are related by the flip of $k\in T_1$, so there is a right-equivalence $\varphi:\mu_k(Q(T_1),S(T_1))\rightarrow (Q(T_2),S(T_2))$ given by $b^*\mapsto-b^*$, $u^*\mapsto-u^*$, exactly as in \cite[Case 1 of the proof of Theorem~30]{labardini2009quivers}. The curve $\gamma$ spirals in the clockwise sense at one end and in the counterclockwise sense at the other. To it we associate a $\jacobalg{Q(T_1,S(T_1))}$-module $M(T_1,\gamma)$ and a $\jacobalg{Q(T_2,S(T_2))}$-module $M(T_2,\gamma)$, whose images under the forgetful functor $F$ from \S\ref{sec:quiver-reps-vs-modules-over-comp-path-algs} are the following quiver representations (none of which is locally nilpotent):
\[
\xymatrix{
 & \field^{\bbZ_{>0}} \ar[rrrr]^{\iota_1p_1} \ar[dl]_{\myid} & & & & \field^{(\bbZ_{>0})} \ar[dll]|-{\left[\begin{array}{c}0\\ \Sigma\end{array}\right]} \\
 \field^{\bbZ_{>0}} \ar[dr]_{\myid} & & &  \field^{\bbZ_{>0}}\oplus \field^{(\bbZ_{>0})} \ar[ull]|-{\left[\begin{array}{cc} \Sigma^- & 0 \end{array}\right]} \ar[drr]|-{\left[\begin{array}{cc}0 & \myid \end{array}\right]} & & & \field^{(\bbZ_{>0})} \ar[ul]_{\myid}\\
 & \field^{\bbZ_{>0}} \ar[urr]|-{\left[\begin{array}{c}\myid \\ 0\end{array}\right]} & & & & \field^{(\bbZ_{>0})}, \ar[llll]^{0}  \ar[ur]_{\myid} &
}
\]
where $\Sigma(x_1,x_2,\ldots):=(x_2,\ldots)$, $\Sigma^-(x_1,x_2,\ldots):=(0,x_1,x_2,\ldots)$, $p_1:\field^{\bbZ_{>0}}\rightarrow\field$ is the projection from the product onto the first direct factor, and $\iota_1:\field \rightarrow \field^{(\bbZ_{>0})}$ is the inclusion of the first direct summand into the coproduct, and:
\[
\xymatrix{
& \field^{\bbZ_{>0}} \ar[drr]^{p_1} \ar[dl]_{\myid} & & & & \field^{(\bbZ_{>0})} \ar[dd]_{\Sigma} & \\
\field^{\bbZ_{>0}} \ar[dr]_{\myid} & & & \field \ar[urr]^{\iota_1} \ar[dll]^{0} &  & & \field^{(\bbZ_{>0})} \ar[ul]_{\myid}\\
& \field^{\bbZ_{>0}} \ar[uu]_{\Sigma^-}   & & & & \field^{(\bbZ_{>0})} \ar[ull]^{0} \ar[ur]_{\myid} &
}
\]

Form DWZ's $\alpha$-$\beta$-$\gamma$-triangle
\[
\xymatrix{
 & M(T_1,\gamma)_k=\field^{\bbZ_{>0}}\oplus \field^{(\bbZ_{>0})} \ar[dr]^{\qquad\qquad \beta=\text{{\tiny $\left[\begin{array}{cc} \Sigma^- & 0 \\ 0 & \myid\end{array}\right]$}}} & \\
 \field^{\bbZ_{>0}}\oplus \field^{(\bbZ_{>0})} \ar[ur]^{\alpha=\text{{\tiny $\left[\begin{array}{cc}\myid & 0 \\ 0 & \Sigma\end{array}\right]$}}\qquad\qquad} & & \field^{\bbZ_{>0}}\oplus \field^{(\bbZ_{>0})} .\ar[ll]^{\gamma=\left[\begin{array}{cc}0 & 0\\ \iota_1p_1 & 0\end{array}\right]}
}
\]
Noticing that $\ker\alpha=\image\gamma$, to compute $\overline{M(T_1,\gamma)}$ we choose to give preference to $\coker\beta\cong\field$, thus obtaining
\[
\xymatrix{
 & \overline{M(T_1,\gamma)}_k\cong\field \ar[dl]_{\overline{\beta}:=\left[\begin{array}{c}\iota_1 \\ 0\end{array}\right]\qquad} & \\
 \field^{\bbZ_{>0}}\oplus \field^{(\bbZ_{>0})} & & \field^{\bbZ_{>0}}\oplus \field^{(\bbZ_{>0})}. \ar[ul]_{\qquad\overline{\alpha}=\left[\begin{array}{cc}-p_1 & 0\end{array}\right]}
}
\]
Therefore, $\overline{M(T_1,\gamma)}$ is
\[
\xymatrix{
& \field^{\bbZ_{>0}} \ar@/^0.55pc/[rrrr]^{\iota_1p_1} \ar[drr]|-{-p_1} \ar[dl]_{\myid} & & & & \field^{(\bbZ_{>0})} \ar@/^0.55pc/[llll]^{0} \ar[dd]_{\Sigma} & \\
\field^{\bbZ_{>0}} \ar[dr]_{\myid} & & & \field \ar[urr]|-{\iota_1} \ar[dll]|-{0} &  & & \field^{(\bbZ_{>0})} \ar[ul]_{\myid}\\
& \field^{\bbZ_{>0}} \ar@/^0.55pc/[rrrr]^{0} \ar[uu]_{\Sigma^-}   & & & & \field^{(\bbZ_{>0})}. \ar@/^0.55pc/[llll]^{0} \ar[ull]|-{0} \ar[ur]_{\myid}  &
}
\]
Thus, $\varphi^{\#}(\overline{M(T_1,\gamma)})=M(T_2,\gamma)$, hence $\mu_k(M(T_1,\gamma),0)=
%(\varphi^{\#}(\overline{M(T_1,\gamma)}),0)=
(M(T_2,\gamma))$.
\end{example}

Yet another research direction opened by the present paper consists in determining the relation between the infinite-dimensional representation theories of QPs related by a mutation. For instance, DWZ-mutation of modules always sends pure-injective modules to pure-injective ones.
Furthermore, through a combination of results from \cite{laking2016thesis,laking2018krullgabriel} with results from \cite{labardini2010quivers}, Rosie Laking and the author have obtained the following.

\begin{theorem}\label{thm:homeo-between-Ziegler-spec-for-affine-A}
    Suppose $(Q,S)$ is the quiver with potential associated in \cite{labardini2009quivers} to a triangulation of an unpunctured annulus. For any vertex $k\in Q_0$, the DWZ-mutation of modules is a homeomorphism between the Ziegler spectra of the Jacobian algebras $\jacobalg{Q,S}$ and $\jacobalg{\mu_k(Q,S)}$.
\end{theorem}

The proof of Theorem \ref{thm:homeo-between-Ziegler-spec-for-affine-A} and the next Conjecture \ref{conj:homeo-between-Ziegler-spec} will be addressed in the sequels to the present paper.

\begin{conj}\label{conj:homeo-between-Ziegler-spec}
    For every QP $(Q,S)$ and every vertex $k\in Q_0$ not incident to any $2$-cycle of $Q$, the DWZ-mutation of modules is a homeomorphism between the Ziegler spectra of the Jacobian algebras $\jacobalg{Q,S}$ and $\jacobalg{\mu_k(Q,S)}$.
\end{conj}

\section{DWZ-mutations commute with vector space duality}\label{sec:muts-commute-with-duality}

Denote by $(Q^{\operatorname{op}},S^{\operatorname{op}})$ the QP opposite to $(Q,S)$. For any decorated $\jacobalg{Q,S}$-module $\mathcal{M}=(M,V)$, the pair $\mathcal{M}^{\vee}:=(M^{\vee},V^\vee)$ is a decorated $\jacobalg{Q^{\operatorname{op}},S^{\operatorname{op}}}$-module, where $?^{\vee}:=\Hom_{\field}(?,\field)$ is the usual vector space duality, see \cite[the paragraph preceding Proposition 7.3]{derksen2010quivers}. This subsection is devoted to the following result, proved in \cite[\S4]{labardini2022landau} in the finite-dimensional case. We give the admittedly elementary proof in full detail here in order to make it clear that absolutely no assumption of finite-dimensionality is needed (e.g., no recourse to dimension countings is needed to establish bijectivity).

\begin{theorem}\label{thm:duality-commutes-with-mutation}
If no $2$-cycle of $Q$ is incident to $k\in Q_0$, then $\mu_{k}(Q^{\operatorname{op}},S^{\operatorname{op}})=(\mu_{k}(Q)^{\operatorname{op}},\mu_k(S)^{\operatorname{op}})$ as QPs, and for any decorated $\jacobalg{Q,S}$-module $\mathcal{M}=(M,V)$, the decorated $\jacobalg{\mu_{k}(Q^{\operatorname{op}},S^{\operatorname{op}})}$-modules $\mu_k(\mathcal{M}^{\vee})$ and $\mu_k(\mathcal{M})^{\vee}$ are isomorphic.
\end{theorem}

To compute $\mu_k(\mathcal{M})$ and $\mu_k(\mathcal{M}^{\vee})$, one needs to consider the $\alpha$-$\beta$-$\gamma$-triangles 
\[
\xymatrix{
 & M_k \ar[dr]^{\beta} & \\
 M_{\operatorname{in}} \ar[ur]^{\alpha} & & M_{\operatorname{out}} \ar[ll]_{\gamma}
}
\quad \text{and} \quad
\xymatrix{
 & M_k^\vee \ar[dl]_{\alpha^\vee} & \\
 M_{\operatorname{in}}^\vee \ar[rr]_{\gamma^\vee} & & M_{\operatorname{out}}^\vee \ar[ul]_{\beta^\vee}
}
\]

\begin{lemma}\label{lemma:duality-commutes-with-mutation-linear-map-psi_k} 
The map
\begin{equation}\label{eq:duality-commutes-with-mutation-linear-map-psi_k}
\delta_k:\coker(\alpha^\vee)\oplus\frac{\ker(\beta^\vee)}{\image(\gamma^\vee)}\oplus V_k^\vee\longrightarrow (\coker\beta)^\vee\oplus\left(\frac{\ker\alpha}{\image\gamma}\right)^\vee\oplus V_k^\vee
\end{equation}
which for $f\in M_{\operatorname{in}}^\vee$, $g\in \ker(\beta^\vee)\subseteq M_{\operatorname{out}}^\vee$ and $\mathbf{z}\in V_k^\vee$ is given by
\begin{equation}\label{eq:duality-commutes-with-mutation-linear-map-psi_k-correspondence-rule}
\delta_k(f+\image(\alpha^\vee),g+\image(\gamma^\vee),\mathbf{z}):=(f|_{\image\gamma}\circ\overline{\gamma}+\overline{g|_{\ker\gamma}\circ\rho},f|_{\ker\alpha}\circ\sigma,\mathbf{z}),
\end{equation}
is well-defined, and in fact an isomorphism of $\field$-vector spaces.
\end{lemma}

\begin{proof}
For $g\in \ker(\beta^\vee)$, one has $g\circ \beta=0$, i.e., $g|_{\image\beta}=0$, hence $(g|_{\ker\gamma}\circ\rho)|_{\image\beta}=0$ because $\image\beta\subseteq\ker\gamma$. Thus, $g|_{\ker\gamma}\circ\rho:M_{\operatorname{out}}\rightarrow\field$ does induce a well-defined $\field$-linear map $\overline{g|_{\ker\gamma}\circ\rho}:\coker\beta\rightarrow\field$. Furthermore, if $f\in \image(\alpha^\vee)$ and $g\in\image(\gamma^\vee)$, so that $f=f'\circ\alpha$ and $g=g'\circ\gamma$, then 
\begin{align*}
f|_{\ker\alpha}\circ\sigma&=(f'\circ\alpha)|_{\ker\alpha}\circ\sigma=0\\
f|_{\image\gamma}\circ\overline{\gamma}+\overline{g|_{\ker\gamma}\circ\rho} &=
(f'\circ\alpha)|_{\image\gamma}\circ\overline{\gamma}+\overline{(g'\circ\gamma)|_{\ker\gamma}\circ\rho} = 0
\end{align*}
since $(f'\circ\alpha)|_{\ker\alpha}=0$ $\alpha\circ\gamma=0$ and $(g'\circ\gamma)|_{\ker\gamma}=0$. Therefore, the $\field$-linear map \eqref{eq:duality-commutes-with-mutation-linear-map-psi_k} given by the rule \eqref{eq:duality-commutes-with-mutation-linear-map-psi_k-correspondence-rule} is well-defined.

If $f\in M_{\operatorname{in}}^\vee$, $g\in \ker(\beta^\vee)\subseteq M_{\operatorname{out}}^\vee$ and $\mathbf{z}\in V_k^\vee$ satisfy $\delta_k(f+\image(\alpha^\vee),g+\image(\gamma^\vee),\mathbf{z})=0$, then 
\begin{align}
    %g\circ\beta&=0,\\
    \label{eq:duality-commutes-with-mutation-what-it-means-to-be-in-ker-psi_k}
    -f|_{\image\gamma}\circ\overline{\gamma} &= \overline{g|_{\ker\gamma}\circ\rho}, &
    f|_{\ker\alpha}\circ\sigma&=0
\end{align}
and $\mathbf{z}=0$. From \eqref{eq:duality-commutes-with-mutation-what-it-means-to-be-in-ker-psi_k} we deduce that for every $m\in\ker\gamma$:
\begin{align*}
g(m)&=(g|_{\ker\gamma}\circ\rho)(m)=(\overline{g|_{\ker\gamma}\circ\rho})(m+\image\beta)=\\
&=(-f|_{\image\gamma}\circ\overline{\gamma})(m+\image\beta)=(-f|_{\image\gamma}\circ\gamma)(m)=0,
\end{align*}
which means that there is a well-defined $\field$-linear map $g'':\image\gamma\rightarrow \field$ such that $g=g''\circ\gamma$. Regardless of whether the dimensions of $M_{\operatorname{in}}$ and $\image\gamma$ are finite or infinite, $g''$ can definitely be extended to a $\field$-linear map $g':M_{\operatorname{in}}\rightarrow \field$, which certainly satisfies $g=g'\circ\gamma$. Thus, $g\in\image(\gamma^\vee)$.

We deduce now that $f|_{\image\gamma}\circ\overline{\gamma} = -\overline{(g'\circ\gamma)|_{\ker\gamma}\circ\rho}=0$, which implies $f|_{\image\gamma}=0$ because $\image\overline{\gamma}=\image\gamma$. Take $n\in \ker\alpha$. Then $n=n-(\sigma\circ \pi)(n)+(\sigma\circ \pi)(n)$, where $\pi:\ker\alpha\rightarrow\frac{\ker\alpha}{\image\gamma}$ is the canonical projection. Since $n-(\sigma\circ \pi)(n)\in\ker \pi=\image\gamma$, we have $f(n-(\sigma\circ \pi)(n))=0$, and from~\eqref{eq:duality-commutes-with-mutation-what-it-means-to-be-in-ker-psi_k} we deduce that $f((\sigma\circ \pi)(n))=0$. Hence $f(n)=0$, which means that there is a well-defined $\field$-linear map $f'':\image\alpha\rightarrow \field$ such that $f=f''\circ\alpha$. Regardless of whether the dimensions of $M_{k}$ and $\image\alpha$ are finite or infinite, $f''$ can definitely be extended to a $\field$-linear map $f':M_{k}\rightarrow \field$, which certainly satisfies $f=f'\circ\alpha$. Thus, $f\in\image(\alpha^\vee)$. Therefore, $\delta_k$ is injective.

To prove the surjectivity of $\delta_k$, take $(\overline{h},\overline{j})\in (\coker\beta)^\vee\oplus\left(\frac{\ker\alpha}{\image\gamma}\right)^\vee$. Set $h:=\overline{h}\circ p:M_{\operatorname{out}}\rightarrow \field$ and $j:=\overline{j}\circ\pi:\ker\alpha\rightarrow\field$, where $p:M_{\operatorname{out}}\rightarrow\coker\beta$ and $\pi:\ker\alpha\rightarrow\frac{\ker\alpha}{\image\gamma}$ are the canonical projections. Notice that $\ker\gamma\subseteq\ker(h-h|_{\ker\gamma}\circ\rho)$, hence there is an induced $\field$-linear map $h'':\image\gamma\rightarrow\field$ such that $h-h|_{\ker\gamma}\circ\rho=h''\circ\gamma$. Extend $h''$ to a $\field$-linear map $h':M_{\operatorname{in}}\rightarrow\field$. Then $h'$ certainly satisfies $h-h|_{\ker\gamma}\circ\rho=h'\circ\gamma$.
Furthemore, extend $j-h'|_{\ker\alpha}\circ\sigma\circ\pi$ to a $\field$-linear map $\widetilde{j}:M_{\operatorname{in}}\rightarrow\field$ (once more, this can be done regardless of the finiteness or infiniteness of the dimensions), and set $f:=\widetilde{j}+h'$ and $g:=h-f|_{\image\gamma}\circ\gamma$. Then $g\circ\beta = \overline{h}\circ p\circ\beta-f|_{\image\gamma}\circ\gamma\circ\beta=0$, i.e., $g\in\ker(\beta^\vee)$. Furthermore, since $\widetilde{j}|_{\image\gamma}=(j-h'|_{\ker\alpha}\circ\sigma\circ\pi)|_{\image\gamma}=0$, we have
\begin{align*}
    f|_{\image\gamma}\circ\overline{\gamma}+\overline{g|_{\ker\gamma}\circ\rho} &= (\widetilde{j}+h')|_{\image\gamma}\circ\overline{\gamma}+\overline{(h-f|_{\image\gamma}\circ\gamma)|_{\ker\gamma}\circ\rho}\\
    &= h'|_{\image\gamma}\circ\overline{\gamma}+\overline{h|_{\ker\gamma}\circ\rho}\\
    &=h''\circ\overline{\gamma}+\overline{h|_{\ker\gamma}\circ\rho}\\
    &=\overline{h-h|_{\ker\gamma}\circ\rho}+\overline{h|_{\ker\gamma}\circ\rho} =\overline{h},\\
    f|_{\ker\alpha}\circ\sigma&=(\widetilde{j}+h')|_{\ker\alpha}\circ\sigma\\
    &=(j-h'|_{\ker\alpha}\circ\sigma\circ\pi)\circ\sigma+h'|_{\ker\alpha}\circ\sigma\\
    &=\overline{j}\circ\pi\circ\sigma-h'|_{\ker\alpha}\circ\sigma\circ\pi\circ\sigma+h'|_{\ker\alpha}\circ\sigma =\overline{j}.
\end{align*}
This finishes the proof of Lemma \ref{lemma:duality-commutes-with-mutation-linear-map-psi_k}.
\end{proof}

\begin{lemma}\label{lemma:mut-of-decoration-commutes-with-duality}
As $\field$-vector spaces, $\frac{\ker(\alpha^\vee)}{\ker(\alpha^\vee)\cap\image(\beta^\vee)}\cong\left(\frac{\ker\beta}{\ker\beta\cap\image\alpha}\right)^\vee$.
\end{lemma}

\begin{proof} It is well-known that there is a commutative diagram with exact rows
\[
\xymatrix{
0 \ar[r] & \ker\beta\cap\image\alpha \ar[r] \ar[d] & \ker\beta \ar[r] \ar[d] & \frac{\ker\beta}{\ker\beta\cap\image\alpha} \ar[r] \ar[d]^{\cong} & 0\\
0 \ar[r] & \image\alpha \ar[r]  & \image\alpha+\ker\beta \ar[r] & \frac{\image\alpha+\ker\beta}{\image\alpha} \ar[r] & 0
}
\]
whose leftmost square is made up of inclusions, and whose rightmost vertical arrow is a $\field$-vector space isomorphism. Hence $\left(\frac{\ker\beta}{\ker\beta\cap\image\alpha}\right)^\vee\cong \left(\frac{\image\alpha+\ker\beta}{\image\alpha}\right)^\vee$ since $?^\vee=\Hom_{\field}(?,\field)$ is a functor. On the other hand, we have the commutative diagram
\[
\xymatrix{
 & & M_k  \ar[d]^{\pi_1} \ar[r]^{\myid} & M_k \ar[d]^{\pi_2} & \\
0 \ar[r]& \frac{\image\alpha+\ker\beta}{\image\alpha}\ar[r] & \frac{M_k}{\image\alpha} \ar[r] & \frac{M_k}{\image\alpha+\ker\beta} \ar[r] & 0}
\]
whose second row is exact, and whose vertical arrows are the canonical projections. Since $?^\vee=\Hom_{\field}(?,\field)$ is an exact functor, we obtain the commutative diagram 
\[
\xymatrix{
 & & M_k^\vee  & M_k^\vee \ar[l]^{\myid} & \\
0 & \left(\frac{\image\alpha+\ker\beta}{\image\alpha}\right)^\vee \ar[l]  & \left(\frac{M_k}{\image\alpha}\right)^\vee \ar[l] \ar[u]_{\pi_1^\vee} & \left(\frac{M_k}{\image\alpha+\ker\beta}\right)^\vee \ar[l] \ar[u]_{\pi_2^{\vee}} & 0 \ar[l]}.
\]
whose second row is exact, and whose vertical arrows are injective $\field$-linear maps.

Now, $\ker(\alpha^\vee)=\image(\pi_1^\vee)$ and
$\image(\beta^\vee)=\{\varphi\in M_k^\vee\suchthat\ker\beta\subseteq\ker\varphi\}$, hence
\[
\ker(\alpha^\vee)\cap\image(\beta^\vee) =
\{\varphi\in M_k^\vee\suchthat \image\alpha+\ker\beta\subseteq\ker\varphi\}= \image(\pi_2^\vee),
\]
so we obtain a commutative diagram 
\[
\xymatrix{
0 & \frac{\ker(\alpha^\vee)}{\ker(\alpha^\vee)\cap\image(\beta^\vee)} \ar[l]& \ker(\alpha^\vee)  \ar[l] & \ker(\alpha^\vee)\cap\image(\beta^\vee) \ar[l] & 0 \ar[l]\\
0 & \left(\frac{\image\alpha+\ker\beta}{\image\alpha}\right)^\vee \ar[l]  & \left(\frac{M_k}{\image\alpha}\right)^\vee \ar[l] \ar[u]_{\pi_1^\vee}^{\cong} & \left(\frac{M_k}{\image\alpha+\ker\beta}\right)^\vee \ar[l] \ar[u]_{\pi_2^{\vee}}^{\cong} & 0 \ar[l]}.
\]
whose rows are exact, and whose right-most vertical arrows are isomorphisms. So,
\[
\frac{\ker(\alpha^\vee)}{\ker(\alpha^\vee)\cap\image(\beta^\vee)}\cong \left(\frac{\image\alpha+\ker\beta}{\image\alpha}\right)^\vee \cong \left(\frac{\ker\beta}{\ker\beta\cap\image\alpha}\right)^\vee.
\]
\end{proof}

The proof of the next lemma is left in the hands of the reader.

\begin{lemma}\label{lemma:duality-commutes-with-mutation-commutativity-of-crucial-diagram}
The diagram
\[
\xymatrix{
 & &\coker(\alpha^\vee)\oplus\frac{\ker(\beta^\vee)}{\image(\gamma^\vee)}\oplus V^\vee \ar[drr]^{\quad \left[\begin{array}{ccc}\overline{\gamma^\vee} & i\circ s & 0\end{array}\right]} \ar[dd]_{\delta_k} & &\\
 M_{\operatorname{in}}^\vee \ar[urr]|-{\left[\begin{array}{c}-q\\ 0\\ 0\end{array}\right]\quad } \ar[dd]_{\myid} & & & & M_{\operatorname{out}}^\vee \ar[dd]^{\myid} \\
 & & (\coker\beta)^\vee\oplus\left(\frac{\ker\alpha}{\image\gamma}\right)^\vee \oplus V^\vee \ar[drr]_{\left[\begin{array}{ccc}-p^\vee & 0 & 0\end{array}\right]} & &\\
 M_{\operatorname{in}}^\vee \ar[urr]|-{\left[\begin{array}{c}\overline{\gamma}^\vee \\ (\iota\sigma)^\vee \\ 0\end{array}\right]} & &  & & M_{\operatorname{out}}^\vee
}
\]
commutes, where $q:M_{\operatorname{in}}^\vee\rightarrow\coker(\alpha^\vee)$ is the projection, $i:\ker(\beta^\vee)\rightarrow M_{\operatorname{out}}^\vee$ is the inclusion and $s:\ker(\beta^\vee)/\image(\gamma^\vee)\rightarrow\ker(\beta^\vee)$ is any section.
\end{lemma}

\begin{proof}[Proof of Theorem \ref{thm:duality-commutes-with-mutation}]
By Lemmas \ref{lemma:duality-commutes-with-mutation-linear-map-psi_k} and \ref{lemma:duality-commutes-with-mutation-commutativity-of-crucial-diagram}, we have a $\pathalg{\widetilde{\mu}_{k}(Q^{\operatorname{op})}}$-module isomorphism $\delta:\overline{M^{\vee}}\rightarrow\overline{M}^\vee$ given by $\delta_k$ at $k$ and by $\myid_{M_{j}^\vee}$ at $j\in Q_0\setminus\{k\}$.
Conjugation yields a $\field^{Q_0}$-ring isomorphism $\delta^\star:\End_{\field}(\overline{M}^\vee)\rightarrow\End_{\field}(\overline{M^\vee})$, $f\mapsto \delta^{-1}\circ f\circ \delta$. Direct computation shows that the diagram 
{
\[
\begin{tikzcd}        
& \compalg{\widetilde{\mu}_k(Q^{\operatorname{op}})} \ar[d,"\widetilde{\eta}^{\overline{M}^\vee}"]
& 
e_{\widehat{k}}\compalg{\widetilde{\mu}_k(Q^{\operatorname{op}})}e_{\widehat{k}} \ar[r,"\kappa"] \arrow[hook]{l}
&   
\compalg{e_{\widehat{k}}\widetilde{\mu}_k^2(Q^{\operatorname{op}})e_{\widehat{k}}} \ar[d, "\psi"] \\
\pathalg{\widetilde{\mu}_k(Q^{\operatorname{op}})} \arrow[r,"\eta^{\overline{M}^\vee}"'] \arrow[dr,"\eta^{\overline{M^\vee}}"'] \arrow[hook]{ur} 
& \End_{\field}(\overline{M}^\vee) \ar[d,"\delta^{\star}"]
& \End_{\field}(e_{\widehat{k}}(M^\vee)) \ar[l,"\theta^{\overline{M}^\vee}"]
& e_{\widehat{k}}\compalg{Q^{\operatorname{op}}}e_{\widehat{k}} \ar[l,"\widetilde{\rho}^{M^\vee}|"] \ar[dl,"\widetilde{\rho}^{M^\vee}|"] \\
& \End_{\field}(\overline{M^\vee}) 
& \End_{\field}(e_{\widehat{k}}(M^\vee)) \ar[l,"\theta^{\overline{M^\vee}}"]
&  
\end{tikzcd}
\]
commutes.} The uniqueness part of Proposition \ref{prop:overlineM-is-left-compalg-module} then implies $\delta^\star\circ\widetilde{\eta}^{\overline{M}^\vee}=\widetilde{\eta}^{\overline{M^\vee}}$. Thus, for every $u\in\compalg{\widetilde{\mu}_k(Q^{\operatorname{op}})}$ we have
\[
\delta^{-1}\circ \widetilde{\eta}^{\overline{M}^\vee}(u)\circ\delta=\widetilde{\eta}^{\overline{M^\vee}}(u), \quad \text{i.e.}, \quad 
\widetilde{\eta}^{\overline{M}^\vee}(u)\circ\delta=\delta\circ \widetilde{\eta}^{\overline{M^\vee}}(u).
\]
Hence $\delta:\overline{M^{\vee}}\rightarrow\overline{M}^\vee$ is an isomorphism of $\jacobalg{\widetilde{\mu}_k(Q,S)}$-modules. Therefore,
\[
\varphi^{\#}(\delta):=(\overline{\iota}^{\natural}\circ\overline{\varphi}^{\natural})(\delta):\varphi^{\#}(\overline{M^\vee})\rightarrow\varphi^{\#}(\overline{M}^\vee)=\varphi^{\#}(\overline{M})^\vee
\]
is a $\jacobalg{\mu_k(Q,S)}$-module isomorphism and $(\varphi^{\#}(\delta),\varepsilon)$ is an isomorphism of decorated $\jacobalg{\mu_k(Q^{\operatorname{op}},S^{\operatorname{op}})}$-modules $\mu_k(\mathcal{M}^\vee)\rightarrow\mu_k(\mathcal{M})^\vee$, where $\varepsilon:\overline{V^\vee}\rightarrow\overline{V}^\vee$ is the $\field^{Q_0}$-isomorphism given by Lemma \ref{lemma:mut-of-decoration-commutes-with-duality} at $k$, and by $\myid_{V_j^\vee}$ at $j\in Q_0\setminus\{k\}$.
\end{proof}

\section{Isomorphism preservation and involutivity}\label{sec:iso-preservation-and-involutivity}

We work under the assumptions and notations from \S\ref{subsec:Preparation} and Definition \ref{def:mut-of-dec-rep}.

\begin{theorem}\label{thm:iso-modules-DWZmutate-to-iso-modules}
    If $\mathcal{M}=(M,V)$ and $\mathcal{N}=(N,W)$ are isomorphic as decorated $\jacobalg{Q,S}$-modules, then their Derksen-Weyman-Zelevinsky mutations in direction $k$ with respect to the same $\varphi$ are isomorphic as decorated $\jacobalg{\mu_k(Q,S)}$-modules
\end{theorem}

\begin{proof}
    We first show that $\overline{M}$ and $\overline{N}$ are isomorphic as quiver representations of $\widetilde{\mu}_k(Q)$. As in \S\ref{subsubsec:def-of-overlineM-via-a-pushout}, we use the retraction $\rho^N:N_{\operatorname{out}}\rightarrow\ker\gamma^N$ to pick the composition $p^N\mathfrak{s}^N:\image\gamma_k^N\rightarrow \coker\beta_k^N$ which is a section to %the induced map
    $\overline{\gamma_k^N}:\coker\beta_k^N\rightarrow\image\gamma_k^N$. 

    Suppose that $(f,g)$ is a pair consisting of a left $\jacobalg{Q,S}$-module isomorphism $f:M\rightarrow N$ and a left $\field^{Q_0}$-module isomorphism $g:V\rightarrow W$. Define a pair $(\widetilde{f},\widetilde{g})$ of $Q_0$-tuples of $\field$-linear maps $\widetilde{f}=(\widetilde{f}_j)_{j\in Q_0}:\overline{M}\rightarrow \overline{N}$ and $\widetilde{g}=(\widetilde{g}_j)_{j\in Q_0}:\overline{V}\rightarrow \overline{W}$ by setting
$\widetilde{f}_j:=f_j:\overline{M}_j:=M_j\rightarrow N_j=:\overline{N}_j$ and $\widetilde{g}_j:=g_j:\overline{V}_j:=V_j\rightarrow W_j=:\overline{W}_j$ for $j\neq k$, and
\begin{align*}
    \widetilde{f}_k:=
    \left[\begin{array}{ccc}\overline{f_{\operatorname{out}}} & \varepsilon & 0 \\ 0 & \overline{f_{\operatorname{in}}} & 0 \\ 0 & 0 & g_k\end{array}\right]&:\coker\beta_k^M\oplus\frac{\ker\alpha_k^M}{\image\gamma_k^M}\oplus V_k\rightarrow \coker\beta_k^N\oplus\frac{\ker\alpha_k^N}{\image\gamma_k^N}\oplus W_k 
    \\
    \widetilde{g}_k:=\overline{f_k|}&:\frac{\ker\beta^M}{\ker\beta^M\cap\image\alpha^M}\rightarrow \frac{\ker\beta^N}{\ker\beta^N\cap\image\alpha^N},\\
    \text{where} \quad
\varepsilon&:=p^N\mathfrak{s}^N(f_{\operatorname{in}}\iota\sigma^M-\sigma^N\overline{f_{\operatorname{in}}}).
\end{align*}
Direct computation shows that $\widetilde{f}:\overline{M}\rightarrow\overline{N}$ is an isomorphism of quiver representations of $\widetilde{\mu}_k(Q)$ and $\widetilde{g}:\overline{V}\rightarrow\overline{W}$ is an isomorphism of $\field^{Q_0}$-modules.

Conjugation by $\widetilde{f}$ yields a $\field^{Q_0}$-ring isomorphism $\widetilde{f}_{\star}:\End_{\field}(\overline{M})\rightarrow\End_{\field}(\overline{N})$ that makes the diagram {%
\[
\begin{tikzcd}        
& \compalg{\widetilde{\mu}_k(Q)} \ar[d,"\widetilde{\eta}^{\overline{M}}"]
& 
e_{\widehat{k}}\compalg{\widetilde{\mu}_k(Q)}e_{\widehat{k}} \ar[r,"\kappa"] \arrow[hook]{l}
&   
\compalg{e_{\widehat{k}}\widetilde{\mu}_k^2(Q)e_{\widehat{k}}} \ar[d, "\psi"] \\
\pathalg{\widetilde{\mu}_k(Q)} \arrow[r,"\eta^{\overline{M}}"'] \arrow[dr,"\eta^{\overline{N}}"'] \arrow[hook]{ur} 
& \End_{\field}(\overline{M}) \ar[d,"\widetilde{f}_{\star}"]
& \End_{\field}(e_{\widehat{k}}M) \ar[l,"\theta^{\overline{M}}"]
& e_{\widehat{k}}\compalg{Q}e_{\widehat{k}} \ar[l,"\widetilde{\rho}^{M}|"] \ar[dl,"\widetilde{\rho}^{N}|"] \\
& \End_{\field}(\overline{N}) 
& \End_{\field}(e_{\widehat{k}}N) \ar[l,"\theta^{\overline{N}}"]
&  
\end{tikzcd}
\]
commute.} 
The uniqueness part of Proposition \ref{prop:overlineM-is-left-compalg-module} then implies that $\widetilde{f}_{\star}\circ\widetilde{\eta}^{\overline{M}}=\widetilde{\eta}^{\overline{N}}$. Thus, for every $u\in\compalg{\widetilde{\mu}_k(Q)}$ we have $\widetilde{f}\circ\widetilde{\eta}^{\overline{M}}(u)\circ\widetilde{f}^{-1}=\widetilde{\eta}^{\overline{N}}(u)$, i.e., $\widetilde{f}\circ\widetilde{\eta}^{\overline{M}}(u)=\widetilde{\eta}^{\overline{N}}(u)\circ \widetilde{f}$. Hence $\widetilde{f}:\overline{M}\rightarrow\overline{N}$ is an isomorphism of $\jacobalg{\widetilde{\mu}_k(Q,S)}$-modules.
Therefore, $\varphi^{\#}(\widetilde{f}):=(\overline{\iota}^{\natural}\circ\overline{\varphi}^{\natural})(\widetilde{f}):\varphi^{\#}(\overline{M})\rightarrow\varphi^{\#}(\overline{N})$ is an isomorphism of $\jacobalg{\mu_k(Q,S)}$-modules and $(\varphi^{\#}(\widetilde{f}),\widetilde{g})$ is an isomorphism of decorated $\jacobalg{\mu_k(Q,S)}$-modules $\mu_k(\mathcal{M})\cong\mu_j(\mathcal{N})$.
\end{proof}

Suppose that $Q$ is $2$-acyclic.
As pointed out in \cite[(5.19) and (5.20)]{derksen2008quivers}, Definition~\ref{def:premut-and-mut-of-a-QP} immediately implies
\[
\widetilde{\mu}_k^2(S)\sim_{\operatorname{cyc}} [S] + \sum_{\substack{a,b\in Q_1\\ h(a)=k=t(b)}} ([ba]+ba)[a^*b^*].
\]
Furthermore, \cite[Proof of Theorem 5.7]{derksen2008quivers} shows that the QP $(C,T)$ with vertex set $C_0:=Q_0$, arrow set $C_1:=\{[ba],[a^*b^*] \suchthat a,b\in Q_1, h(a)=k=t(b)\}$, and potential
\[
T:=\sum_{\substack{a,b\in Q_1\\ h(a)=k=t(b)}} [ba][a^*b^*],
\]
is trivial and there exists a right-equivalence $h:(Q,S)\oplus(C,T)\rightarrow (\widetilde{\mu}_k^2(Q),\widetilde{\mu}_k^2(S))$
such that whenever $c\in Q_1$:
\[
h(c)=\begin{cases}
c & \text{if $t(c)\neq k$};\\
-c & \text{if $t(c)=k$}.
\end{cases}
\]
We thus have the commutative diagram of $\field^{Q_0}$-ring homomorphisms
\[
\xymatrix{
\compalg{Q} \ar@{^{(}->}[r]^{\jmath} \ar[d] & \compalg{Q\oplus C} \ar[r]^{h} \ar[d]_{\pi} & \compalg{\widetilde{\mu}_k^2(Q)} \ar[d] \\
\jacobalg{Q,S} \ar[r]^-{\overline{\jmath}}_-{\cong} & \jacobalg{Q\oplus C,S+T} \ar[r]^-{\overline{h}}_-{\cong} & \jacobalg{\widetilde{\mu}_k^2(Q),\widetilde{\mu}_k^2(S)}.
}
\]

\begin{prop}\label{prop:mukmukM-iso-toM}
If $Q$ is $2$-acyclic, then
    for every decorated $\jacobalg{Q,S}$-module $\mathcal{M}=(M,V)$ there is an isomorphism of decorated $\jacobalg{Q,S}$-modules
    \[
    (\overline{\jmath}^\natural\circ\overline{h}^\natural)(\widetilde{\mu}_k^2(\mathcal{M}))\cong \mathcal{M}.
    \]
\end{prop}

\begin{proof} Denote $\mathcal{M}=(M,V)$ and $\widetilde \mu_k^2(\mathcal M)=(\overline{\overline{M}}, \overline{\overline{V}})$. Define a $\field$-linear map  $\delta = (\delta_j)_{j\in Q_0} \colon \overline{\overline M} \rightarrow M$ exactly as in the last part of \cite[Proof of Theorem 10.13]{derksen2008quivers} (thus, $\delta_k$ is called $\psi$ therein and $\delta_j = \myid$ for $j\neq k$). It follows exactly in the same way as in the referred proof that $\delta$ is a $\field$-vector space isomorphism in the infinite dimensional case.

Let $\widetilde{\rho}^M:\compalg{Q}\rightarrow\End_{\field}(M)$ be the $\field^{Q_0}$-ring homomorphism giving the left $\compalg{Q}$-module structure on $M$. Then $(\pi^\natural\circ(\overline{\jmath}^{-1})^{\natural})(M)$ is the $\compalg{Q\oplus C}$-module which as a $\field$-vector space is equal to $M$, with the action of $\compalg{Q}$ given precisely by $\widetilde{\rho}^M$, and with the action of the 2-sided ideal generated by $C$ in $\compalg{Q\oplus C}$ being identically zero. Denote by $\widetilde{\zeta}:\compalg{Q\oplus C}\rightarrow \End_{\field}(M)$ the $\field^{Q_0}$-ring homomorphism corresponding to this action. Let also $\widetilde{\eta}^{\overline{\overline{M}}}:\compalg{\widetilde{\mu}_k^2(Q)}\rightarrow \End_{\field}(\overline{\overline{M}})$ be the $\field^{Q_0}$-ring homomorphism obtained by applying Definition \ref{def:premut-of-decorated-rep} and Proposition \ref{prop:overlineM-is-left-compalg-module} to obtain $(\overline{\overline{M}},\overline{\overline{V}})$ from $(\overline{M},\overline{V})$. Furthermore, set $\delta^*(f):=\delta^{-1}\circ f\circ\delta\in\End_{\field}(\overline{\overline{M}})$. for $f\in \End_{\field}(M)$. Direct computation using \eqref{eq:expression-of-element-of-tilde-mu-k(Q)}, \eqref{eq:action-of-arbitrary-u-on-overlineM} and \cite[(10.29)]{derksen2008quivers} shows that the diagram {
\[
\begin{tikzcd}        
& \compalg{\widetilde{\mu}_k^2(Q)} \ar[d,"\widetilde{\zeta} h^{-1}"]
& 
e_{\widehat{k}}\compalg{\widetilde{\mu}_k^2(Q)}e_{\widehat{k}} \ar[r,"\kappa"] \arrow[hook]{l} \ar[ddl,"\widetilde{\eta}^{\overline{\overline{M}}}|"]
&   
\compalg{e_{\widehat{k}}\widetilde{\mu}_k^3(Q)e_{\widehat{k}}} \ar[dd, "\psi"] \\
\pathalg{\widetilde{\mu}_k^2(Q)}  \arrow[dr,"\eta^{\overline{\overline{M}}}"'] \arrow[hook]{ur} 
& \End_{\field}(M) \ar[d,"\delta^{\star}"] & & \\
& \End_{\field}(\overline{\overline{M}}) 
& \End_{\field}(e_{\widehat{k}}M), \ar[l,"\theta^{\overline{\overline{M}}}"]
&  e_{\widehat{k}}\compalg{\widetilde{\mu}_k(Q)}e_{\widehat{k}} \ar[l,"\widetilde{\rho}^{\overline{M}}|"]
\end{tikzcd}
\]
commutes.} By the uniqueness part of Proposition \ref{prop:overlineM-is-left-compalg-module}, this implies $\delta^*\widetilde{\zeta}h^{-1}=\widetilde{\eta}^{\overline{\overline{M}}}$. Hence for $u\in\compalg{\widetilde{\mu}_k^2(Q)}$ we have
\[
\delta^{-1}\circ(\widetilde{\zeta}h^{-1}(u))\circ\delta =\widetilde{\eta}^{\overline{\overline{M}}}(u).
\]
Therefore, for $v\in\compalg{Q}$ we have 
\[
(\widetilde{\rho}^M(v))\circ\delta =\delta\circ\widetilde{\eta}^{\overline{\overline{M}}}(h\jmath(v)),
\]
so $\delta$ is an isomorphism of $\jacobalg{Q,S}$-modules $(\overline{\jmath}^\natural\circ\overline{h}^\natural)(\overline{\overline{M}})\rightarrow M$. The proposition follows.
\end{proof}

The following direct consequence of Proposition \ref{prop:mukmukM-iso-toM} generalizes \cite[Theorem 10.13]{derksen2008quivers} from the finite-dimensional setting to the arbitrary-dimensional one. 

\begin{theorem}
If $(Q,S)$ is a QP with $2$-acyclic underlying quiver $Q$, then every decorated $\jacobalg{Q,S}$-module $\mathcal{M}$ is right-equivalent to $\mu_k^2(\mathcal M)$.
\end{theorem}

Combining Theorem \ref{thm:iso-modules-DWZmutate-to-iso-modules} with Proposition \ref{prop:mukmukM-iso-toM} we obtain the following bijective correspondence of isoclasses, unfortunately not stated explicitly  in \cite{derksen2008quivers} in the finite-dimensional case.

\begin{theorem}
Suppose $(Q,S)$ is a QP with $2$-acyclic underlying quiver $Q$, fix a right-equivalence $\varphi:\mu_{k}(Q,S)\oplus(C,T)\rightarrow \widetilde{\mu}_k(Q,S)$, where $(C,T)$ is a trivial QP, and pick decorated $\jacobalg{Q,S}$-modules $\mathcal{M}=(M,V)$ and $\mathcal{N}=(N,W)$.
\begin{enumerate}
    \item The DWZ-mutations with respect to $\varphi$, $\mu_k(\mathcal{M})$ and $\mu_k(\mathcal{N})$, are isomorphic as decorated $\jacobalg{\mu_k(Q,S)}$-modules, if and only if the premutations $\widetilde{\mu}_k(\mathcal{M})$ and $\widetilde{\mu}_k(\mathcal{N})$ are isomorphic as  decorated $\jacobalg{\widetilde{\mu}_k(Q,S)}$-modules;
    \item the premutations $\widetilde{\mu}_k(\mathcal{M})$ and $\widetilde{\mu}_k(\mathcal{N})$ are isomorphic as  decorated $\jacobalg{\widetilde{\mu}_k(Q,S)}$-modules if and only if $\mathcal{M}$ and $\mathcal{N}$ are isomorphic as $\jacobalg{Q,S}$-modules.
\end{enumerate}
\end{theorem}

\section*{Acknowledgements}

Thanks most of all to Bea de Laporte for all the inspiring discussions, it is thanks to them and the collaboration \cite{labardini2022landau} that this paper exists.
I am grateful as well to Maximilian Kaipel, Rosie Laking, Lang Mou, Markus Reineke, Håvard Terland and Jerzy Weyman for many insightful conversations and explanations. 

The author received financial support of \emph{Fondazione Cariparo} through a ``Starting Package per attrarre ricercatrici e ricercatori eccellenti dall'estero'' (CUP: C93C22008360007) granted to him at Università degli Studi di Padova.

\bibliographystyle{abbrv}
\bibliography{reference.bib}

\end{document}